\documentclass [a4paper, 12pt]{article}
\usepackage[english]{babel}
\usepackage{color}
\usepackage{subcaption}
\usepackage{babel}
\usepackage{fontenc}
\usepackage{graphicx}
\usepackage{epstopdf}
\usepackage{amsmath,amsthm,amssymb}
\usepackage[colorlinks=true,citecolor=blue, linkcolor=blue, urlcolor=blue]{hyperref}
\usepackage{float}
\usepackage{lineno}
\newtheorem{theorem}{Theorem}

\newtheorem{lemma}{Lemma}

\usepackage[percent]{overpic}
\usepackage{multicol}
\usepackage{pgfplots}
\usepackage{soul}
\pgfplotsset{compat=1.12}
\newcommand{\di}{\displaystyle}

\usepackage[mathscr]{euscript}
\DeclareSymbolFont{rsfs}{U}{rsfs}{m}{n}
\DeclareSymbolFontAlphabet{\mathscrsfs}{rsfs}
\usepackage[title]{appendix}
\newtheorem{prop}{Proposition}
\usepackage{xcolor}
\definecolor{mg}{RGB}{255,0,255}
\theoremstyle{remark}
\newtheorem{remark}{Remark}
\usepackage{color,soul}
\usepackage[normalem]{ulem}
\usepackage{makecell}

\title{{Spatio-temporal dynamics in  a diffusive Bazykin model: effects of group defense and prey-taxis}}
\author{ Subrata Dey, Malay Banerjee, S. Ghorai\thanks{Corresponding author}\\[1em]
 Department of Mathematics and Statistics,
IIT Kanpur,
 India}
\date{}

\begin{document}
\maketitle
\vspace{-2em}
{\center
E-mails:  subratad@iitk.ac.in, malayb@iitk.ac.in, sghorai@iitk.ac.in\\[2mm]  }
\begin{abstract}

Mathematical modeling and analysis of spatial-temporal population distributions of interacting species have gained significant attention in biology and ecology in recent times. In this work, we investigate a Bazykin-type prey-predator model with a non-monotonic functional response to account for the group defense among the prey population. Various local and global bifurcations are identified in the temporal model. Depending on the parameter values and initial conditions, the temporal model can exhibit long stationary or oscillatory transient states due to the presence of a local saddle-node bifurcation or  a global saddle-node bifurcation of limit cycles, respectively. We further incorporate the movement of the populations consisting of a  diffusive flux modelling random motion and an advective flux modelling group defense-induced prey-taxis of the predator population. The global existence and boundedness of the spatio-temporal solutions are established using $L^p$-$L^q$ estimate.   We also demonstrate the existence of a non-homogeneous stationary solution near the Turing thresholds using weakly nonlinear analysis. A few interesting phenomena, which include extinction inside the Turing region, long stationary transient state, and non-homogeneous oscillatory solutions inside the Hopf region, are also identified. 
\end{abstract}



\section{Introduction}
Understanding complex dynamics of predator-prey interactions through mathematical modelling is a key topic in evolutionary biology and ecology. Since the seminal work of Lotka \cite{LotkaUNDAMPEDOD} and Volterra \cite{Volterra}, mathematical models have been developed to explore a diverse range of complex ecological phenomena that include Allee effects \cite{stephens1999allee}, group defense \cite{venturino2013spatiotemporal},  hunting cooperation \cite{alves2017hunting}, intra-guild predation \cite{kang2013dynamics}, any many others\cite{murray2002mathematical}. The incorporation of the movement of the population makes these models more realistic in the field of spatial ecology \cite{murray2001mathematical}. The movement of a population can be due to various reasons. The main reasons include the random motion of individuals and the directed motion of individuals in response to the gradient of some signals, among others. The first one is modelled using a diffusive flux and the latter one uses an advective flux in which the velocity field depends on the gradient of the signal.



Let $N(T)$ and $P(T)$ respectively be  the densities of the prey and predator populations at time $T.$ The Bazykin type prey-predator model \cite{bazykin1976structural,arancibia2021bifurcation,avila2017bifurcation}, which incorporates intra-species 
 competition among predators, is given by:
\begin{subequations}
\begin{alignat}{4}
\frac{dN}{dT}&=NG(N)-F(N)P,\\
\frac{dP}{dT}&=\zeta F(N)P-\gamma P-\delta P^2,
\end{alignat}
\label{1}
\end{subequations}
where $G(N)$ denotes the per capita growth rate of the prey species, $F(N)$ is the prey-dependent functional response and $\zeta$ is the conversion coefficient. Further, $\gamma$ and $\delta$ respectively denote the natural mortality rate and intra-species competition of the predator population.



The prey-predator interaction  is completely reliant on the functional response function $F(N)$. Therefore, the selection of an appropriate functional response is critical in determining the outcome of the predator-prey dynamics. The literature commonly favors strictly monotonic functional responses that increase with prey density with a finite upper bound, such as $\di{\frac{\alpha N}{1+\beta N}}$ (Holling type-II) \cite{dey2022analytical}, $\di{\frac{\alpha N^2}{1+\beta N^2}}$ (Holling type-III) \cite{li2008traveling}, $\di{\alpha(1-e^{-\beta N})}$ (Ivelev) \cite{wang2010pattern}, $\di{\frac{\alpha N^2}{1+\mu N+ \beta N^2}}$ (Sigmoidal) \cite{huang2014bifurcations},
 and $\di{\alpha\;\text{tanh}(\beta N)}$ (hyperbolic tangent) \cite{seo2018sensitivity}. However, experimental and observational evidence suggest that the assumption of  monotonicity in the functional response  is not always valid. For example, the experiment by Andrews \cite{andrews1968mathematical} in  microbial dynamics suggests that  higher nutrient concentrations can result in an inhibitory effect on microorganisms, which can be modelled using Monod–Haldane function $\di{F(N)=\frac{ \alpha N}{1+ \mu N+\beta N^2}}.$ 
Sokol and Howell \cite{sokol1981kinetics}  conducted a study on the uptake of phenol by Pseudomonas putida in continuous culture, where they found that a simplified form of the Monod-Haldane function, $\di{F(N)=\frac{ \alpha N}{1+ \beta N^2}}$, provided a better fit to their experimental data. The main  characterization of the non-monotonic functional response is that there exists a $N_m,$ such that \begin{equation}
F'(N) = \begin{cases}
    \geq 0\quad\text{for } 0\leq N\leq N_m,\\
    < 0\;\quad\text{for } N> N_m,
\end{cases} \text{ and } F(N) \rightarrow 0 \text{ as }N \rightarrow \infty.
\end{equation}

Group defense refers to the ability of the prey species to defend or conceal themselves more effectively, leading to a reduction or complete prevention of predation when their population size is sufficiently large \cite{xiao2001global,broer2007dynamics}. Tener \cite{tener1965muskoxen} provided a clear example of this phenomenon in which lone musk oxen are vulnerable to wolf attacks, while small herds of two to six oxen are occasionally attacked, but larger herds are never successfully attacked. This same pattern was also observed by Holmes and Bethel \cite{holmes1972modification} in their study of insect populations, where large swarms of insects made it difficult for predators to identify individual prey. Clearly, group defense can be incorporated into a prey-predator system using a non-monotonic functional response.

Here, we consider a Bazykin prey-predator model with a logistic growth rate for the prey species and a simplified form of the Monod-Haldane function, $\di{F(N)=\frac{ \alpha N}{1+ \beta N^2}}$, for the functional response. The governing equations are
\begin{subequations}
\begin{alignat}{4}
 \frac{dN}{dT}&=N(\sigma-\eta N)-\frac{\alpha NP}{1+\beta N^2},\\
 \frac{dP}{dT}&=\frac{\zeta\alpha NP}{1+\beta N^2}-\gamma P-\delta P^2,
\end{alignat}
\label{2}
\end{subequations}
where $\sigma$ and $\eta$ respectively represent the intrinsic growth rate and intra-species competition of the prey species.


Self-organized spatio-temporal pattern formation is a fundamental process that plays a critical role  in comprehending diverse  intricate  ecological  phenomena  in nature. The pioneering work of Turing on chemical morphogenesis \cite{Turing}  has played a pivotal role in advancing our understanding of pattern formation using the reaction-diffusion (RD) theory.
 RD models have since been extensively studied to explain the formation of patterns in various complex biological  systems such as 
 fish skin \cite{fish}, mussel bed   distribution \cite{cangelosi2015nonlinear}, 
 insect wings \cite{blagodatski2015diverse}, predator-prey interactions \cite{prey}, terrestrial vegetation \cite{vegetation}, and many others \cite{murray2001mathematical}. Spatial Turing and temporal Hopf instabilities are key mechanisms in the development of spatio-temporal patterns. The diffusion-driven Turing instability takes place when a small amplitude spatial perturbation about a stable homogeneous steady-state becomes unstable, resulting in stationary patterns like stripes,  spots, or a mixture of both \cite{renji2023,turing1}. In addition to stationary patterns, various dynamic patterns that include traveling waves, periodic traveling waves, target patterns,   spiral patterns,  quasi-periodic spatial patterns, and even spatio-temporal chaotic patterns are typically found due to the Hopf instability and Turing-Hopf instability \cite{turing3,sherratt1997oscillations}.

The diffusion mechanism  corresponds to the random movement of species from a higher concentration area to a lower concentration area. On the other hand, taxis is a directional movement of species in response to a particular stimulus, such as phototaxis due to light \cite{foster1980light,ghorai2005penetrative} and chemotaxis due to chemical gradients \cite{hillen2009user,sleeman2005existence}. Taxis allows species to move towards or away from their stimulus and it is critical in many ecological processes, including foraging, mating, migration, photosynthesis, and dispersal \cite{foster1980light,hillen2009user,weng2008migration}.There are two types of taxis in a prey-predator system depending on the movement of species \cite{wang2021pattern}. The directional movement of predator  species in response to the prey density is called prey-taxis and the opposite is called predator-taxis. These types of taxis-based movements play a significant role in  the formation of complex spatial patterns \cite{wang2021pattern,ainseba2008reaction,sapoukhina2003role,myerscough1998pattern}. A RD model with Rosenzweig–MacArthur kinetics is unable to form Turing structure \cite{turing1}, whereas the same system  in the presence of taxis can show stationary Turing patterns \cite{sapoukhina2003role,zhang2019pattern}.  We extend the temporal model \eqref{2} to include random movement due to diffusion and 
directed movement of the predator population because of prey-taxis. Thus, our spatio-temporal model becomes
\begin{subequations}
\begin{alignat}{4}
\frac{\partial N}{\partial T}&=  D_1\nabla^2 N+N(\sigma-\eta N)-\frac{\alpha NP}{1+\beta N^2},\; X\in \tilde\Omega,\;T>0,\\
\frac{\partial P}{\partial T}&= D_2\nabla^2  P+ \nabla\cdot ( \chi(N) P \nabla N)+ \frac{\zeta\alpha NP}{1+\beta N^2}-\gamma P-\delta P^2,\;X\in \tilde\Omega,\;T>0,
\end{alignat}
\label{pde2}
\end{subequations}
where  $\tilde{\Omega}\subset \mathbb{R}^n$ is a bounded  domain with boundary $\partial \tilde{\Omega}.$  Further, $D_1$ and $D_2$ respectively are the self-diffusion coefficients of the prey and predator species, and $\chi$ is the prey-taxis coefficient.  Note that $\chi<0$ and $\chi>0$ corresponds to attractant and repellent prey-taxis respectively \cite{gambino2018cross,luo2022global}. Due to group defense by the prey species, the predator species prefer low-density prey areas for their  hunting and avoid high-density prey areas. Our temporal model incorporates the group defense in prey species through a non-monotonic functional response. To take into account the group defense induced prey-taxis  in the RD model, we  take $\chi(N)=\chi_0>0,$  where $\chi_0$ is a constant for simplicity \cite{hillen2009user,luo2022global,maini1991bifurcating}.  A dimensionless version of (\ref{pde2}), using  $u=\frac{\eta }{\sigma}N$, $ v=\frac{\delta }{\sigma}P$, $ t=\sigma T$ and $x= \sqrt{\frac{\sigma}{D_1}}X$  for dimensionless prey, predator, time and space, is   
\begin{subequations}
\begin{alignat}{4}
\frac{\partial u}{\partial t}&=  \nabla^2 u+u(1-u) -\frac{auv}{1+bu^2},\; x\in \Omega,\;t>0\\
\frac{\partial v}{\partial t}&= d\nabla^2  v+c \nabla\cdot(v \nabla u)+ \frac{eauv}{1+bu^2}-f v- v^2,\;x\in \Omega,\;t>0,
\end{alignat}
\label{pde}
\end{subequations}
where $ \di{a=\frac{\alpha }{\delta }}$,\; $\di{b=\frac{\beta \sigma^2}{\eta^2}},$ $\di{c=\frac{\chi\sigma}{D_1\eta},}$ $\di{d=\frac{D_2}{D_1},}$
$ \di{e=\frac{\zeta \delta}{\eta}}$   and $\di{f=\frac{\gamma}{\sigma}},$  are dimensionless positive parameters. Here, $\Omega$ is the dimensionless domain with corresponding boundary $\partial\Omega,$ and $d$ is the ratio of self-diffusion of predator and prey.  Further, $c>0$ represents the dimensionless prey-taxis coefficient  that characterizes  the tendency of the predator population to keep away from the high-density prey areas. The system \eqref{pde} is subjected to non-negative initial conditions $u(x,0)\equiv u_0(x)$ and $v(x,0)\equiv v_0(x)$ for $x\in \Omega,$ and no-flux boundary conditions $\di{\frac{\partial u}{\partial n}=\frac{\partial v}{\partial n}=0}$ for $x\in \partial\Omega$ and $t>0.$


Long transient dynamics is currently an  important topic  for predicting and managing ecological systems in the face of environmental change \cite{morozov2020long}. It refers to the  slowly varying dynamics over a long  period of time before the system reaches its  final state.  This long transient dynamics can be complex and unpredictable, which may appear as  stationary, oscillatory or even chaotic \cite{morozov2020long,morozov2016long,lai2011transient}. In the context of dynamical system, these transient dynamics are guided by various temporal and spatio-temporal bifurcations. The length of this transient period depends on the distance of the control parameter from the bifurcation threshold and initial conditions \cite{morozov2020long,chowdhury2023attractors}. Here, we show  interesting long transient dynamics for the temporal and spatio-temporal system, which are unpredictable in their final state.

The predator species do not have any cooperative nature due to the Bazykin-type reaction kinetics in contrast to prey species for which group defense is a kind of cooperative behavior \cite{bazykin1976structural,arancibia2021bifurcation,avila2017bifurcation}. To the best of our knowledge, this is the first study to investigate group defense-induced prey-taxis in an ecological model. For the temporal model, we perform a  bifurcation analysis and illustrate some representative dynamics through one- and two-parametric  bifurcation diagrams.  We establish global existence and  boundedness of solutions for the spatio-temporal model with Neumann boundary conditions. The stability of the homogeneous steady states  and Turing instability are discussed for the spatio-temporal model. The existence of a non-homogeneous stationary solution is shown using weakly nonlinear analysis (WNA) which is then validated with numerical solution. A key strength of our work is the investigation of long transient dynamics.  We have  established the roles of local and global bifurcations on the  stationary and oscillatory transient dynamics exhibited by the temporal model. The fate of these transient dynamics under the influence of diffusion and taxis have also been investigated. For certain parameter values, appearance of steady homogeneous solution in Turing domain and non-homogeneous in space but oscillatory in time solution in Hopf domain are some other key findings of our work.   

The temporal model together with its equilibria, stability, local and global bifurcations are described in Section \ref{temporalmodel}. This section also contains long transient dynamics with numerical visualization. Section \ref{spatialmodel} consists of the global existence and boundedness of the spatio-temporal solution, stability analysis of the homogeneous steady states, and  Turing bifurcation. The derivation of the amplitude equation using WNA has been carried out in Section \ref{wna}. Using numerical simulations, we validate the results of WNA in Section \ref{numerical}. This section also contains an extensive range of numerical simulations that show long transient dynamics as well as various interesting stationary and dynamic patterns. Finally, the paper concludes with a comprehensive discussion in Section \ref{discussion}.

\section{Temporal model}\label{temporalmodel}
Here, we discuss temporal dynamics corresponding to the homogeneous system of \eqref{pde}, i.e.,
\begin{subequations}
\begin{alignat}{4}
 \frac{du}{dt}&=u(1-u) -\frac{auv}{1+bu^2} &\equiv F_1(u,v)&\equiv uf_1(u,v), \\
 \frac{dv}{dt}&=\frac{eauv}{1+bu^2}-f v- v^2&\equiv F_2(u,v)&\equiv vf_2(u,v).
\end{alignat}
\label{ode}
\end{subequations}
The system \eqref{ode} is subjected to non-negative  initial conditions  $u(0)\geq0,$ $v(0)\geq0.$
\subsection{ Positivity and boundedness }
\begin{theorem}
Every solution $(u(t),v(t))$ of the system (\ref{ode}) with non-negative initial condition remains bounded and  non-negative for all time $t$.
\end{theorem}
\begin{proof}
The proof is straightforward and has been omitted.
\end{proof}

\subsection{Existence and stability of equilibria}
Here, we discuss the number and types of all possible  equilibria in $\mathbb{R}_2^+=\{(u,v):u\geq0,v\geq0\}$. 
The system (\ref{ode}) has trivial equilibrium point $E_0(0,0)$ and axial equilibrium point $E_1(1,0)$ irrespective of parameter values. 
\begin{figure}[!ht]
\begin{subfigure}[b]{.46\textwidth}
 \centering
\includegraphics[scale=0.5]{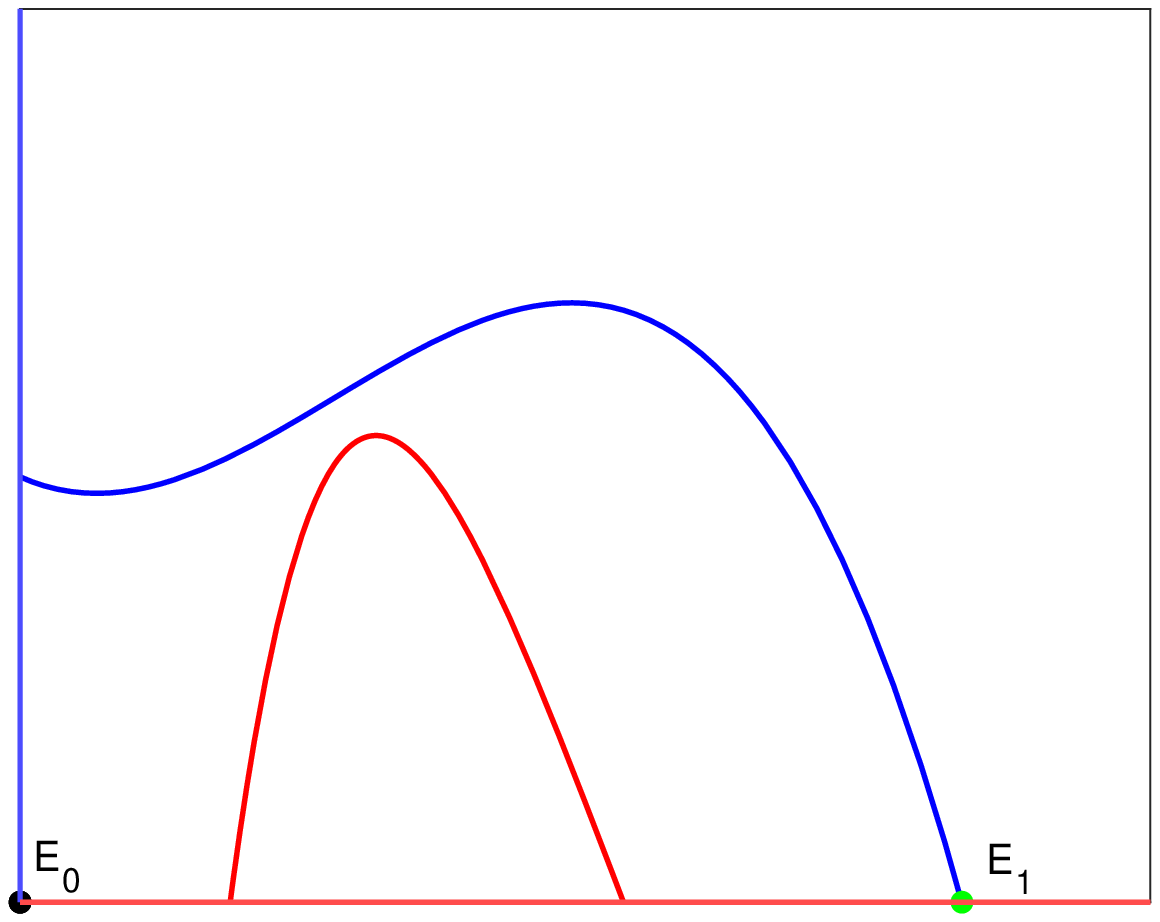}\\ 
 \caption{}
  \end{subfigure}
 \begin{subfigure}[b]{.46\textwidth}
 \centering
 \hspace*{-1.3cm}
\includegraphics[scale=0.5]{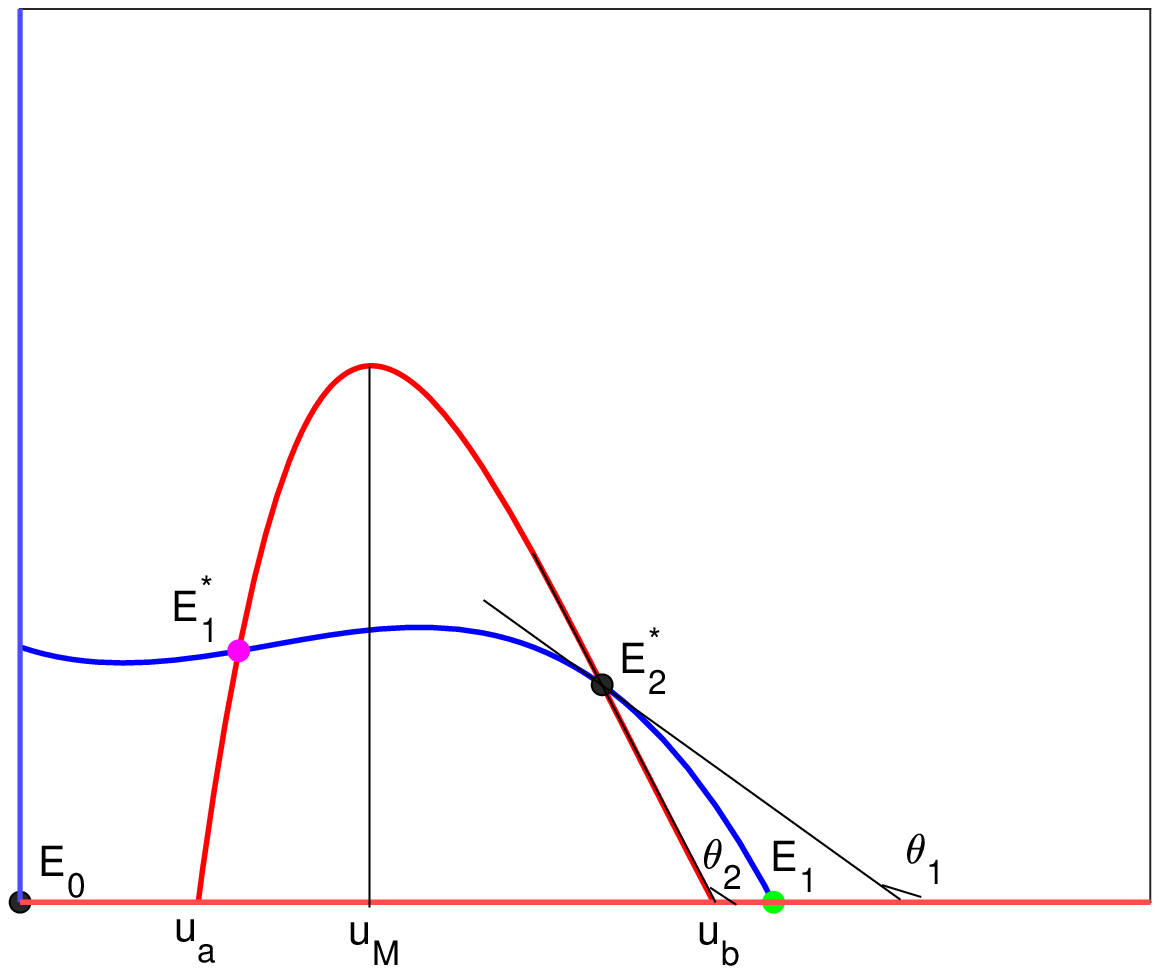}\\ 
 \caption{}
  \end{subfigure}
  \begin{subfigure}[b]{.46\textwidth}
 \centering
\includegraphics[scale=0.5]{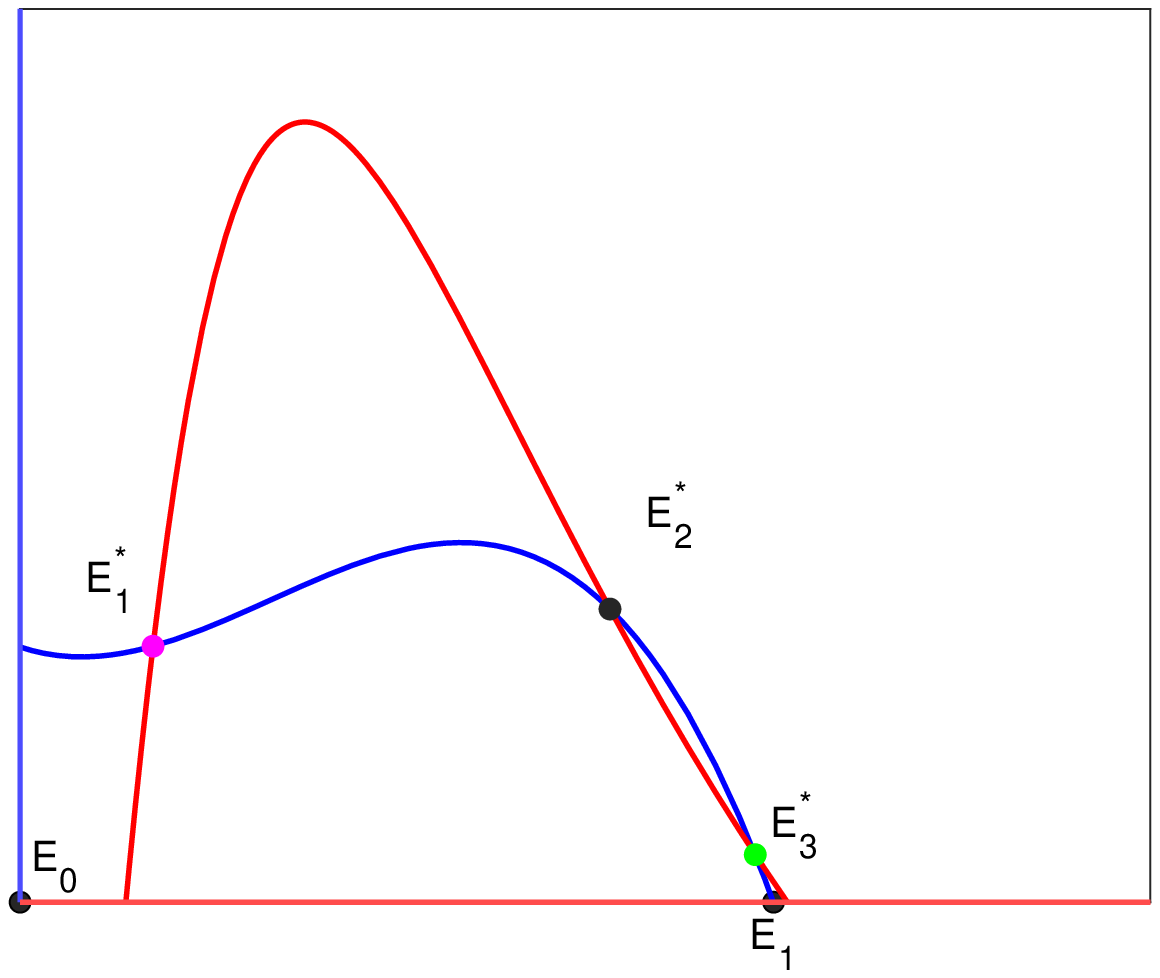}\\ 
 \caption{}
  \end{subfigure} 
  \hspace{1.5em}
 \begin{subfigure}[b]{.46\textwidth}
 \centering
\includegraphics[scale=0.5]{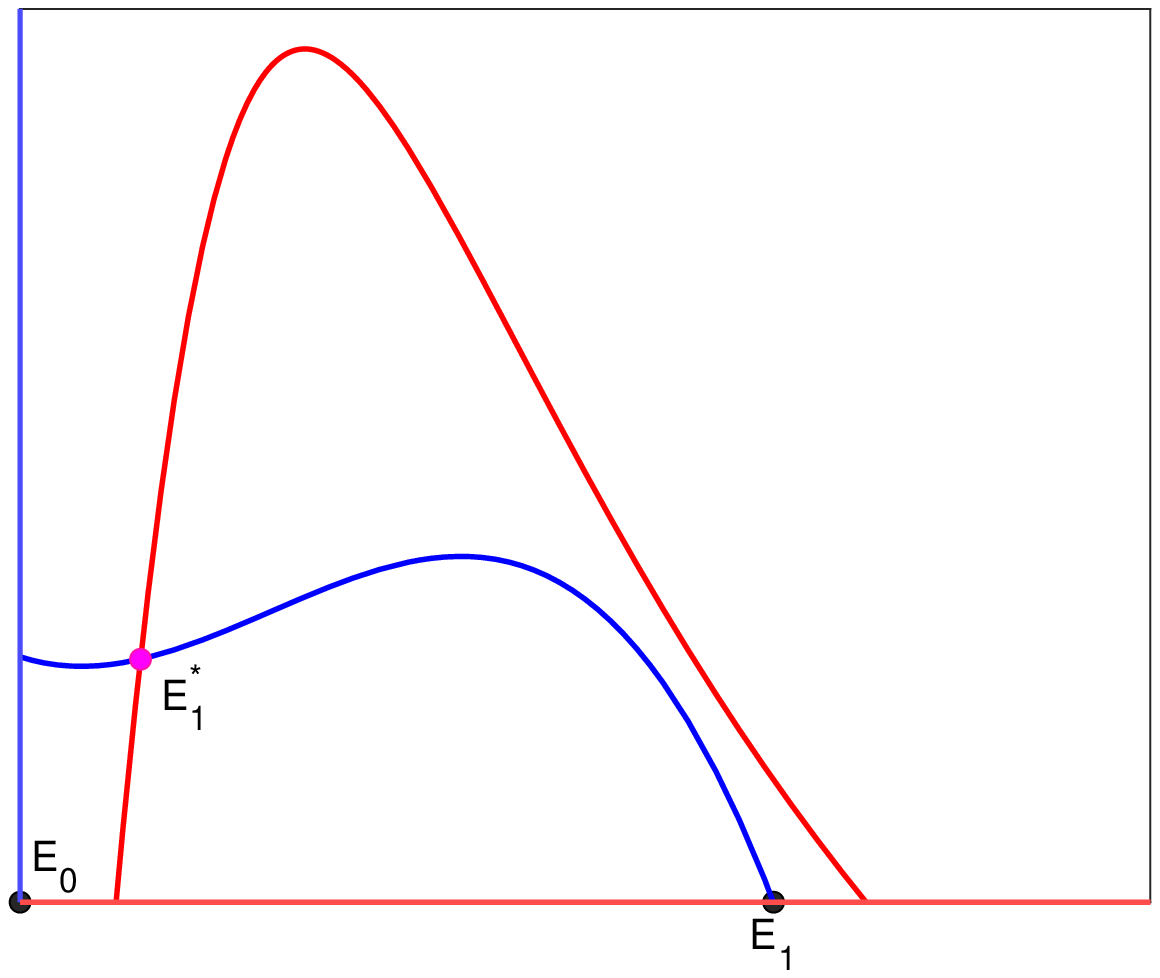}\\ 
 \caption{}
  \end{subfigure}
 \caption{Location of nullclines and equilibria of the system \eqref{ode} for all possible cases. Here  blue and red color curves denote the prey and predator nullclines respectively and dots represent various equilibria. Green and black dots represent the stable and saddle equilibria respectively and the stability of magenta dots depends on the Hopf bifurcation. } 
\label{nullcline}
\end{figure}

An interior equilibrium  $E_j^*(u_j^*,v_j^*)$ (where $j$ can be $1$,$2$ or $3$) is a point of intersection of the nontrivial prey  nullcline $\di{v=n(u):=\frac{(1+bu^2)(1-u)}{a}}$ and predator nullcline $\di{v=p(u):=\frac {eau}{bu^2+1}-f}.$  The  prey component $u_j^*$ of the interior equilibrium  $E_j^*$ satisfies 
\begin{equation}
Q(u)\equiv {b}^{2}{u}^{5}-{b}^{2}{u}^{4}+2 b{u}^{3}-b \left( af+2 \right) {u}^{2 }+ \left( {a}^{2}e+1 \right) u-(af+1) =0, \label{eqQ}
\end{equation}
and the component $v_j^*$ is obtained from
$$
v_j^*={\frac {eau_j^*}{b{u_j^*}^{2}+1}}-f.$$
Now, $n(u)\geq 0$ for $u\leq 1$ and $p(u)$ intersects  $u$ axis at two points, say $(u_a,0)$ and $(u_b,0),$ where 
$$\di{u_a={\frac {ea-\sqrt {{e}^{2}{a}^{2}-4 b{f}^{2}}}{2bf}}\;\mbox{ and }\; u_b={\frac {ea+\sqrt {{e}^{2}{a}^{2}-4 b{f}^{2}}}{2bf}}, }\;\; \text{ whenever } {e}^{2}{a}^{2}\geq 4 b{f}^{2}.
$$ 
Thus,  we must have $u_a<u^*<\text{min}\{1, u_b\}$ for feasibility of $E^*$. Also, $p(u)$ has a maximum at  $\di{u_M=\frac{1}{\sqrt{b}}}$.  The system \eqref{ode} has at least one coexisting equilibrium point when $p(u_M)\geq n(u_M).$  In the case of two equilibria, we label them as $E_1^*(u_1^*,v_1^*)$ and $E_2^*(u_2^*,v_2^*)$ with $\di{0<u_1^*<\frac{1}{\sqrt{b}}<u_2^*<1}$ [see Fig. \ref{nullcline}(b)].  The system can have three equilibria for parameter value $f_{SN_2}<f<f_{TC}$ with $\di{0<u_1^*<\frac{1}{\sqrt{b}}<u_2^*<u_{sn_1}<u_3^*<1},$ where $u_{sn_1}$ and $f_{SN_2}$ are discussed in the next subsection.  Depending on the parameter restriction, we summarize the number of equilibria in  Table \ref{Tab1}. Next, we discuss the stability of the different equilibria. The Jacobian of the system (\ref{ode}) at a point $E(u,v)$ is given by 
\begin{equation}
J(E)=\begin{bmatrix}
 1-2u+  {\frac {av \left( b{u}^{2}-1 \right) }{ \left( b{u}^{2}+1 \right) ^{2}
}} & -{\frac {au}{b{u}^{2}+1}}   \vspace{0.06in} \\
  {\frac {eav \left( 1-b{u}^{2} \right) }{ \left( b{u}^{2}+1 \right) ^{
2}}}
 &   {\frac {eau}{b{u}^{2}+1}}-f-2 v
\end{bmatrix}\equiv \begin{bmatrix}a_{10}& a_{01}\\b_{10} & b_{01} \end{bmatrix}. 
\label{jaco}
\end{equation}
Considering the  Jacobian matrix \eqref{jaco} at an equilibrium point, we have  the following propositions:
\begin{prop}
    The trivial equilibrium point $E_0$ is always a saddle point and the axial equilibrium point $E_1$ is a saddle point for $\di{f<f_{TC}:=\frac{ea}{b+1}}$ and  asymptotically stable for $\di{f>f_{TC}}$.
\end{prop}
\begin{proof}
    Since, the eigenvalues of $J(E_0)$ are 1 and $-f$, it is always a saddle point. The eigenvalues of $J(E_1)$ are  $-1$ and $\di{{\frac {ea}{b+1}}-f}.$ Therefore, $E_1$ is asymptotically stable when $\di{f>f_{TC}}$ and a saddle point when $f<f_{TC}.$ 
\end{proof}
\begin{table}[!t]
    \caption{Existence and stability of different equilibria of the system \eqref{ode}.}\vspace{-0.3cm}
  \begin{center}
    \begin{tabular}{|c|l|l|} 
    \hline
      \makecell[c]{\textbf{Equilibria}} & \textbf{Existence criteria } & \textbf{Stability criteria} \\ 
      \hline
      $E_0$& independent of parameter values & saddle point.\\\hline
      $E_1$ & independent of parameter values &  \makecell[l]{ a saddle point if $\di{f<f_{TC}}$,\\ and  asymptotically stable if $\di{f>f_{TC}}.$ } \\\hline
      $E_1^*$ &  $p(u_M)\geq n(u_M)$ and ${e}^{2}{a}^{2}\geq 4 b{f}^{2}$  & \makecell[l]{asymptotically stable if $\di{b<b_{H}}$, \\ and unstable  if $\di{b>b_{H}}$.}  \\\hline
      $E_2^*$ & \makecell[l]{$p(u_M)> n(u_M)$, ${e}^{2}{a}^{2}\geq 4 b{f}^{2}$ with \\either  $u_b<1$ or $u_b>1$ and $f_{SN_2}<f<f_{TC}$ } & saddle point.  \\\hline
       $E_3^*$ & \makecell[l]{$p(u_M)> n(u_M)$, ${e}^{2}{a}^{2}\geq 4 b{f}^{2}$, $u_b>1$ \\and $f_{SN_2}<f<f_{TC}$  } & asymptotically stable.\\\hline
          \end{tabular}\label{Tab1} \vspace{-1em}
  \end{center} 
\end{table}
\begin{prop} \label{prop2}
    For $p(u_M)\geq n(u_M)$ and ${e}^{2}{a}^{2}\geq 4 b{f}^{2}$, the system \eqref{ode} has at least one coexisting equilibrium and at most three different coexisting  equilibria. The following  hold for the stability of the co-existing equilibria:
    \begin{itemize}
        \item[(i)] $E_1^*$ is asymptotically stable for $b<b_H$ and unstable for $b>b_H$, where $b_H$ is defined in the text. 
        \item[(ii)] Whenever $E_2^*$ exists, it is always a saddle point.
        \item[(iii)] Whenever $E_3^*$ exists, it is always asymptotically stable.
    \end{itemize}
\end{prop}
\begin{proof}
The Jacobian $J(E^*),$ at a coexisting  equilibrium point $E^*(u^*, v^*)$,  is given by
 \begin{equation*}
     J(E^*)=\begin{bmatrix}
u \frac{\partial f_1}{\partial u}  & u \frac{\partial f_1}{\partial v}  \vspace{0.06in}\\ 
v \frac{\partial f_2}{\partial u}  & v \frac{\partial f_2}{\partial v}
\end{bmatrix}_{(u^*, v^*)}=\begin{bmatrix}
-u \frac{\partial f_1}{\partial v} \frac{dv}{du}^{(f_1)}  & u \frac{\partial f_1}{\partial v}  \vspace{0.06in}\\ 
-v \frac{\partial f_2}{\partial v} \frac{dv}{du}^{(f_2)} & v \frac{\partial f_2}{\partial v}
\end{bmatrix}_{(u^*, v^*)},
 \end{equation*}
 where $\frac{dv}{du}^{(f_j)}$  represents the  slope of the tangent to the curve $f_j(u,v)=0$ ($j=1,2$). Also,
 \begin{equation}
     \text{det} (J(E^*))=\left( uv \frac{\partial f_1}{\partial v} \frac{\partial f_2}{\partial v} \left(\frac{dv}{du}^{(f_2)} -\frac{dv}{du}^{(f_1)} \right)\right)_{(u^*, v^*)}.
     \label{det}
 \end{equation}
  We observe that
  \begin{subequations}
\begin{eqnarray}
   \frac{\partial f_1 (u^*, v^*)}{\partial v}=-\frac{a}{1+bu^2}<0\quad\mbox{and}\quad
   \frac{\partial f_2 (u^*, v^*)}{\partial v}=-1.\nonumber 
 \end{eqnarray}
\end{subequations}
Suppose that $\theta_1$ and $\theta_2$ respectively denote the  inclination angles of  the tangents to $f_1(u,v)=0$ and $f_2(u,v)=0$ at $E_2^*$ [see Fig. \ref{nullcline}(b)]. We find that $\frac{\pi}{2}<\theta_2<\theta_1<\pi$ holds whenever $E_2^*$ exists, which implies 
$$\left.\frac{dv}{du}^{(f_2)}\right|_{(u_2^*, v_2^*)}<\left.\frac{dv}{du}^{(f_1)}\right|_{(u_2^*, v_2^*)}.$$
Hence,  we have  $\text{det} (J(E_2^*))<0$ from (\ref{det}) and therefore $E_2^*$ is a saddle point. Similarly, we obtain $\text{det} (J(E_1^*))>0$ for $E_1^*$. Using Routh-Hurwitz stability criteria, the  coexisting equilibrium point $E_1^*$ is asymptotically stable if $\text{tr} (J(E_1^*))<0,$ which holds when $$\di{b<b_H:={\frac {u_1^*+v_1^*}{{u_1^*}^{2} \left( 2-3u_1^*-v_1^* \right) }}}.$$ 
Thus, $E_1^*$ is asymptotically stable for $b<b_H$ and unstable for $b>b_H.$

The coexisting equilibrium point $E_3^*$ exchanges stability with stable $E_1$  through a transcritical bifurcation at $f=f_{TC}$ (discussed in the next subsection) and it is feasible for $f<f_{TC}$. Therefore, it is asymptotically stable  whenever it exists.  
\end{proof}

\subsection{Local bifurcation Analysis}
Here, we discuss transcritical,  saddle-node, Hopf, generalized Hopf (GH), and Bautin bifurcations exhibited by the system \eqref{ode}.
\subsubsection{Transcritical bifurcation}
 \begin{prop}
  The temporal model (\ref{ode}) encounters a transcritical  bifurcation at $E_1$ when the parameter $f$ satisfies the threshold $f=f_{_{TC}}=\di{\frac {ea}{b+1}}$.
 \end{prop}
 \begin{proof}
 The Jacobian matrix $J(E_1)$ given in (\ref{jaco}) has a  zero eigenvalue at $f=f_{_{TC}}$. Let the eigenvectors of the Jacobian matrix $J(E_1)$ and its transpose $J(E_1)^T$ corresponding to the zero eigenvalue  be $\boldsymbol{\zeta}=[\frac{ a}{b+1}, 1]^T$ and $\boldsymbol{\eta}=[1,0]^T$ respectively. Now, the transversality conditions \cite{Perko} become 
\begin{eqnarray*}
\boldsymbol{\eta}^T\mathcal{F}_f({E_{1}};f=f_{_{TC}})&=&0,\\
\boldsymbol{\eta}^TD\mathcal{F}_f({E_{1}};f=f_{_{TC}})\boldsymbol{\zeta}&=&0,\\
\boldsymbol{\eta}^TD^2\mathcal{F}({E_{1}};f=f_{_{TC}})(\boldsymbol{\zeta},\boldsymbol{\zeta})&=&-4\,{\frac {{a}^{2}}{ \left( b+1 \right) ^{3}}} \neq 0.  
\end{eqnarray*}

Here, $\mathcal{F}=[F_1(u,v),F_2(u,v)]^T$ and all the other notations are the same as in \cite{Perko}.
Thus, all the transversality conditions of degenerate transcritical  bifurcation \cite{degeneratetranscritical,dey2022bifurcation} are satisfied for the system (\ref{ode}) when $f=f_{_{TC}}$.
 \end{proof}

\subsubsection{Saddle-node bifurcation}
Suppose that for the polynomial $Q$ given in \eqref{eqQ}, there exist two  real roots, say  $u_{sn_1}$ and $u_{sn_2}$ of $Q'(u)=0$  with $0<u_{sn_1}<u_{sn_2}$. Then, opposite signs of  $Q(u_{sn_1})$ and $Q(u_{sn_2})$ lead to  three positive roots $u$ of the equation $Q(u)=0$  with $u_1^*<u_{sn_1}<u_2^*<u_{sn_2}<u_3^*$ [see Fig. \ref{Qu}(a)].   Note that for the feasibility of coexisting equilibria, we must have  $u_i^*<\text{min}\{1, u_b\}$ for $i=1,2,3$.  A variation in one temporal parameter may result in two cases:  either $E_2^*$ coincides with $E_1^*$ when $Q(u_{sn_1})=0$ [see Fig. \ref{Qu}(b))] or $E_2^*$ coincides with $E_3^*$ when  $Q(u_{sn_2})=0$ [see Fig. \ref{Qu}(c)] provided  $u_{sn_i}<\text{min}\{1, u_b\}$ for $i=1,2$. Thus,  a saddle-node bifurcation $SN_1$ occurs in the former case and another saddle-node bifurcation $SN_{2}$ occurs in the latter case.  Taking $f$ as the control parameter, the threshold value for $SN_1$ is given by 
$$ f_{_{SN_1}}= \frac{5b^2u_{sn_1}^4-4b^2u_{sn_1}^3+6bu_{sn_1}^2-4bu_{sn_1}+a^2e+1}{2ab}.$$ 
 To obtain the threshold value $f_{_{SN_{2}}}$, we replace  $u_{sn_1}$  with $u_{sn_2}$ in the  expression for $f_{_{SN_1}}.$  In the following proposition, transversality conditions are verified for this bifurcation.

 \begin{figure}[!ht]
\begin{subfigure}[b]{.32\textwidth}
 \centering
\includegraphics[scale=0.4]{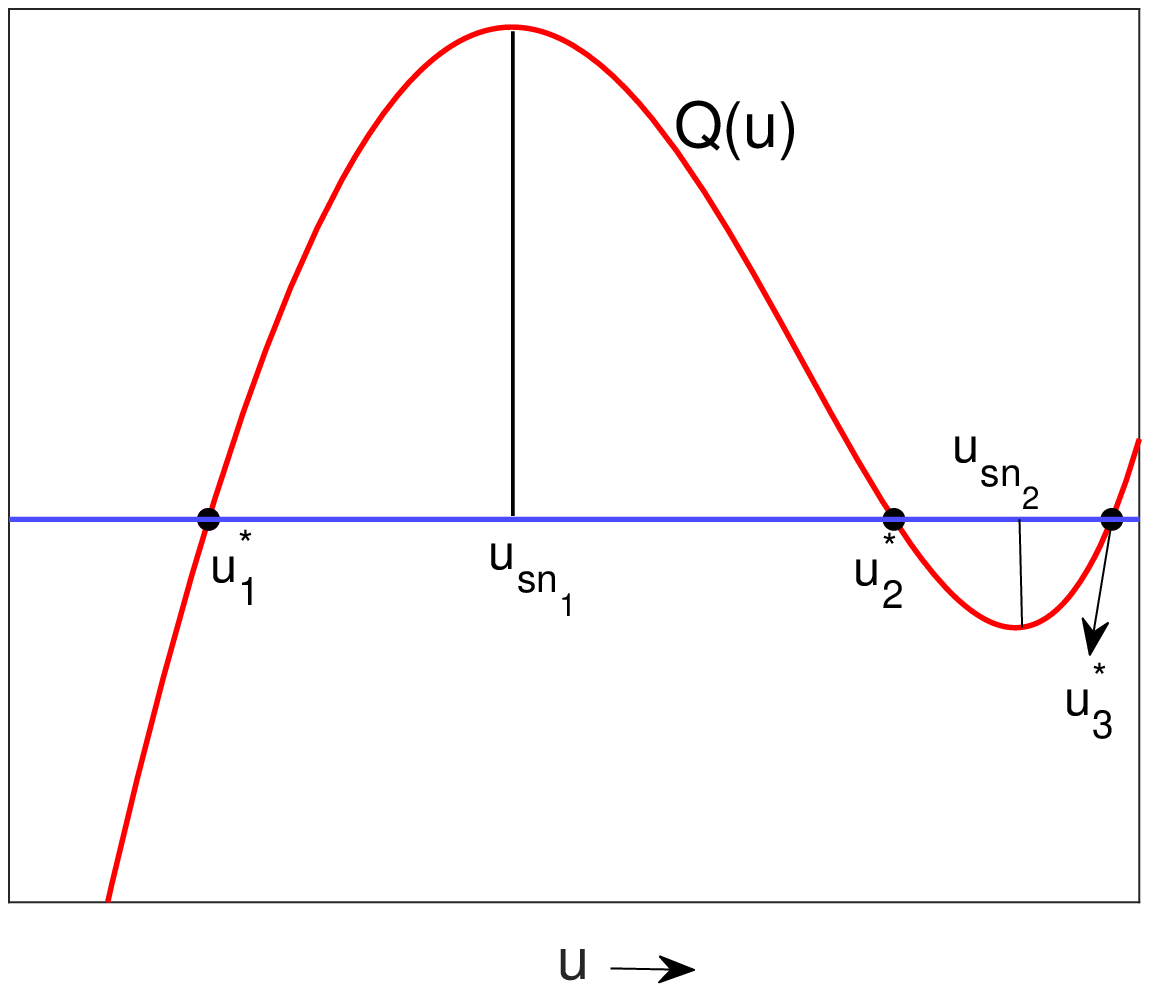}\\ 
 \caption{}
  \end{subfigure}
 \begin{subfigure}[b]{.32\textwidth}
 \centering
\includegraphics[scale=0.4]{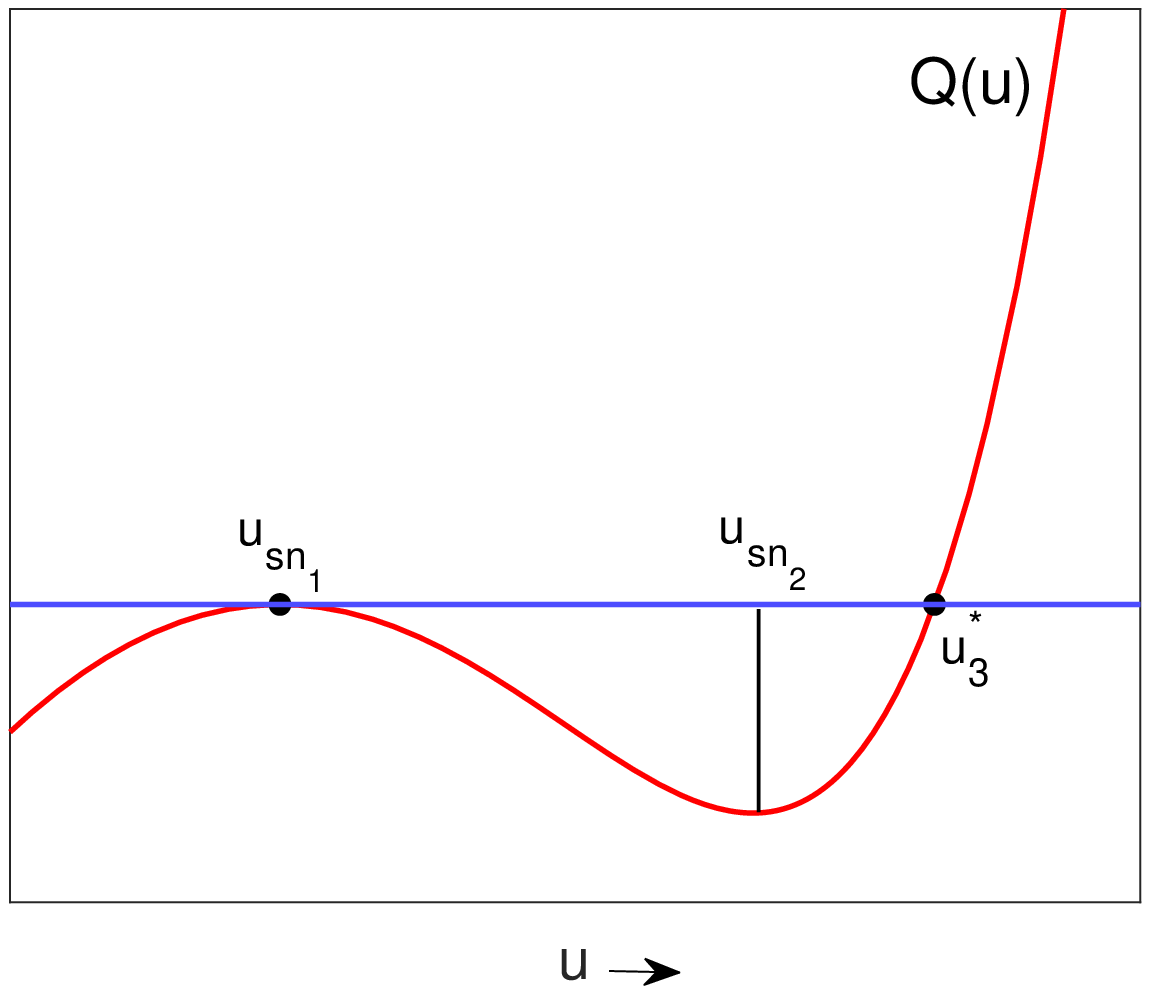}\\ 
 \caption{}
  \end{subfigure} 
  \begin{subfigure}[b]{.32\textwidth}
 \centering
\includegraphics[scale=0.4]{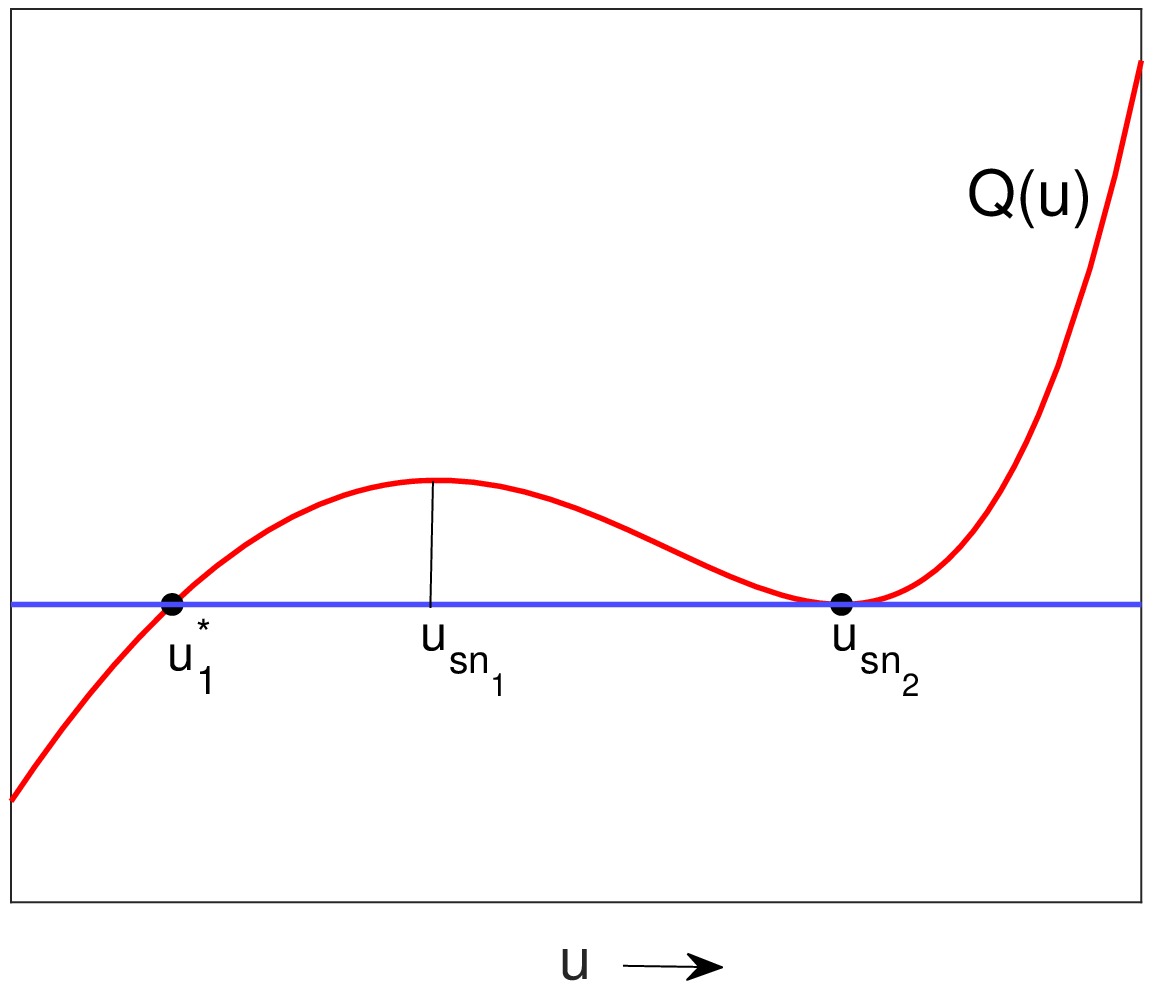}\\ 
 \caption{}
  \end{subfigure} 
 \caption{Plot of $Q(u)$ for three different cases.}
\label{Qu}
\end{figure}
 
 \begin{prop} \label{SNprop}
 When $Q(u)=0$  has a double root $u<\text{min}\{1, u_b\}$,  then the temporal system (\ref{ode}) exhibits saddle-node bifurcation when the control parameter $f$ is varied.
 \end{prop}
 \begin{proof}
 
  Suppose that $u_{sn},$ with $u_{sn}<\text{min}\{1, u_b\}$, is a double root of $Q(u)=0$, i.e., $Q(u_{sn})=Q'(u_{sn})=0$ but $Q''(u_{sn})\neq 0$ when $f=f_{_{SN}}$. Let the corresponding interior equilibrium point be $E_{SN}^*=(u_{sn},v_{sn}).$ Therefore,  the nontrivial nullclines $f_1(u,v)=0$ and $f_2(u,v)=0$ touch each other at $E_{SN}^*,$ where both  of them have  same slope  $\frac{dv^{(f_1)}}{du}|_{E_{SN}^*}=\frac{dv^{(f_2)}}{du}|_{E_{SN}^*}.$ Using $\di{\frac{dv^{(g)}}{du}=-\frac{\frac{\partial g}{\partial u}}{\frac{\partial g}{\partial v}}},$ we find
  $$\text{det}(J({E_{_{SN}}^*}))= \left[\ 
 uv\left(\frac{\partial f_1}{\partial u}\frac{\partial f_2}{\partial v}-\frac{\partial f_1}{\partial v}\frac{\partial f_2}{\partial u} \right)
 \right]_{E_{SN}^*}=0.$$ Therefore,  the Jacobian matrix  $J(E_{SN}^*)$ has a  zero  eigenvalue.  Let $\di{\boldsymbol{\zeta}=[p, 1]^T}$ and $\di{\boldsymbol{\eta}=[1,q]^T}$ respectively be the  eigenvectors of $J({E_{SN}^*})$ and $[J({E_{SN}^*})]^T$ corresponding to the zero eigenvalue, where 
 $$
 \di{p=-{\frac {a}{1+ \left( 3\,{u_{sn}}^{2}-2\,u_{sn} \right) b}}\;\text{ and }\;q=-{\frac {au}{ \left( b{u}^{2}+1 \right) v}}}.
 $$
 To check the transversality conditions, we calculate
\begin{subequations}
\begin{equation*}
\boldsymbol{\eta}^T\mathcal{F}_f({E_{SN}^*};f=f_{_{SN}})={\frac {au_{sn}}{b{u_{sn}}^{2}+1}},
\end{equation*}
\begin{equation*} \label{sntc}
   \boldsymbol{\eta}^TD^2\mathcal{F}({E_{SN}^*};f=f_{_{SN}})(\boldsymbol{\zeta},\boldsymbol{\zeta})=\Big(
\frac{\partial^2 F_1}{\partial u^2}p^2+2\frac{\partial^2 F_1}{\partial u \partial v}p +\frac{\partial^2 F_1}{\partial v^2}+q (\frac{\partial^2 F_2}{\partial u^2}p^2+2\frac{\partial^2 F_2}{\partial u \partial v}p +\frac{\partial^2 F_2}{\partial v^2})\Big)_{({E^*_{SN}};f_{_{SN}})},
\end{equation*}
\end{subequations}
where we omit explicit expressions in the  last equation since it is too cumbersome. The system (\ref{ode}) undergoes a non-degenerate  saddle-node bifurcation \cite{Perko}  at $f=f_{_{SN}}$ if $\boldsymbol{\eta}^T\mathcal{F}_f({E_{SN}^*};f=f_{_{SN}})\ne 0$ and 
$\boldsymbol{\eta}^TD^2\mathcal{F}({E_{SN}^*};f=f_{_{SN}})(\boldsymbol{\zeta},\boldsymbol{\zeta})\neq0$. Clearly, the first conditions is satisfied and the last condition is verified numerically.
\end{proof}
 \subsubsection{Cusp bifurcation}
  We have found  that  a transcritical bifurcation $TC$ occurs at $E_1$ and a saddle-node bifurcation  $SN_{2}$ occurs at $E_{SN_{2}}^*=(u_{sn_2},v_{sn_2})$ with variation of the control parameter $f$. Now, $E_{SN_{2}}^*$ and $E_1$ coincide with the variation of a different temporal parameter $b$. Thus,  the saddle-node bifurcation curve $SN_{2}$ and transcritical curve $TC$  intersect at a cusp bifurcation point ($f_{_{CP}},b_{_{CP}}$) in the $f$-$b$ parametric plane.
\begin{prop}
 The system (\ref{ode})  encounters  a cusp bifurcation when  $Q(1)=Q'(1)=0$.
 \end{prop}
  \begin{proof}
   Suppose at $f=f_{_{CP}}$ and $b=b_{_{CP}}$,  interior  equilibria $E_2^*$ and $E_3^*$ coincide with the axial equilibrium point  $E_1$. Here, the curves $TC$ and $SN_2,$ corresponding to the transcritical and saddle-node bifurcations respectively, meet  at the cusp bifurcation point $(f_{_{CP}},b_{_{CP}})$ in the $f$-$b$ plane. From proposition \ref{SNprop}, we know that $Q(u)$ has a double root. Since transcritical bifurcation also happens simultaneously,  $Q(u)$ has a double root $1$, i.e.,  $Q(1)=Q'(1)=0$ at the cusp bifurcation threshold. Here,  $f_{_{CP}}$ is the positive root of the equation 
   $2z^3-aez^2+ae^2=0$ and the corresponding $\di{b_{_{CP}}=({ae-f_{_{CP}}})/{f_{_{CP}}}}.$ 
  \end{proof}
 \subsubsection{Hopf and Bautin bifurcation}
 At a Hopf-bifurcation threshold, a stable equilibrium point changes stability and a limit cycle is generated that can be stable or unstable. In proposition \ref{prop2}{\it (i)},  we have observed that the  interior equilibrium $E_1^*$ changes its stability when the trace of  $J(E_1^*)$ changes its sign  due to variation in $b$.  We thus have $tr(J(E_1^*))=0$ at the Hopf bifurcation threshold $b=b_H.$ Since $(u_1^*,v_1^*)$ depends on $b$, the threshold $b_H$ is an implicit expression. The temporal system (\ref{ode}) exhibits a Hopf bifurcation  at  $b=b_{_{H}}$ if the  non-hyperbolicity and transversality conditions given below are  satisfied: 
\begin{eqnarray*}
&H1:&  \; \;  \text{det}[J(E_1 ^*;b=b_{_H})]>0,\\   
&H2:& \; \; \frac{d}{d b} (tr[J(E_1 ^*)] ) |_{b=b_{_H}} \neq 0.
\end{eqnarray*} 
Hopf-bifurcations are classified as supercritical or subcritical depending on the stability of the generated limit cycle. Supercritical Hopf-bifurcation occurs when the limit cycle is stable, whereas subcritical Hopf-bifurcation occurs when the limit cycle is unstable. The first case corresponds to the first Lyapunov coefficient $l_1<0,$ whereas the second one corresponds to $l_1>0$ \cite{Perko}. Due to the unavailability of explicit expression of interior equilibrium $E_1^*$, it is difficult to determine the sign of $l_1$ analytically. However, we obtain the value of $l_1$  numerically using the Matcont software. If we fix the temporal parameters $a=7$, $f=1.05$ and  $e=0.95,$  then the system (\ref{ode}) exhibits a subcritical Hopf bifurcation around $E_1^*=(0.2412,0.1455)$ at $b_{_{H}} =5.8759$ with $l_1=0.0085.$  For the same parameter values except $f=1.06$, a supercritical Hopf bifurcation has been  found around $E_1^*=(0.2443,0.1455)$ at $b_{H} =5.8234$ with  $l_1=-0.043$. 

Clearly, the first Lyapunov coefficient $l_1$ vanishes between $f=1.05$ and $f=1.06$ with $a=7$ and $e=0.95.$ When  $l_1$ becomes zero, the system (\ref{ode}) undergoes a codimension-2 bifurcation known as a Bautin bifurcation or  generalized Hopf bifurcation (GH). A global saddle-node bifurcation curve of the limit cycle, where a stable limit cycle collides with an unstable limit cycle, emerges  from the GH point in the two-parametric plane of bifurcation. The system (\ref{ode}) undergoes  Bautin bifurcation  at $E_1^*=(0.2417,0.1455)$ for the bifurcating parameter values $f_{_{GH}}=1.0517 $ and  $b_{_{GH}}=5.8671$ with $a=7$ and $e=0.95.$
 \subsubsection{Bogdanov-Takens bifurcation}
 A Bogdanov-Takens (BT) bifurcation is a codimension-2 bifurcation that occurs in a dynamical system when a Hopf bifurcation colloids with a saddle-node bifurcation.  The system (\ref{ode}) exhibits a Hopf bifurcation and two saddle-node bifurcations which suggests that a BT bifurcation may occur in our system. Both the  determinant and trace of the Jacobian matrix, evaluated at an equilibrium point,  vanish simultaneously at a BT bifurcation point. 
 A homoclinic or heteroclinic global bifurcation curve comes out from the BT point in the two-parametric bifurcation plane. Here, we take $f$ and $b$ as  control parameters for BT bifurcation. For fixed  parameter values $a=7$ and $e=0.95,$ the system (\ref{ode}) exhibits a BT bifurcation for $E_1^*(0.4086,0.1637)$ at $(f_{_{BT}},b_{_{BT}})=(1.2388,5.6146).$
 \begin{figure}[!t]
\centering
\includegraphics[scale=0.75]{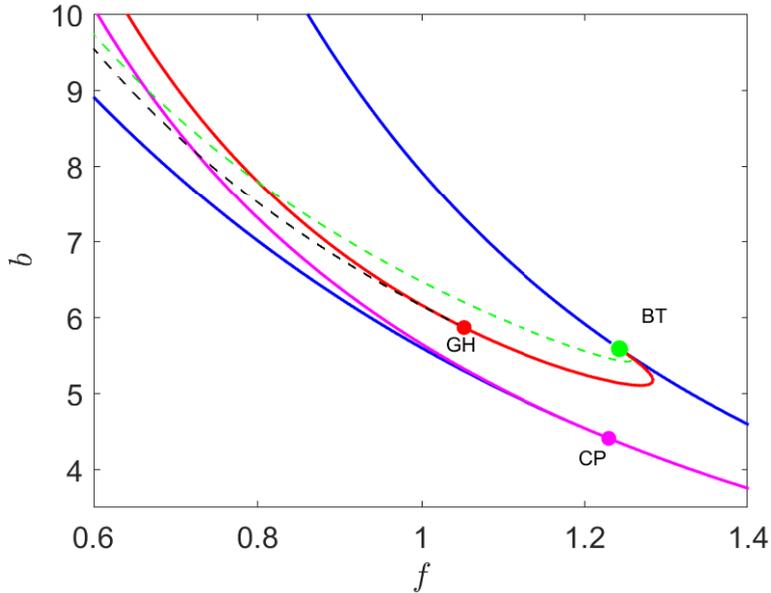}\\
\caption{Two parametric bifurcation diagram in the $f$-$b$ plane. Magenta, blue, and  red colour curves represent the transcritical, saddle-node, and  Hopf bifurcations respectively. Here, the upper and lower blue curves represent the SN$_1$ and SN$_2$, respectively.  The green and black dashed curves denote the global curves, namely  homoclinic and saddle-node bifurcation of limit  cycles, respectively. Further, colored solid dots represent codimension-2 bifurcation points. Other parameter values are $a=7$ and $e=0.95$. }
\label{2dbif}
\end{figure}
\subsection{Numerical visualisation}
Here, we visualize previously described local and global bifurcations with the help of numerical simulations. For fix parameter values $a=7,$ $e=0.95$, we plot a two-dimensional bifurcation diagram in the  $f$-$b$ parametric plane (see Fig. \ref{2dbif}). The coordinate of the cusp point (CP) is (1.2270,4.4195) and the coordinates of GH and BT points have already been mentioned earlier. 
 
\begin{figure}[!t]
\begin{subfigure}[b]{.48\textwidth}
 \centering
\includegraphics[scale=0.55]{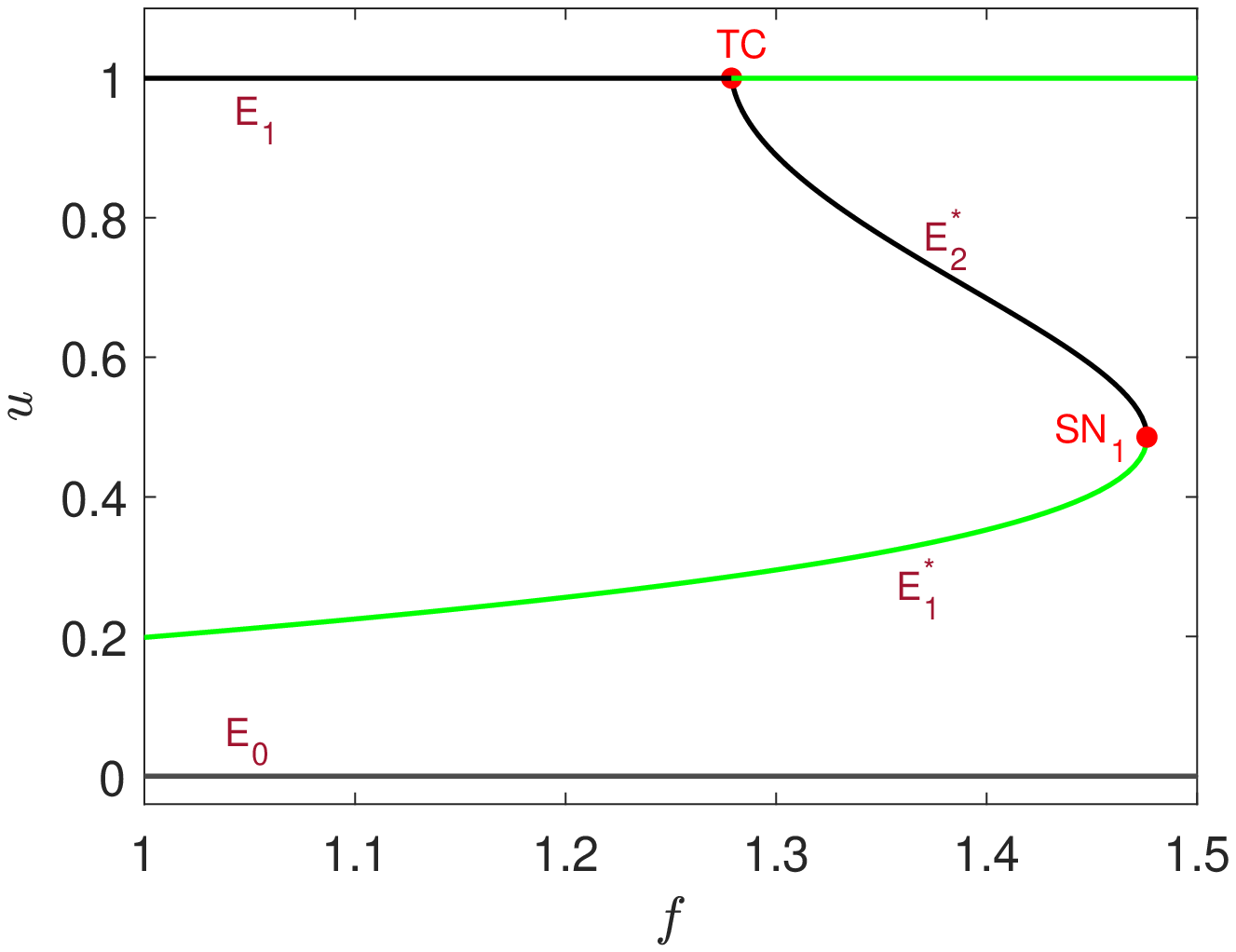}\\ 
 \caption{}
  \end{subfigure}
 \begin{subfigure}[b]{.48\textwidth}
 \centering
\includegraphics[scale=0.55]{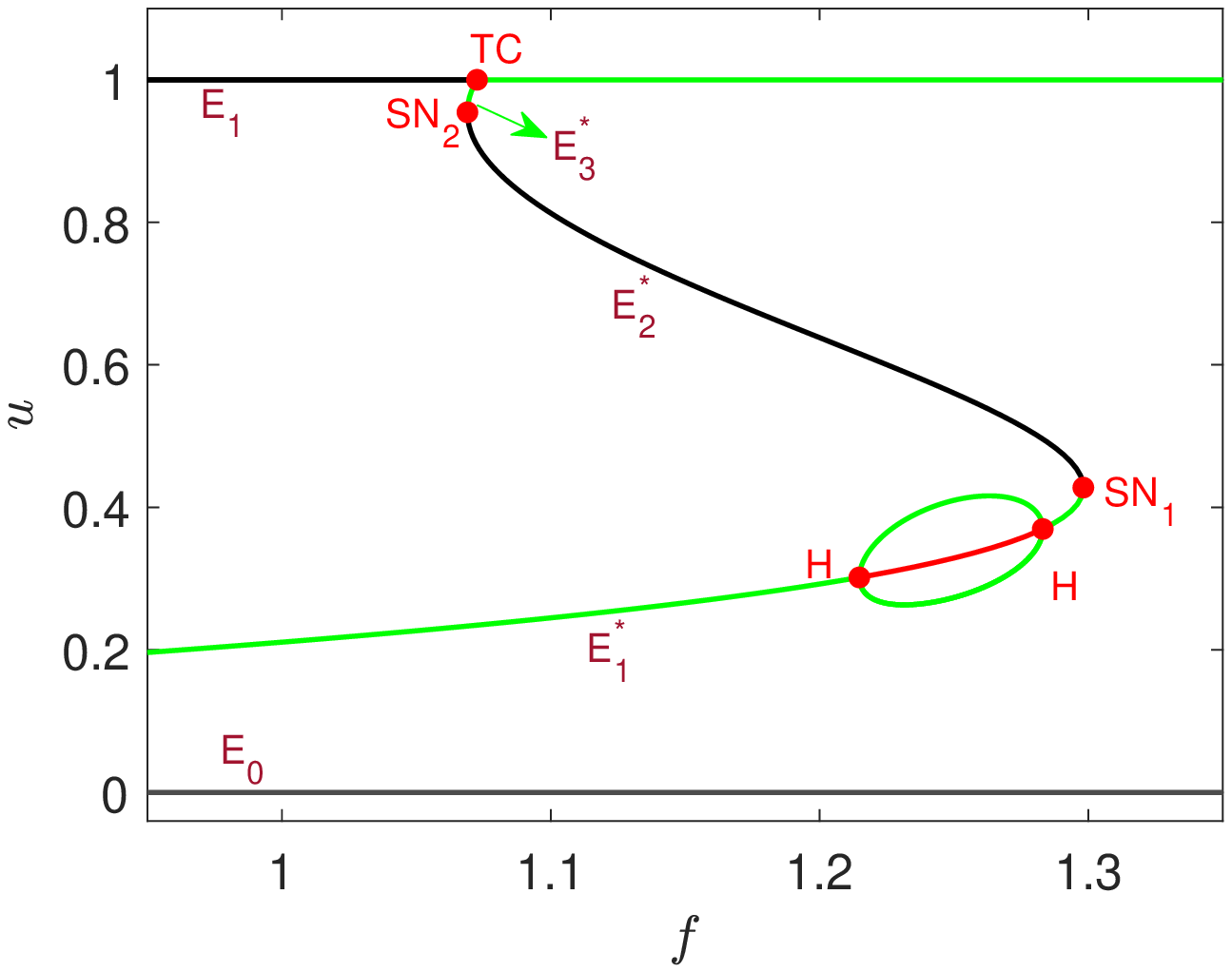}\\ 
 \caption{}
  \end{subfigure}
  \begin{subfigure}[b]{.48\textwidth}
 \centering
\includegraphics[scale=0.55]{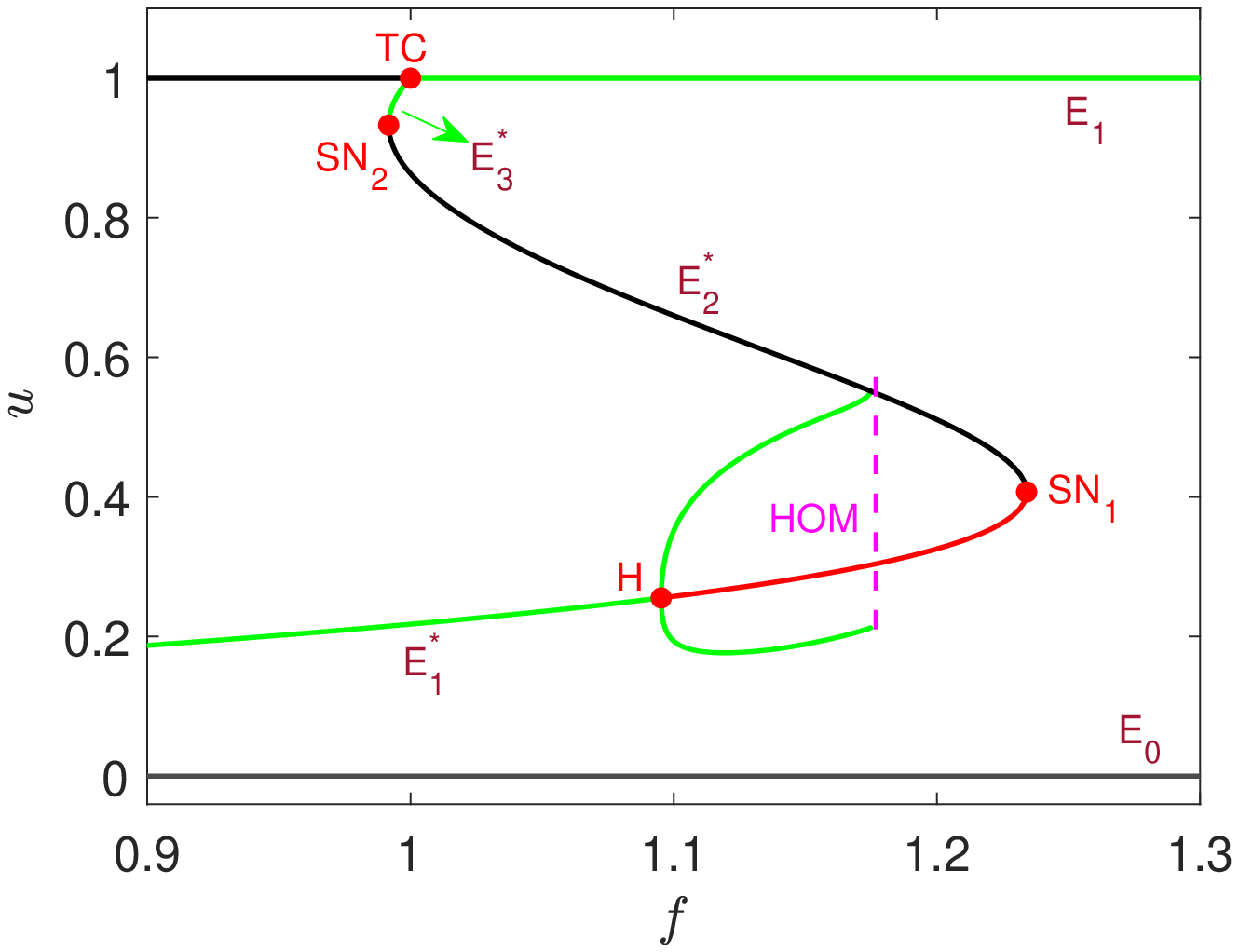}\\ 
 \caption{}
  \end{subfigure}\hspace{1em}
  \begin{subfigure}[b]{.48\textwidth}
 \centering
\includegraphics[scale=0.55]{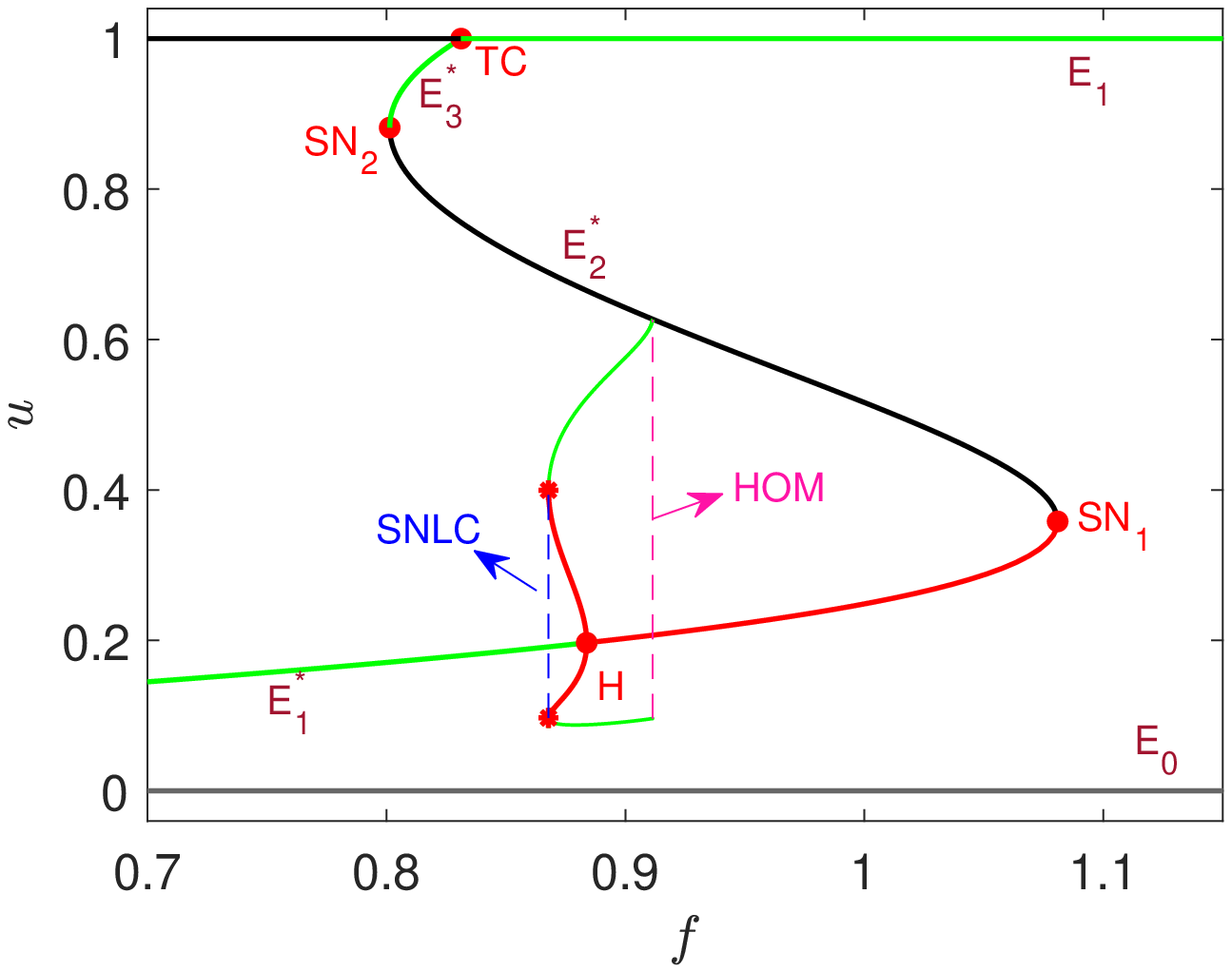}\\ 
 \caption{}
  \end{subfigure} 
 \caption{Bifurcation diagram  of the system \eqref{ode} against parameter $f$ for different values of $b$: (a) $b=4.2$, (b) $b=5.2$, (c) $b=5.65$, and (d) $b=7$. Here SN$_1$ and SN$_2$ represent saddle-node bifurcations; TC and H denote transcritical and Hopf bifurcations respectively. Also, SNLC and HOM represent the saddle-node bifurcation of limit cycles and homoclinic bifurcation respectively. Other parameter values are $a=7$  and $e=0.95$.} 
\label{1dbif}
\end{figure}

To better understand how the system dynamics change across the local and global bifurcation curves in  Fig.~\ref{2dbif}, we consider four different values of $b$ and plot their corresponding one-parametric bifurcation diagrams in Fig.~\ref{1dbif}. For $b=4.2,$ equilibria $E_1^*$ and $E_2^*$ appear in the system through the saddle-node bifurcation  SN$_1$, and $E_2^*$ disappears from feasibility region through a transcritical bifurcation TC [see Fig. \ref{1dbif}(a)]. For $b=5.2,$ two qualitative changes are observed [see Fig. \ref{1dbif}(b))]. First, the system \eqref{ode} exhibits another saddle-node bifurcation SN$_2.$  Second, two supercritical Hopf bifurcations occur around $E_1^*$ which leads to an oscillatory coexisting solution between two Hopf bifurcation thresholds. Note that these two Hopf bifurcation points lie on the same Hopf curve (marked by red color) below the BT point in Fig. \ref{2dbif}. The system \eqref{ode} shows a bistable dynamics in between $E_1^*$ and $E_3^*$  for parameter value in $f_{SN_2}<f<f_{TC}.$

Figures \ref{1dbif}(a) and (b) correspond to parameter value $b<b_{BT}.$ Now, we consider the case $b>b_{BT}$ for which two bifurcation diagrams are shown in Figs. \ref{1dbif}(c) and (d).  Here, the coexisting equilibrium point $E_1^*,$ generated due to saddle-node bifurcation SN$_1,$ becomes unstable compared to the stable case for $b<b_{BT}$. The Hopf bifurcating limit cycle disappears due to  collision with the coexisting equilibria $E_2^*$ through a homoclinic bifurcation for $b_{BT}<b<b_{GH}$ [see Fig. \ref{1dbif}(c)]. For $b>b_{GH},$ the Hopf bifurcation becomes subcritical and the corresponding bifurcation diagram is shown in Fig. \ref{1dbif}(d). A stable and an  unstable limit cycles are generated due to a global homoclinic bifurcation (HOM) and a subcritical Hopf bifurcation respectively.  These two limit cycles collide and disappear from the system dynamics through the saddle-node bifurcation of limit cycles (SNLC). Interestingly, the system shows tristability among three attractors, specifically, two equilibria $E_1^*,$ $E_1$ and the stable limit cycle around $E_1^*$ for parameter value $f_{SNLC}<f<f_H.$ The unstable limit cycle around $E_1^*$ and the stable manifold of the $E_2^*$ act as separatrix of these three attractors.
\begin{figure}[!t]
\begin{subfigure}[b]{.48\textwidth}
 \centering
 \includegraphics[scale=0.5]{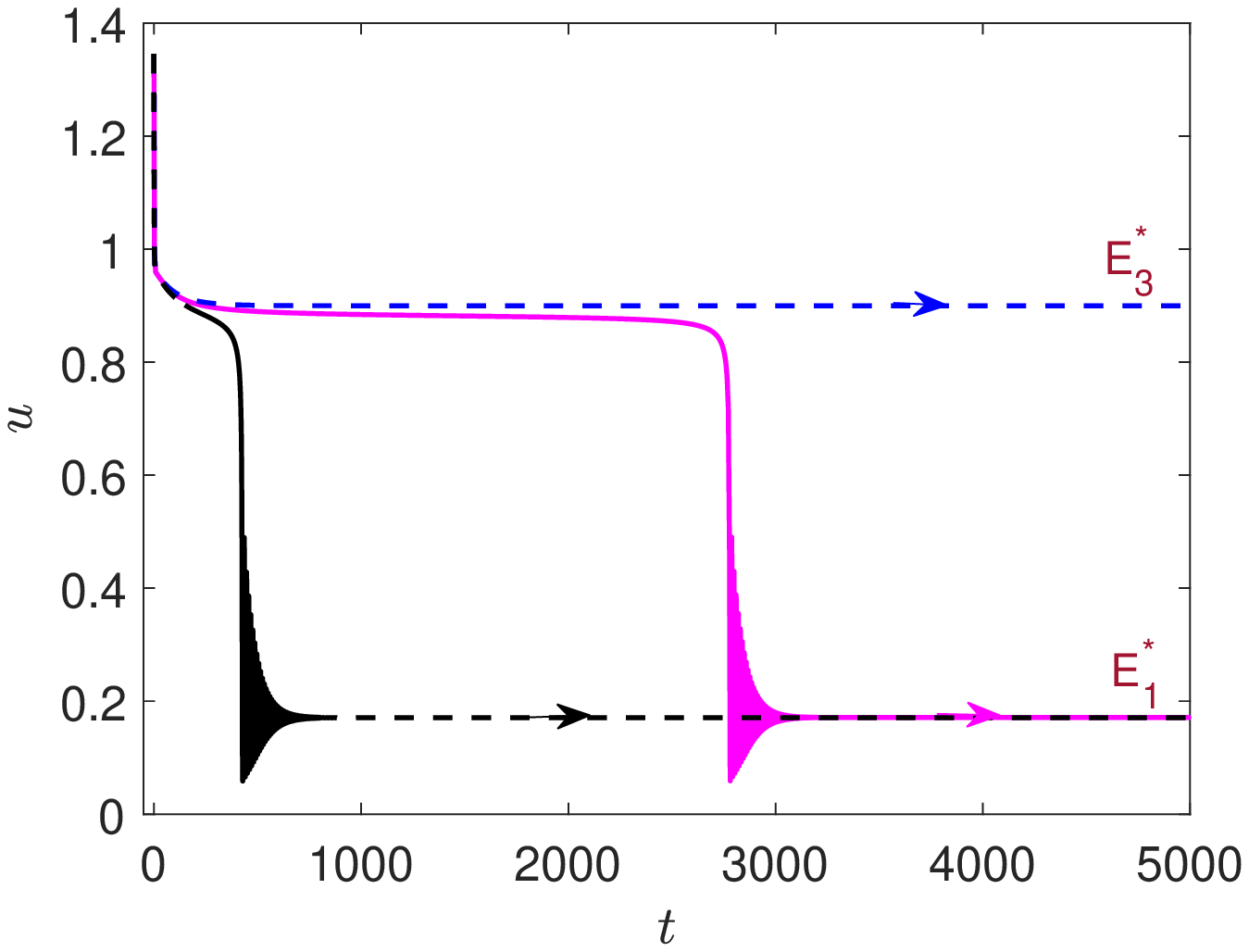}\\
 \caption{}
  \end{subfigure} 
\begin{subfigure}[b]{.48\textwidth}
 \centering
\includegraphics[scale=0.5]{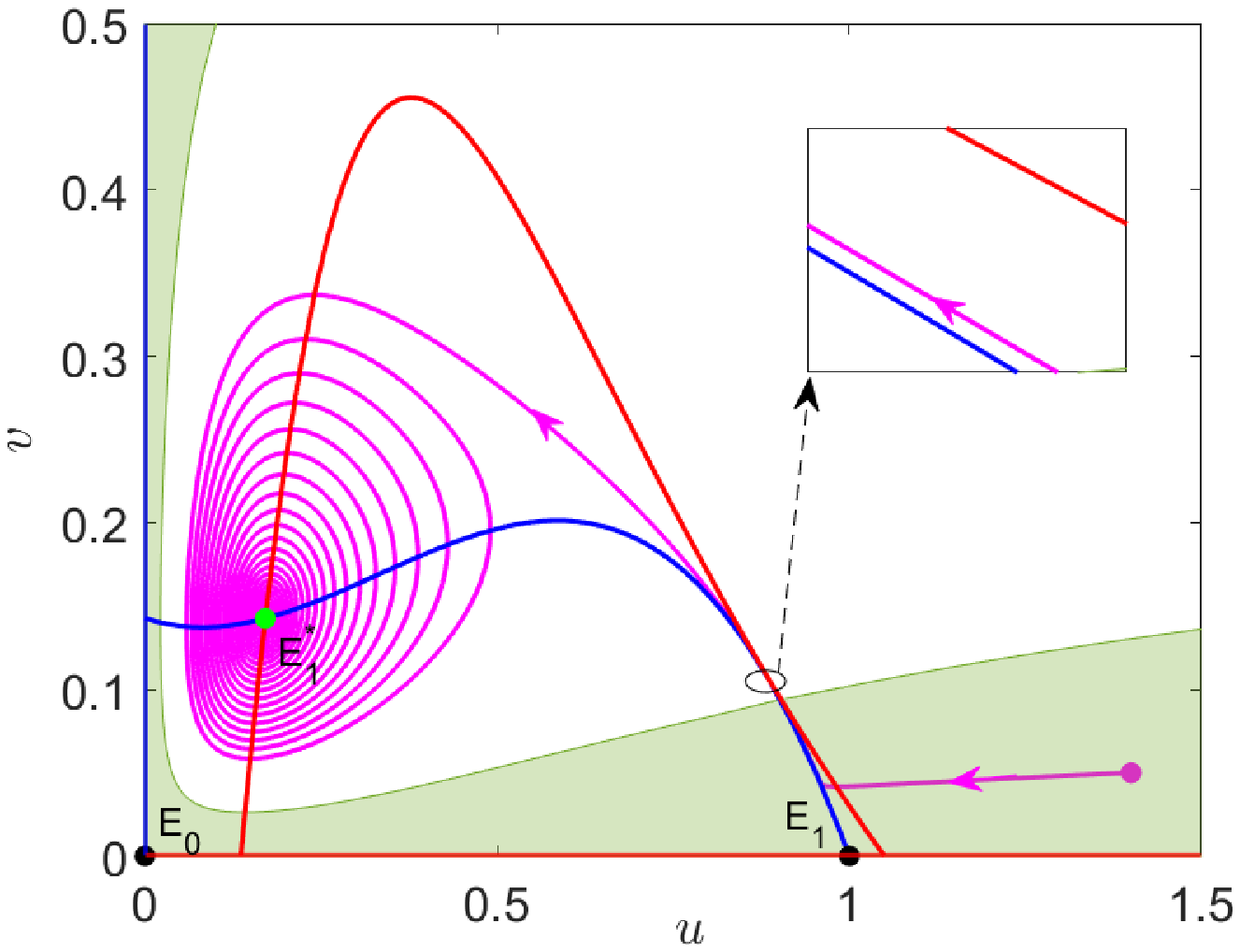}\\  
 \caption{}
  \end{subfigure}
  \caption{Temporal transient dynamics due to saddle-node bifurcation for $a=7,\;b=7$  and $e=0.95$ with  $(u(0),v(0))=(1.4,0.05)$ as initial condition. (a)
  Here, blue, magenta, and black colour curves  correspond to time evolution of $u$ for $f=0.802,\;f=0.8013,$ and $f=0.80$ respectively. (b) Phase trajectory 
 for $f=0.8013$ marked with magenta colour curve. The prey and predator nullclines are also shown. Any solution starting with the initial condition in the shaded region exhibits long transient dynamics similar to that of $f=0.8013$.    }
\label{steadystatetransient}
\end{figure}

\subsubsection*{Temporal transient dynamics}
Here we explore some interesting long-temporal  transient dynamics of the system \eqref{ode}  due to the global and local bifurcations. We consider the same parameter values as in Fig \ref{1dbif}(d) and $f$ is chosen near the SN$_2$ threshold  $f_{SN_2}=0.801336.$ A long 
transient dynamics is observed in the time evolution of $u$ for values of $f$ near $f_{SN_2}$ starting from same initial condition $(u(0),v(0))\equiv (1.4, 0.05).$  To investigate this in detail, we consider the time evolution of $u$ for three different values of $f$ near  $f_{SN_2}$ [see Fig. \ref{steadystatetransient}(a)].  For $f=0.802,$ the equilibrium point $E_3^*$ is asymptotically stable and the time evolution of $u$ approaches $E_3^*$ rapidly. Note that $E_3^*$ disappears for $f<f_{SN_2}=0.801336$ due to a saddle-node bifurcation. The time evolution of $u$ also rapidly approaches the globally stable equilibrium point $E_1^*$ for $f=0.80$. However, if we take $f=0.8013,$ then $u$ spends a considerable amount of time near the solution for $f=0.802$. However, equilibrium point $E_3^*$ is not present in the system  for $f=0.8013$ and the time evolution of $u$ finally settles into the globally stable equilibrium state $E_1^*$. Here, a shadow of $E_3^*$ acts as a ghost attractor in the system \eqref{ode} which causes such long transient dynamics. This long transient dynamics is also observed for other initial conditions too. Time evolution of $u$ shows similar behaviour for $f=0.8013$ if the initial condition lies in the shaded region shown in Fig. \ref{steadystatetransient}(b). When a  trajectory passes through the narrow region between the nullclines [see the inset of the Fig. \ref{steadystatetransient}(b))], their $u$ directional as well as $v$ directional velocities become very small. This leads to the phenomenon of ghost attractor and long transient dynamics in the system. Figure \ref{steadystatetransient}(a) also illustrates hysteresis exhibited by the system \eqref{ode}. A small change in the parameter $f$ results in a significant drop or change (from $E_3^*$ to $E_1^*$) in the steady state solution.
\begin{figure}[!t]
\begin{subfigure}[b]{.48\textwidth}
 \centering
\includegraphics[scale=0.5]{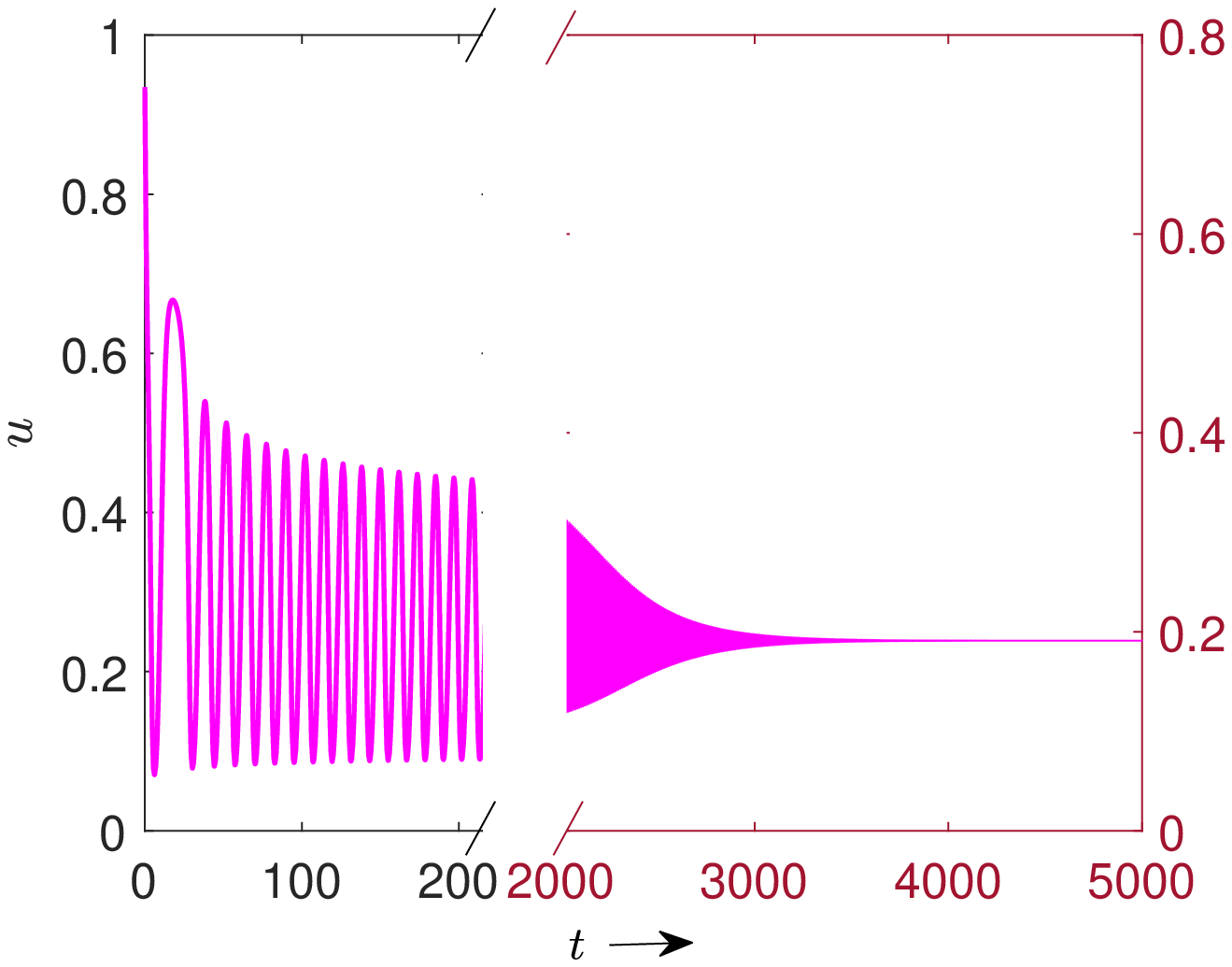}\\ 
 \caption{}
  \end{subfigure}
 \begin{subfigure}[b]{.48\textwidth}
 \centering
\includegraphics[scale=0.5]{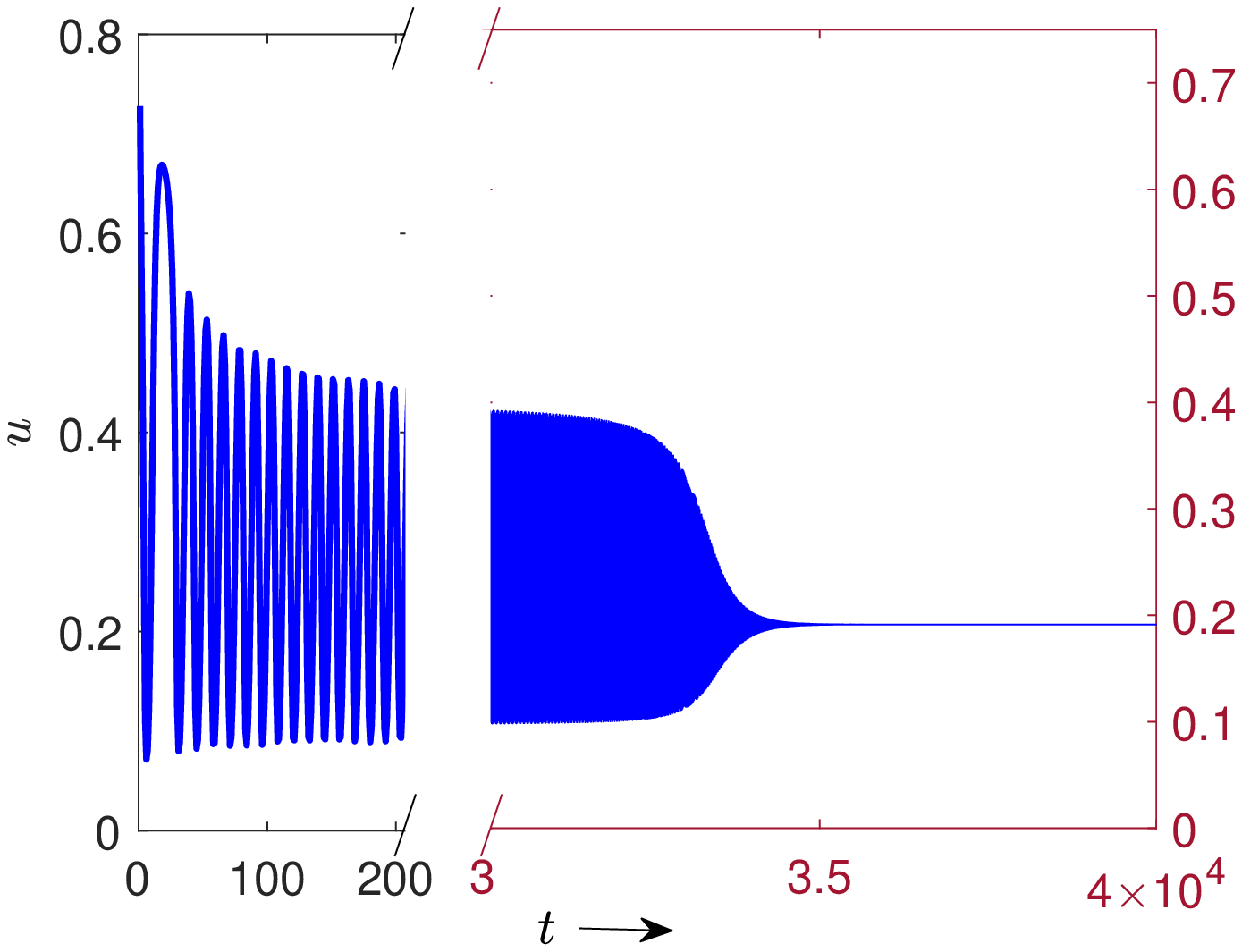}\\ 
 \caption{}
  \end{subfigure} 
 \caption{Time evolution of $u$ showing long oscillatory transient dynamics due to global bifurcation SNLC with $(u(0),v(0))=(0.97,0.45)$ as initial condition: (a) $f=0.867$ and (b) $f=0.867682$. Other parameter values are $a=7,\;b=7, $  and $e=0.95$.}
\label{global_transient}
\end{figure}

The system \eqref{ode} also exhibits oscillatory long transient dynamics due to the global SNLC bifurcation. The coexisting equilibrium point $E_1^*$ is surrounded by an unstable limit cycle, that is also surrounded by a stable limit cycle for $f_{SNLC}=0.867805<f<f_H=0.883805.$  Both the limit cycles collide and disappear from the system dynamics through the SNLC.  We have plotted the time evolution of $u$ for two different values of $f$ near $f_{SNLC}$ in Fig. \ref{global_transient} with initial condition $(u(0),v(0))=(0.97,0.45)$. The coexisting equilibrium point $E_1^*$ is locally asymptotically stable for $f<f_{SNLC}.$ The time evolution of $u$ converges to $E_1^*$ faster for $f=0.867$ compared to that for $f=0.867682.$ Thus, oscillatory long transient dynamics is observed for values of $f$ near $f_{SNLC}$ due to global SNLC bifurcation.


\section{Spatio-temporal model} \label{spatialmodel}
Now, we consider the spatio-temporal model \eqref{pde} to examine the effects of taxis and diffusion on the dynamics. We first establish global existence and boundness of the solution using $L^p$-$L^q$ estimates. Next, we examine the stability of the homogeneous steady states by converting the spatio-temporal problem to an eigenvalue problem. Finally, the Turing bifurcation threshold is also obtained.  

\subsection{ Global existence and boundedness of  solution}
The global existence and boundedness of solution with  different taxis and reaction  kinetics have been established in \cite{wu2018dynamics,cai2022asymptotic,carrillo2013uniqueness}. Here, we establish the same with non-monotonic functional response. Let $\di{W^{1,p}(\Omega)}$ be the Sobolev space  consisting of functions in ${L^p(\Omega)}$ that have weak first order partial derivatives and they belong to ${L^p(\Omega)}.$ 
\begin{lemma}\label{lem1}
Assume that $(u_0,v_0)\in [W^{1,p}(\Omega)]^2$ with $u_0,v_0\geq 0$ and $p>n.$ Then the following  hold for the system \eqref{pde}
\begin{itemize}
    \item[(i)] There exists a positive constant $T$ (maximal existence time) such that the system \eqref{pde} has a unique local-in-time  non-negative solution $(u(x,t),v(x,t))\in [C(\bar\Omega\times [0,T))\cap C^{2,1}(\bar\Omega\times (0,T))]^2.$
    \item[(ii)] The $L_1$ bounds of $(u(x,t),v(x,t)),$ for all time $t\in(0,T),$ satisfy
    $$\int\limits_\Omega u(x,t)dx\leq A \text{ and } \int\limits_\Omega v(x,t)dx\leq B,$$ where $A$ and $B$ are defined in  the proof.
    \item[(iii)] There exists a constant $C$ such that $0\leq u(x,t)\leq C$ and $ v(x,t)\ge 0$ for all $(x,t)\in \bar \Omega\times (0,T).$
\end{itemize}
\end{lemma}
\begin{proof}
$(i)$ Here we apply Amann theory \cite{amann1990dynamic} to prove the local existence of $(u(x,t),v(x,t))$. In terms of $U=(u,v)^T,$ the system \eqref{pde} is written as
\begin{equation}\label{10}
     \begin{cases}
 \frac{\partial U}{\partial t}=\nabla\cdot(\mathcal{A}(U)\nabla U)+F(U), \; x\in \Omega,\;t>0,\\
 \frac{\partial U}{\partial n}=(0,0)^T,\;x\in \partial\Omega,\;t>0,\\
 U(x,0)=(u_0,v_0)^T,\;x\in\;\Omega,
 \end{cases}
 \end{equation}
 where \begin{equation}\label{Au}
     \mathcal{A}(U)=\begin{pmatrix}
     1 & 0\\ cv & d
 \end{pmatrix} \text{ and } F(U)=\begin{pmatrix}
     F_1(U)\\F_2(U)
 \end{pmatrix}.
 \end{equation}
 Since the matrix $\mathcal{A}(U)$ is positive definite, the system \eqref{pde} is normally parabolic and the local existence of classical solutions follows from Theorem 7.3 in \cite{amann1990dynamic}. Thus, there exists a maximal existence time $T>0$  such that the system \eqref{pde} has a unique  solution $(u(x,t),v(x,t))\in C(\bar\Omega\times [0,T))\cap C^{2,1}(\bar\Omega\times (0,T)),$ with $u,v\geq0.$ 
 \\
 $(ii)$ Let $\di{K_1(t)=\int\limits_\Omega u(x,t)dx \text{ and } K_2(t)=\int\limits_\Omega v(x,t)dx.}$
 Integrating the first equation and using no-flux boundary condition, we find 
 $$\frac{d K_1}{dt}=\int\limits_\Omega F_1(u,v)dx\leq\int\limits_\Omega(u-u^2)dx.$$
Using  H\"{o}lder inequality, we find $\di{\int\limits_\Omega u^2\;dx\geq \frac{1}{|\Omega|}\left(\int\limits_\Omega u\;dx\right)^2},$ which leads to $$\frac{d K_1}{dt}\leq \left(K_1-\frac{K_1^2}{|\Omega|}\right).$$
Therefore, $\di{K_1(t)\leq \text{ max } \Big\{||u_0||_{L^1(\Omega)},|\Omega|\Big\}\equiv A.}$
We also have $$\frac{d}{dt}\left(K_1+e K_2\right)\leq \int\limits_\Omega u\;dx-ef\int\limits_\Omega v\;dx =-f(K_1+eK_2)+(f+1)K_1 .   $$
Since $K_1\leq A,$ the above inequality leads to 
$$
\di{K_2\leq \frac{1}{e}(K_1+eK_2)\leq \frac{1}{e}\left(||u_0||_{L^1(\Omega)}+e||v_0||_{L^1(\Omega)}+ \frac{(f^2+1)}{f}A\right)\equiv B}.$$
$(iii)$ From \eqref{pde}, we have 
\begin{equation*}
     \begin{cases}
 \frac{\partial u}{\partial t}=\nabla^2 u+F_1(u,v)\leq \nabla^2 u+u(1-u), \; x\in \Omega,\;t>0,\\
 \frac{\partial u}{\partial n}=0,\;x\in \partial\Omega,\;t>0,\\
 u(x,0)=u_0(x),\;x\in\;\Omega.
 \end{cases}
 \end{equation*}
Using the maximum principle for the parabolic equation \cite{pao2012nonlinear}, we have 
$$0\leq u(x,t)\leq \text{ max } \{||u_0||_{L^\infty(\Omega)},1\}\equiv C\;\text{for all } x\in \Omega,\;t>0.$$
Now,  the second equation in \eqref{pde} can be treated  as a scalar linear equation in $v$. This implies $v(x,t)\ge 0\; \text{for all } x\in \Omega,\;t>0.$ 
\end{proof}
\begin{theorem}
For any  $(u_0,v_0)\in [W^{1,p}(\Omega)]^2$ with $u_0,v_0\geq 0$ and $p>n,$ the system \eqref{pde} admits a unique global solution $(u(x,t),v(x,t))\in [C(\bar\Omega\times [0,\infty))\cap C^{2,1}(\bar\Omega\times (0,\infty))]^2\equiv \mathcal{S}(\bar\Omega)$ and the solution is uniformly bounded, i.e.,
there exists a constant $R(||u_0||_{L^\infty(\Omega)},||v_0||_{L^\infty(\Omega)})>0$ such that $$||u(x,t)||_{L^\infty(\Omega)}+||v(x,t)||_{L^\infty(\Omega)}\leq R(||u_0||_{L^\infty(\Omega)},||v_0||_{L^\infty(\Omega)}) \text{ for all } t>0.$$
\end{theorem}
\begin{proof}
    We have already established $\di{||u(x,t)||_{L^\infty(\Omega)}\leq C}$ in Lemma \ref{lem1}($iii$).  However, we have obtained $L_1$ bound for $v(x,t)$ in Lemma \ref{lem1}($ii$). To prove the theorem, we  need to establish $L_\infty$ bound for $v(x,t).$ The detailed steps are described below. 
    
    First, we establish the boundedness of $||\nabla u||_{L^\infty(\Omega)}.$ Let
    $\{e^{t\nabla^2}\}_{t\geq0}$ be the Neumann heat semigroup generated by 
    $-\nabla^2 $ and  $\lambda_1>0$ be the first non-zero eigenvalue of $-\nabla^2 $ in  $\Omega$. Then using the variation of  constants formula for $u,$ we get 
    \begin{equation}\label{voc}
        u(x,t)=e^{t\nabla^2}u_0(x)+\int_0^t e ^{(t-s)\nabla^2}F_1(u(x,s),v(x,s))ds.
    \end{equation}
    Using $L^p$-$L^q$ estimates of Neumann heat semigroup (see Lemma 1.3 in \cite{winkler2010aggregation}) and \eqref{voc}, we obtain
    \begin{equation*}
    \begin{split}
    ||\nabla u(\cdot,t)||_{L^\infty(\Omega)}\leq &||\nabla e^{t\nabla^2}u_0||_{L^\infty(\Omega)}+\int_0^t ||\nabla e ^{(t-s)\nabla^2}F_1(u(\cdot,s),v(\cdot,s))||_{L^\infty(\Omega)}ds
\\&\leq C_1\left(1+t^{-\frac{1}{2}-\frac{n}{2p}}\right)e^{-\lambda_1t}|| u_0||_{L^\infty(\Omega)}+
\\&\int_0^t C_1\left(1+(t-s)^{-\frac{1}{2}-\frac{n}{2p}}\right)e^{-\lambda_1(t-s)}||F_1(u(\cdot,s),v(\cdot,s))||_{L^\infty(\Omega)}\,ds
\\&\leq C_2|| u_0||_{L^\infty(\Omega)}+C_3||F_1(u,v)||_{L^\infty(\Omega)},
\end{split}
\end{equation*}
for some positive constant $C_1,C_2$ and $C_3.$ Since $||F_1(u,v)||_{L^\infty(\Omega)}\leq \frac{1}{4}$, we have $||\nabla u||_{L^\infty(\Omega)}\leq D(|| u_0||_{L^\infty(\Omega)})$ for some positive constant $D(|| u_0||_{L^\infty(\Omega)}).$ 

Next, we establish boundedness of $||v||_{L^q(\Omega)}$ where $q\geq 2.$   Multiplying the second equation of \eqref{pde} by $v^{q-1}$ and integrating over $\Omega,$ we find 
\begin{equation}\label{11}
\int\limits_\Omega v^{q-1}v_t dx -d\int\limits_\Omega v^{q-1}\nabla^2 v dx=c\int\limits_\Omega v^{q-1}\nabla\cdot (v\nabla u)dx+\int\limits_\Omega v^{q-1}F_2(u,v)dx.
\end{equation}
Using Gauss's divergence theorem and no-flux boundary conditions, we arrive at   
$$
    \frac{1}{q}\frac{d}{dt}\int\limits_\Omega v^q dx+d (q-1)\int\limits_\Omega v^{q-2}|\nabla v|^2 dx=\int\limits_\Omega v^{q-1}F_2(u,v)dx -c (q-1) \int\limits_\Omega v^{q-1}\nabla v\cdot\nabla u\; dx.
$$
Since
$$
F_2(u,v)\leq \left(\frac{ea}{2\sqrt{b}}-f\right)v,
$$
we finally find
\begin{equation}\label{121}
    \frac{1}{q}\frac{d}{dt}\int\limits_\Omega v^q dx+d (q-1)\int\limits_\Omega v^{q-2}|\nabla v|^2 dx
    \leq \left(\frac{ea}{2\sqrt{b}}-f\right)
    \int\limits_\Omega v^q dx -c (q-1) \int\limits_\Omega v^{q-1}\nabla v\cdot\nabla u\; dx.
\end{equation}
Note that the inequality $\di{\frac{|\vec a|^2}{2\varepsilon}+\frac{\varepsilon |\vec b|^2}{2}\ge \vec a\cdot\vec b}$ holds for any 
 $\varepsilon>0$. Considering $\vec a=v^{\frac{q-2}{2}} \nabla v$, $\vec b=c v^{\frac{q}{2}}\nabla u$ and $\varepsilon=\frac{1}{d},$ we get
$$
\frac{d v^{q-2}|\nabla v|^2}{2}+\frac{c^2 v^q |\nabla u|^2}{2d}\ge c v^{q-1}\nabla v\cdot\nabla u .$$
Using this inequality and the estimate $||\nabla u||_{L^\infty(\Omega)}\leq D(|| u_0||_{L^\infty(\Omega)})$ (derived before) in \eqref{121}, we get
\begin{eqnarray}
    \frac{1}{q}\frac{d}{dt}\int\limits_\Omega v^q dx+f\int\limits_\Omega v^q dx+\frac{d (q-1)}{2}\int\limits_\Omega v^{q-2}|\nabla v|^2 dx&\di{\leq \frac{ea}{2\sqrt{b}}\int\limits_\Omega v^q dx+\frac{c^2(q-1) }{2d} \int\limits_\Omega v^q |\nabla u|^2 dx} \nonumber
    \\& \di{\leq \left(\frac{ea}{2\sqrt{b}}+D^2\right) \int\limits_\Omega v^q dx\equiv K \int\limits_\Omega v^q dx.}\label{simp}
\end{eqnarray}
For any $\varepsilon_1>0,$ the inequality $\di{\int\limits_\Omega v^q dx\leq\varepsilon_1||\nabla v^{\frac{q}{2}}||_2^2+M}$ holds for some $M>0$ (see the Appendix A). Using this in \eqref{simp}, we get
\begin{equation*}
    \frac{d}{dt}\int\limits_\Omega v^q dx+fq\int\limits_\Omega v^q dx +\frac{q d (q-1)}{2}\int\limits_\Omega v^{q-2}|\nabla v|^2 dx\leq q K(\varepsilon_1||\nabla v^{\frac{q}{2}}||_2^2+M)
\end{equation*}
Using
$$\di{\int\limits_\Omega v^{q-2}|\nabla v|^2 dx=\frac{4}{q^2}||\nabla v^{\frac{q}{2}}||_2^2}
$$ 
in the  above, we get
\begin{equation*}
    \frac{d}{dt}\int\limits_\Omega v^q dx+fq\int\limits_\Omega v^q dx \leq q K\left(\varepsilon_1||\nabla v^{\frac{q}{2}}||_2^2+M - \frac{2d (q-1)}{Kq^2}\nabla v^{\frac{q}{2}}||_2^2 \right)
\end{equation*}
Setting  $\di{\varepsilon_1=\frac{2d(q-1)}{Kq^2}},$ we have $$\frac{d}{dt}\int\limits_\Omega v^q dx+fq\int\limits_\Omega v^q dx\leq MqK,
$$ 
which leads to 
$$
\di{ ||v||_{L^q(\Omega)}\leq  \left( ||v_0||^q_{L^q(\Omega)}+\left(\frac{MK}{f}\right)\right)^{1/q}\equiv N}.
$$
Now using  Sobolev embedding \cite{tao2013competing} and  Moser–Alikakos iteration procedure \cite{tao2011boundedness}, there exists a constant $P(|| u_0||_{L^\infty(\Omega)},||v_0||_{L^\infty(\Omega)})>0$ such that 
$$\di{||v(.,t)||_{L^\infty(\Omega)}\leq P(|| u_0||_{L^\infty(\Omega)},||v_0||_{L^\infty(\Omega)}).}$$ The proof is now complete.
\end{proof}
\subsection{Homogeneous steady-state analysis}
Here, we discuss the stability of the homogeneous steady states. Note that a homogeneous steady-state $E(\Tilde{u},\Tilde{v})$ of the system \eqref{10} corresponds to an equilibrium point $(\tilde u,\tilde v)$ of the temporal system \eqref{ode}. 
Introducing $u(x,t)=\tilde u+\bar u(x,t)$ and $v(x,t)=\tilde v +\bar v(x,t),$ and linearizing about $E(\Tilde{u},\Tilde{v}),$  we find
$$ \frac{\partial \bar U}{\partial t}=\mathcal{A}(E)\nabla^2 \bar U+J(E) \bar U\equiv \mathcal{L}(E)\bar U,$$
where $\bar U\equiv (\bar u(x,t),\bar v(x,t))\in [C(\bar\Omega\times [0,\infty))\cap C^{2,1}(\bar\Omega\times (0,\infty))]^2.$ Further, $J(E)$ and $\mathcal{A}(E)$ have been defined in \eqref{jaco}  and  \eqref{Au} respectively. 

Consider the eigenvalue problem \begin{eqnarray}
\begin{split}
     -\nabla^2 p&= k p \quad  \mbox{in} \;\Omega,\\
 \frac{\partial p}{\partial n}&=0\quad\;\mbox{on}\;  \partial \Omega.
\end{split} \label{evpr}
 \end{eqnarray}

For an eigenvalue $k_i$ of the eigenvalue problem \eqref{evpr}, let $E(k_i)$ be the corresponding eigenfunction space  with $0=k_0<k_1<\cdots<k_i<\cdots$.  Further, assume that $\{\phi_{i,j}: j=1,\cdots,\text{dim}\big(E(k_i)\big)\}$ be an orthogonal basis set of $E(k_i)$ and $\mathcal{U}_{ij}=\{c \phi_{i,j}: c=(c_1,c_2)^T \}$. Let  $\mathcal{U}_i=\bigoplus_{j=1}^{\mbox{dim}(E(k_i))} \mathcal{U}_{ij}$  be the direct sum of $\mathcal{U}_{ij}$. 
  It can be shown that  \begin{equation}
    \mathcal{U}\equiv\left\{\left(\phi,\psi\right)^T \in C^2(\bar \Omega)\times  C^2(\bar \Omega): \frac{\partial \phi}{\partial n}=\frac{\partial \psi}{\partial n}=0\; \text{ for } x\in \partial\Omega\right\}=\bigoplus_{i=1}^\infty \mathcal{U}_i,\label{defw}
  \end{equation}
  and $\mathcal{U}_i$ is invariant under the operator $\mathcal{L}$. Now, $\lambda$ is an eigenvalue of $\mathcal{L}$ if and only if $\lambda$ is an eigenvalue of the matrix $\mathcal{L}_i=-k_i\mathcal{A}(E)+J(E)$  for some $i\geq0.$  The characteristic equation of $\mathcal{L}_i$ is given by 
\begin{equation}
        \lambda^2-\mathrm{T}(k_i)\lambda+\mathrm{H}(k_i)=0,
        \label{chracteristic}
            \end{equation}
    where 
    \begin{alignat*}{4}
    \mathrm{T}(k_i)&=a_{10}+b_{01}-(1+d)k_i,\\
    \mathrm{H}(k_i)&=dk_i^2-\{a_{10}d+b_{01}-c\Tilde{v}a_{01}\}k_i+(a_{10}b_{01}-a_{01}b_{10}).
    \end{alignat*}
    Now, from  the principle of linearized stability of parabolic partial differential equation \cite{drangeid1989principle}, the homogeneous steady state $E$ is said to be locally asymptotically stable if $T(k_i)<0,$ and $H(k_i)>0$ for all $i.$  
    \begin{theorem}\label{Th3}
The followings hold for the system \eqref{10}: 
\begin{itemize}
    \item[(a)]  Trivial homogeneous steady-state $E_0(0,0)$ is always unstable.
    \item[(b)] Axial homogeneous steady-state $E_1(1,0)$ is locally  asymptotically  stable if $f>f_{TC}$ and  unstable if $f<f_{TC}$.
     \item[(c)]  Coexisting homogeneous steady state $E_1^*(u_1^*,v_1^*)$ is locally asymptotically  stable if \\$\di{a_{10}<\text{min}\left\{v_1^*,{\frac { c u_1^*\left( {u_1^*}-1 \right) +v_1^*}{d}}\right\}\equiv R_1}$ and unstable if $a_{10}>v_1^*.$ 
\end{itemize}
\end{theorem}
\begin{proof}
 $(a)$ For $E_0$, we find $\mathrm{T}(k_i)=1-f-(1+d)k_i$ and $\mathrm{H}(k_i)=dk_i^2+(d+f)k_i-f.$ Since $\mathrm{H}(k_0)<0,$ $E_0$ is unstable.
    
$(b)$ For $E_1,$ we have 
$$\mathrm{T}(k_i)=-1+\left(f_{TC} -f\right)-(1+d)k_i
$$ 
and 
$$\mathrm{H}(k_i) = dk_i^2+\left(d-\left(f_{TC} -f\right)\right)k_i-\left(f_{TC} -f\right).
$$ 
If $f>f_{TC},$ then $\di{\mathrm{T}(k_i)<0 \text{ and } \mathrm{H}(k_i)>0}$ for all $i.$ Thus, $E_1$ is locally  asymptotically  stable if $f>f_{TC}$. If $f<f_{TC},$ then $\di{ \mathrm{H}(k_0)<0}$ and hence $E_1$ is unstable.\\

$(c)$ If  $a_{10}>-b_{01}=v_1^*,$ then $E_1^*$ is unstable since $\di{\mathrm{T}(k_0)>0}.$  Now if $a_{10}<R_1,$ then we find $a_{10}+b_{01}<0$ and $a_{10}d+b_{01}-cv_1^*a_{01}<0$. These conditions lead to 
   \begin{equation}\label{las}
       \mathrm{T}(k_i)<0 \text{ and }\mathrm{H}(k_i)>0 \text{ for all } i.
   \end{equation}
Hence, $E_1^*$ is locally asymptotically stable if $a_{10}<R_1.$
  \end{proof}
   \begin{remark}
Whenever homogeneous steady state  $E_2^*$ exists, it is unstable since $\mathrm{H}(k_0)<0.$ In case of $E_3^*$ (when it exists), it is  locally asymptotically stable if $\mathrm{T}(k_i)<0$ and $\mathrm{H}(k_i)>0$ for all $i.$ It is unstable if either of these conditions does not hold for some $i$. 
\end{remark}
\subsection{Turing instability}\label{Turingsubsection}
In the case of Turing instability, the coexisting equilibrium point is  asymptotically  stable for the temporal system \eqref{ode}. But the corresponding homogeneous steady state of the spatio-temporal model \eqref{pde} becomes unstable under spatial perturbation. Thus,     
$ \mathrm{T}(k_i)<0 \text{ and }\mathrm{H}(k_i)>0 \text{ for } i=0,$ but at least one of these conditions is violated for some $i\ge 1$. Note that $\di{\mathrm{T}(k_i)<\mathrm{T}(k_0)}$ for all $i\geq 1.$  Hence, for Turing instability, there exists some $k_i$ ($i\ge 1$) for which $\mathrm{H}(k_i)<0.$ The eigenvalues $k_i$ are discrete for a bounded $\Omega.$ To find the Turing bifurcation threshold, we consider an unbounded domain for which a continuous spectrum of eigenvalues is obtained. Hence, the Turing instability condition becomes $\mathrm{H}(k)<0$ for some $k\ne 0$. Here, we consider the homogeneous steady state corresponding to the coexisting equilibrium point $E_1^*(u_1^*,v_1^*).$ 

The minimum value of $\mathrm{H}(k)$ is 
$$\di{\mathrm{H}_{min}={\frac {-{{\it a_{10}}}^{2}{d}^{2}+ \left(  \left( 2 {\it a_{10}} c v_1^*-
4 {\it b_{10}} \right) {\it a_{01}}+2 {\it a_{10}} {\it b_{01}} \right) d-
 \left( {\it a_{01}} c v_1^*-{\it b_{01}} \right) ^{2}}{4d}},}
 $$ which occurs at wavenumber   $$\di{k_{min}=\,{\frac {{\it a_{10}}\,d+{\it b_{01}}-{\it a_{01}}\,cv_1^*}{2d}}}.
 $$
Now, at the threshold of the Turing bifurcation, we have $\mathrm{H}_{min}=0$ at critical wavenumber $k_T=k_{min},$ and the corresponding critical Turing value  $c_T$ of $c$ is determined as follows  
\begin{equation}\label{ct}
    c_T=\frac{ a_{10} d+ b_{01}-2\sqrt {d\, \text{det}(J)}}{ a_{01}  v_1^*}.
\end{equation}
The system becomes Turing unstable for $c>c_T$. We plot the Turing surface boundary in the $f$-$d$-$c$ space for parameter values $a=7$, $b=5.65$, and  $e=0.95$ (see Fig. \ref{TuringCurve}). The  Turing instability  sets in above the Turing boundary surface. In front of the Hopf surface, we have an additional  Hopf instability. These two instabilities, divide the $f$-$c$-$d$ space into four regions: (i) a stable region below  the Turing boundary surface and behind the Hopf plane, (ii)  a Turing region above  the Turing boundary surface and behind the Hopf plane,  (iii) a Hopf region below  the Turing boundary surface and in front of the Hopf plane, and (iv) a Turing-Hopf region above  the Turing boundary surface and in front of the Hopf plane.  We will discuss the spatio-temporal dynamics in these regions in the 
 section \ref{numerical}.
\begin{figure}[!ht]
\centering
\includegraphics[scale=0.6]{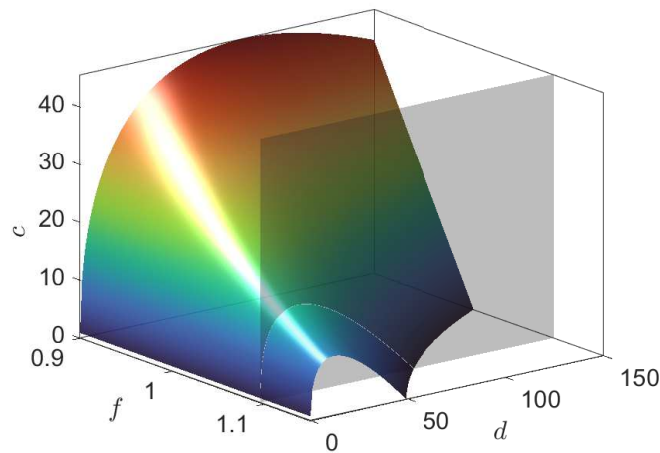}\\
\caption{Turing boundary surface (coloured surface) and Hopf plane (grey vertical plane) in the $f$-$d$-$c$ space. Note that the system \eqref{pde} exhibits Turing instability  above the coloured  surface and Hopf instability in front of the grey plane. The values of the other parameters are  $a=7$, $b=5.65$, and  $e=0.95$. }
\label{TuringCurve}
\end{figure}

\section{Weakly nonlinear analysis}\label{wna}
Here, we perform weakly nonlinear analysis using the method of multiple scales near the Turing bifurcation threshold $(c=c_T)$. We derive the amplitude equations for the Turing solutions, which also shows the existence of non-homogeneous stationary solutions near the Turing thresholds. For simplicity, we consider one-dimensional spatial domain as $\Omega=[0,L]\in \mathbb{R}$.   First,  we employ a Taylor series expansion up to the third order to expand the system (\ref{pde}) around $E_*^1(u_*^1,v_*^1)$ 
    \begin{eqnarray}
\frac{\partial \boldsymbol{W}}{\partial t}=\mathscr{L}^c\boldsymbol{W}+\mathscr{N}+\frac{1}{2}\mathscr{B}(\boldsymbol{W},\boldsymbol{W})+\frac{1}{6}\mathscr{T}(\boldsymbol{W},\boldsymbol{W},\boldsymbol{W})\equiv \mathcal{Z}(\boldsymbol{W}), 
\label{12}
\end{eqnarray}
where 
$$
\boldsymbol{W}= \begin{pmatrix}u-u_*^1\\v-v_*^1 \end{pmatrix},\quad \mathcal{L}^c= \begin{pmatrix}
a_{10}+\partial_{xx}  & a_{01}  \\
b_{10}+cv_1^*\partial_{xx} & b_{01}+d \partial_{xx} 
\end{pmatrix},\quad\mbox{and}\quad \mathscr{N}=  \begin{pmatrix}0\\ c(v\nabla^2 u+\nabla u\cdot \nabla v)\end{pmatrix}
$$
denote the perturbation vector, linear operator, and nonlinear prey-taxis term, respectively. Further, the bi-linear operator $\mathscr{B}$ and tri-linear operator $\mathscr{T}$ are defined as \begin{equation*}
    \mathscr{B}(\boldsymbol{P_1},\boldsymbol{P_2})=\begin{pmatrix}
	\di{\sum\limits_{i,j\in \{ 1,2\}} \frac{\partial^2 F_1}{\partial  y_i \partial  y_j }|_{(u^1_*,v^1_*)} P_1^{(i)}P_2^{(j)}} \\ \di{ \sum\limits_{i,j\in \{ 1,2\}} \frac{\partial^2 F_2}{\partial  y_i \partial  y_j }|_{(u^1_*,v^1_*)} P_1^{(i)}P_2^{(j)}}
\end{pmatrix},   \end{equation*}
\begin{equation*}
    \mathscr{T}(\boldsymbol{P_1},\boldsymbol{P_2},\boldsymbol{P_3})= \begin{pmatrix}
	\di{\sum\limits_{i,j,k\in \{ 1,2\}} \frac{\partial^2 F_1}{\partial  y_i \partial  y_j \partial  y_k }|_{(u^1_*,v^1_*)} P_1^{(i)}P_2^{(j)} P_3^{(k)}} \\ \di{ \sum\limits_{i,j,k\in \{ 1,2\}} \frac{\partial^2 F_2}{\partial  y_i \partial  y_j \partial  y_k }|_{(u^1_*,v^1_*)} P_1^{(i)}P_2^{(j)} P_3^{(k)}}
\end{pmatrix},
\end{equation*}
where $\boldsymbol{P}_m= \begin{pmatrix}
 P_m^{(1)}\\P_m^{(2)}   
\end{pmatrix}$, for $m=1,2,3$;  $y_1=u$ and $y_2=v.$ 

We introduce slow time scale $\tau=\varepsilon^2 t,$ where  $\varepsilon$  measures the distance between prey-taxis parameter $c$ and the critical Turing  threshold $c_T$ as $c=c_T+\varepsilon^2 c_{2}$. We also expand the solution of the system \eqref{12} in terms of $\varepsilon$ as $$\boldsymbol{W}=\varepsilon \boldsymbol{W_1}+\varepsilon^2 \boldsymbol{W_2}+\varepsilon^3 \boldsymbol{W_3}+O(\varepsilon^4).$$ 
Accordingly, the operator  $\mathscr{L}^c$, $\mathscr{B}$ and $\mathscr{T}$ in \eqref{12} can be expressed as 
\begin{alignat*}{4} &\mathscr{L}^c =\mathscr{L}^{c_T}+\varepsilon^2 c_{2} v_1^*\begin{pmatrix}
    0 & 0\\1 &0
\end{pmatrix}\nabla^2+O(\varepsilon^4),\\ &
\mathscr{B}(\boldsymbol{W},\boldsymbol{W})=\varepsilon^2\mathscr{B}(\boldsymbol{W_1},\boldsymbol{W_1})+2\varepsilon^3 \mathscr{B} (\boldsymbol{W_1}, \boldsymbol{W_2})+ O(\varepsilon^4),\\ &\mathscr{T}(\boldsymbol{W},\boldsymbol{W},\boldsymbol{W})=\varepsilon^3\mathscr{T}(\boldsymbol{W_1},\boldsymbol{W_1},\boldsymbol{W_1})+ O(\varepsilon^4).
\end{alignat*}
Substituting all the above expansions into \eqref{12}, we obtain a series of equations in $\boldsymbol{W_i}\, (i=1,2,\cdots)$ by collecting the terms at each order in $\varepsilon.$ Each of these equations is subjected to  the Neumann boundary condition.

\noindent 
At $O(\varepsilon),$ we have $\mathscr{L}^{c_T}\boldsymbol{W_1}=0,$ whose solution is $$\boldsymbol{W_1}=A(\tau)\boldsymbol{\Phi}\; \text{cos}(k_T x)\; \text{ with } \boldsymbol{\Phi}\in \text{Ker}(J(E_1^*)-k_T^2 D^{c_T}).$$
Here, $A(\tau)$ is the amplitude of the growing pattern, unknown at the moment. We normalize $\Phi$ vector as $$\boldsymbol{\Phi}=\begin{pmatrix}
    1\\ \phi
\end{pmatrix} \text{ with } \phi=\frac{k_T^2-a_{10}}{a_{01}}.$$
At $O(\varepsilon^2)$, we obtain 
\begin{equation} \label{13}
    \mathscr{L}^{c_T}\boldsymbol{W_2}=(\boldsymbol{h_{20}}+\boldsymbol{h_{22}}\text{ cos}(2k_Tx))A^2\equiv \boldsymbol{H},
\end{equation}
 $\di{ \text{ where } \boldsymbol{h_{20}}=-\frac{1}{4}\mathscr{B}(\boldsymbol{\Phi},\boldsymbol{\Phi}),\; \boldsymbol{h_{22}}=-\frac{1}{4}\mathscr{B}(\boldsymbol{\Phi},\boldsymbol{\Phi})+c_Tk_T^2\begin{pmatrix}
    0\\ \phi
\end{pmatrix}}.$ 
Using Fredholm alternative theorem, the system \eqref{13} has a solution  when $\di{<\boldsymbol{H},\boldsymbol{\Psi}>\equiv \int\limits_0^L \boldsymbol{H}\, \boldsymbol{\Psi} \; dx =0},$ where $\boldsymbol{\Psi}=\begin{pmatrix}
    \psi\\ 1
\end{pmatrix} \text{ cos}(k_Tx) \in \text{Ker}( \mathscr{L}^{c_T})^\dagger$ with $\psi=\frac{dk_T^2-b_{01}}{a_{01}}.$ Note that $A^\dagger$ denotes the adjoint of the operator $A.$ Note that here the Fredholm alternative theorem automatically  holds. 
The solution of \eqref{13} is given by $$\boldsymbol{W}_2=(\boldsymbol{k_{20}}+\boldsymbol{k_{22}}\text{ cos}(2k_Tx))A^2,$$  where $J(E_1^*)\boldsymbol{k_{20}}=\boldsymbol{h_{20}}$ and $\left(J(E_1^*)-4k_T^2 D^{c_T}\right)\boldsymbol{k_{22}}=\boldsymbol{h_{22}}.$

At $O(\varepsilon^3)$, we find 
\begin{equation} \label{14}    \mathscr{L}^{c_T}\boldsymbol{W_3}=\left(\frac{d A }{d \tau} \boldsymbol{\Phi} +\boldsymbol{g_{11}}A +\boldsymbol{g_{31}}A^3\right)\text{ cos}(k_Tx)+\boldsymbol{g_{33}}\text{ cos}(3k_Tx)A^3\equiv \boldsymbol{G},
\end{equation}
where $$\boldsymbol{g_{11}}=c_{2}k_T^2 v_1^*\begin{pmatrix}
    0\\1
\end{pmatrix},\; \boldsymbol{g_{31}}=-\mathscr{B}(\boldsymbol{\Phi},\boldsymbol{k}_{20})-\frac{1}{2}\mathscr{B}(\boldsymbol{\Phi},\boldsymbol{k}_{22})-\frac{1}{8}\mathscr{T}(\boldsymbol{\Phi},\boldsymbol{\Phi},\boldsymbol{\Phi})+c_Tk_T^2\begin{pmatrix}
    0\\k_{20}^{(1)}\phi+k_{22}^{(1)}-\frac{1}{2}k_{22}^{(2)}
\end{pmatrix}, $$ $\text{ and }\di{\boldsymbol{g_{33}}=-\frac{1}{2}\mathscr{B}(\boldsymbol{\Phi},\boldsymbol{k}_{22})-\frac{1}{24}\mathscr{T}(\boldsymbol{\Phi},\boldsymbol{\Phi},\boldsymbol{\Phi})+c_Tk_T^2\begin{pmatrix}
    0\\3k_{22}^{(1)}\phi+\frac{3}{2}k_{22}^{(2)}
\end{pmatrix}.}$

\noindent Using Fredholm alternative theorem in  equation (\ref{14}), we obtain $\di{<\boldsymbol{G},\boldsymbol{\Psi}> =0},$ which leads to the cubic  Stuart-Landau equation  
\begin{equation}
    \frac{d A}{d \tau } = {\sigma} A -  {l} A^3, 
    \label{Amp3}
\end{equation}
where $$   {\sigma}=-\frac{<\boldsymbol{g}_{11}, \boldsymbol{\Psi} >}{< \boldsymbol{\Phi}, \boldsymbol{\Psi}>} \;\; \text{and}\; \;  {l} = \frac{<\boldsymbol{g}_{31}, \boldsymbol{\Psi} >}{< \boldsymbol{\Phi}, \boldsymbol{\Psi}>}.$$ 
Here, $\sigma$ is always positive 
and (\ref{Amp3}) is similar to normal form of pitchfork bifurcation.  
\subsection*{Supercritical case}
If $l>0$, then the system (\ref{Amp3}) has  two stable fixed points $A_*=\pm \sqrt{\frac{\sigma}{l}}$ and an unstable fixed point $A_*=0.$ 
These stable fixed points represent the Turing pattern with wavenumber $k_T$ for $c>c_T$ and all non-zero solutions approach them after a long time, i.e.,     $$\lim_{{t} \to {\infty}} A(t)\equiv A_{\infty} =\pm \sqrt{\frac{\sigma}{l}}.$$ 
Therefore, the final solution of (\ref{12}) is $\boldsymbol{U}_\infty(x)=\sqrt{c-c_T} A_{\infty}\Phi\cos(k_T  x)+{O}(\varepsilon^2)$ for $c>c_T$.

\subsection*{Subcritical case}
For $l<0,$ the system (\ref{Amp3}) does not admit any nontrivial equilibrium solution. Here, we need to include higher order terms in the expansion. Thus,   we employ a Taylor series expansion up to the fifth order to expand the system (\ref{pde}) around $E_*^1(u_*^1,v_*^1)$
\begin{equation}
  \frac{\partial \boldsymbol{W}}{\partial t}=\mathcal{Z}(  \boldsymbol{W})+\frac{1}{24}\mathscr{Q}(\boldsymbol{W},\boldsymbol{W},\boldsymbol{W},\boldsymbol{W})+\frac{1}{60}\mathscr{P}(\boldsymbol{W},\boldsymbol{W},\boldsymbol{W},\boldsymbol{W},\boldsymbol{W}),
\end{equation}
where $\mathscr{Q}$ and  $\mathscr{P}$ are the quad-linear operator  and penta-linear operator  defined similar to $\mathscr{B}$ and $\mathscr{T}.$  We also modify the multiple time scales $t=t(\tau,\tau_1)$, where $\tau=\epsilon^2 t$ and $\tau_1=\epsilon^4 t$; $c$ and $\boldsymbol W$ up to  the fifth order of $\varepsilon$ as  
\begin{alignat*}{4} 
&c=c_T+\varepsilon^2 c_2+\varepsilon^4 c_4+O(\varepsilon^6),\\ &\boldsymbol{W}=\varepsilon \boldsymbol{W_1}+\varepsilon^2 \boldsymbol{W_2}+\varepsilon^3 \boldsymbol{W_3}+\varepsilon^4 \boldsymbol{W_4}+\varepsilon^5 \boldsymbol{W_5}+O(\varepsilon^6).
\end{alignat*}
The right hand side of  \eqref{14} using \eqref{Amp3} becomes
$$
G=\left(\boldsymbol{G_{11}}A +\boldsymbol{G_{13}}A^3\right)\text{ cos}(k_Tx)+\boldsymbol{G_{33}}\text{ cos}(3k_Tx)A^3,$$
where $\boldsymbol{G_{11}}=\boldsymbol{g_{11}}+\sigma A \boldsymbol{\Phi}, $   $\boldsymbol{G_{13}}=\boldsymbol{g_{13}}-l A^3 \boldsymbol{\Phi}, $ and 
 $\boldsymbol{G_{33}}=\boldsymbol{g_{33}}. $ Note that the amplitude now becomes $A=A(\tau,\tau_1).$
The solution of \eqref{14} is given by $$\boldsymbol{W}_3=(\boldsymbol{C_{11}}+\boldsymbol{C_{13}}A^2)A\text{ cos}(k_Tx)+\boldsymbol{C_{33}}\text{ cos}(3k_Tx)A^3,$$  where $\left(J(E_1^*)-k_T^2 D^{c_T}\right)\boldsymbol{C_{11}}=\boldsymbol{G_{11}},$
 $\left(J(E_1^*)-k_T^2 D^{c_T}\right)\boldsymbol{C_{13}}=\boldsymbol{G_{13}},$ and $\left(J(E_1^*)-9k_T^2 D^{c_T}\right)\boldsymbol{C_{33}}=\boldsymbol{G_{33}}.$ 

 At $O(\varepsilon^4)$, we obtain  
\begin{equation} \label{17}    \mathscr{L}^{c_T}\boldsymbol{W_4}=\boldsymbol{H}_{20}A^2+\boldsymbol{H}_{22}A^2\text{ cos}(2k_Tx)+\boldsymbol{H}_{40}A^4+\boldsymbol{H}_{42}A^4\text{ cos}(2k_Tx)A^2+  \boldsymbol{H}_{44}A^4\text{ cos}(4k_Tx)\equiv \boldsymbol{H},
\end{equation}
Again the Fredholm
alternative condition automatically holds due to the choice of our perturbation. Further solving \eqref{17}, we find $$\boldsymbol{W}_4=\boldsymbol{D}_{20}A^2+\boldsymbol{D}_{22}A^2\text{ cos}(2k_Tx)+\boldsymbol{D}_{40}A^4+\boldsymbol{D}_{42}A^4\text{ cos}(2k_Tx)A^2+  \boldsymbol{D}_{44}A^4\text{ cos}(4k_Tx),$$  where $J(E_1^*)\boldsymbol{D_{20}}=\boldsymbol{H_{20}},$
 $\left(J(E_1^*)-4k_T^2 D^{c_T}\right)\boldsymbol{D_{22}}=\boldsymbol{H_{22}},$ $J(E_1^*)\boldsymbol{D_{40}}=\boldsymbol{H_{40}},$ $\left(J(E_1^*)-4k_T^2 D^{c_T}\right)\boldsymbol{D_{42}}=\boldsymbol{H_{42}},$  and $\left(J(E_1^*)-16k_T^2 D^{c_T}\right)\boldsymbol{D_{44}}=\boldsymbol{H_{44}}.$
 
Finally, at $O(\varepsilon^5)$, we obtain  
\begin{equation} \label{18}    \mathscr{L}^{c_T}\boldsymbol{W_5}=\frac{\partial A}{\partial \tau_1}\boldsymbol{\Phi}+\boldsymbol{I}_{11}A \text{ cos}(k_Tx)+\boldsymbol{I}_{31}A^3 \text{ cos}(k_Tx)+\boldsymbol{I}_{51}A^5 \text{ cos}(k_Tx)+\boldsymbol{I}^*\equiv \boldsymbol{I},
\end{equation}
where $\boldsymbol{I}^*$ contains the orthogonal terms of $\boldsymbol{W}_1$ in $\boldsymbol{I}$. Further, $\boldsymbol{H}_{ij}$ and $\boldsymbol{I}_{ij}$ are explicitly computed in terms of the system parameters. As these expressions are quite involved, we omit them for the sake of brevity. 

 Using the Fredholm solvability condition once more in \eqref{18}, we arrive at a quintic Stuart-Landau equation 
 \begin{equation} \label{amp5}
     \frac{\partial A}{\partial \tau_1}=\sigma'A-l'A^3+\rho' A^5,
 \end{equation}
 where  $$ \sigma'=-\frac{<\boldsymbol{I}_{11}, \bar{\Phi} >}{< \Phi, \bar{\Phi}>},\;  l' = \frac{<\boldsymbol{I}_{31}, \bar{\Phi} >}{< \Phi,\bar{\Phi}>} \;\; \text{and}\; \;  \rho'=-\frac{<\boldsymbol{I}_{51}, \bar{\Phi} >}{< \Phi, \bar{\Phi}>}.$$ 
Adding \eqref{Amp3} and \eqref{amp5}, we finally obtain  
\begin{equation} \label{AMP35}
    \frac{d A}{d t } =\varepsilon^2(\hat{\sigma}A - \hat{l} A^3 +\hat{\rho} A^5),
\end{equation}
where $ \hat{\sigma}={\sigma} +\epsilon^2 \sigma' $,$\; \hat{l}= {l}+\epsilon^2 l' $ and $\hat{\rho}=\epsilon^2 \rho'$. 
If $\hat{\sigma}>0,\;\hat{l}<0$ and $ \hat{r}<0,$ then equation (\ref{AMP35}) admits two stable equilibria $\,\pm  \sqrt{\frac{\hat{l}-\sqrt{\hat{l}^2-4\hat{r}\hat{\sigma}}}{2\hat{r}}}.$ Thus, we obtain the amplitude  of the stationary pattern  solution for $c>c_T.$ 


\section{Numerical results} \label{numerical}
Here, we first validate the results of weakly nonlinear analysis (WNA) and then discuss the effects of prey-taxis parameter $c$ on the Turing solutions. Next we investigate spatio-temporal transient dynamics and non-homogeneous oscillatory solution in the Hopf region. 
\subsection{Validation of WNA results}
We take the parameter values $a=7$, $b=5.65$,  $e=0.95,$  $f=0.98$, and $d=80$. For this parameter set, the coexisting homogeneous steady state $E_1^*(0.210978,0.141065)$ is asymptotically stable under homogeneous perturbation. The corresponding Turing threshold is $c_T=26.889081$ 
with critical wave number $k_T=0.283128.$ Using WNA, we find the cubic  Stuart-Landau equation  
  \begin{equation}\label{3wna}
        \frac{d A}{d \tau } = 0.054965 A +2.925863 A^3,
  \end{equation}
  which shows the subcritical nature of the bifurcation. Extending WNA up to the fifth order, we obtain quintic Stuart-Landau equation \begin{equation} \label{5wna}
         \frac{\partial A}{\partial \tau_1}=-0.000463A+6.926411 A^3-43.690556 A^5.
    \end{equation}
    Using \eqref{3wna} and \eqref{5wna}, we obtain the amplitude equation \begin{equation}   \label{AMP}  \frac{d A}{d t } =\varepsilon^2( \hat{\sigma}A - \hat{l} A^3 +\hat{\rho} A^5),
    \end{equation}
    where $\hat{\sigma}=0.054965-0.000463 \varepsilon^2,\;\hat{l}=-(2.925863+6.926411\varepsilon^2),$ and $\hat{\rho}=-43.690556 \varepsilon^2.$ The equilibrium amplitude $A_\infty$ of the Turing  pattern solution at $c=c_T(1+\varepsilon^2)$ satisfies $\hat{\sigma} - \hat{l} A^2 +\hat{\rho} A^4=0,$ and \textcolor{black}{the stationary} solution of $u$ up to third order is 
\begin{equation}
    \begin{split}
        u_\infty(x)=&u_1^*+\varepsilon A_\infty \text{ cos}(k_Tx)+\varepsilon^2 A^2_\infty (1.8+1.811195\text{ cos}(k_Tx))\\&+\varepsilon^3 A_\infty ((1+A_\infty^2)\text{ cos}(k_Tx)+1.969026 \text{ cos}(3k_Tx)). 
    \end{split}
\end{equation}
\begin{figure}[!b]
\centering
\includegraphics[scale=0.6]{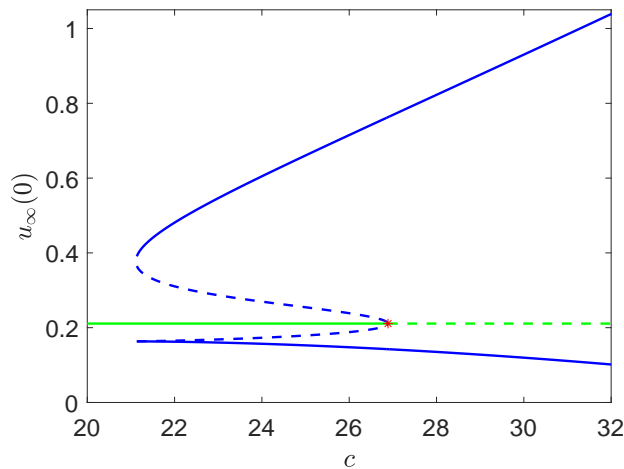}\\
\caption{Plot of the solution up to third order from the weakly nonlinear analysis. Here green and blue color curves denote homogeneous stationary solution and pattern solution, respectively. The solid and dashed curves respectively correspond to the stable and unstable branches.  Other parameter values are  $a=7$, $b=5.65$, $f=0.98$, $e=0.95$, $d=80$  and $k_T=0.283128$.}
\label{wnabif}
\end{figure}

\begin{figure}[!h]
\centering
\begin{subfigure}[b]{.48\textwidth}
 \centering
\includegraphics[scale=0.55]{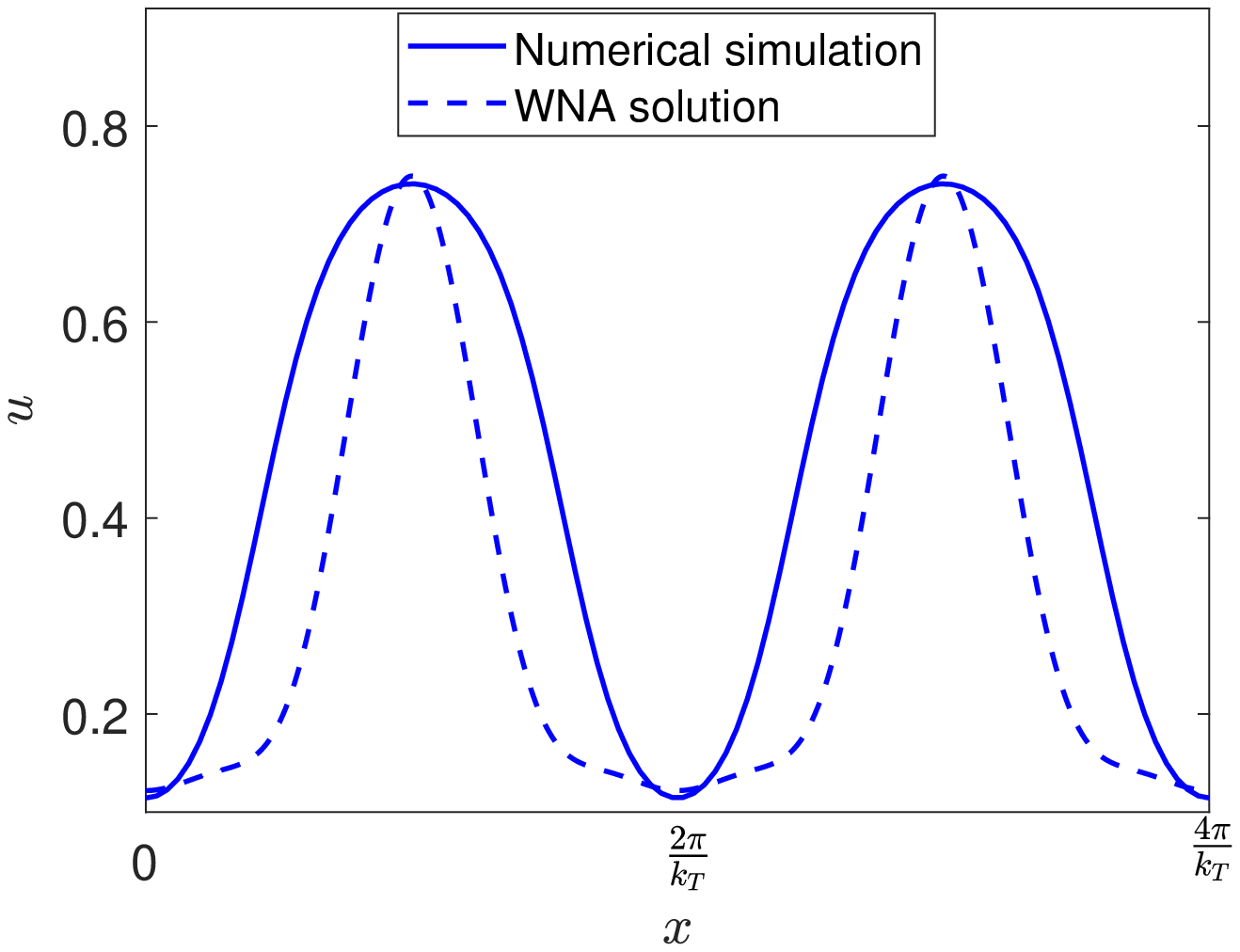}\\ 
 \caption{}
  \end{subfigure}
 \begin{subfigure}[b]{.48\textwidth}
 \centering
\includegraphics[scale=0.55]{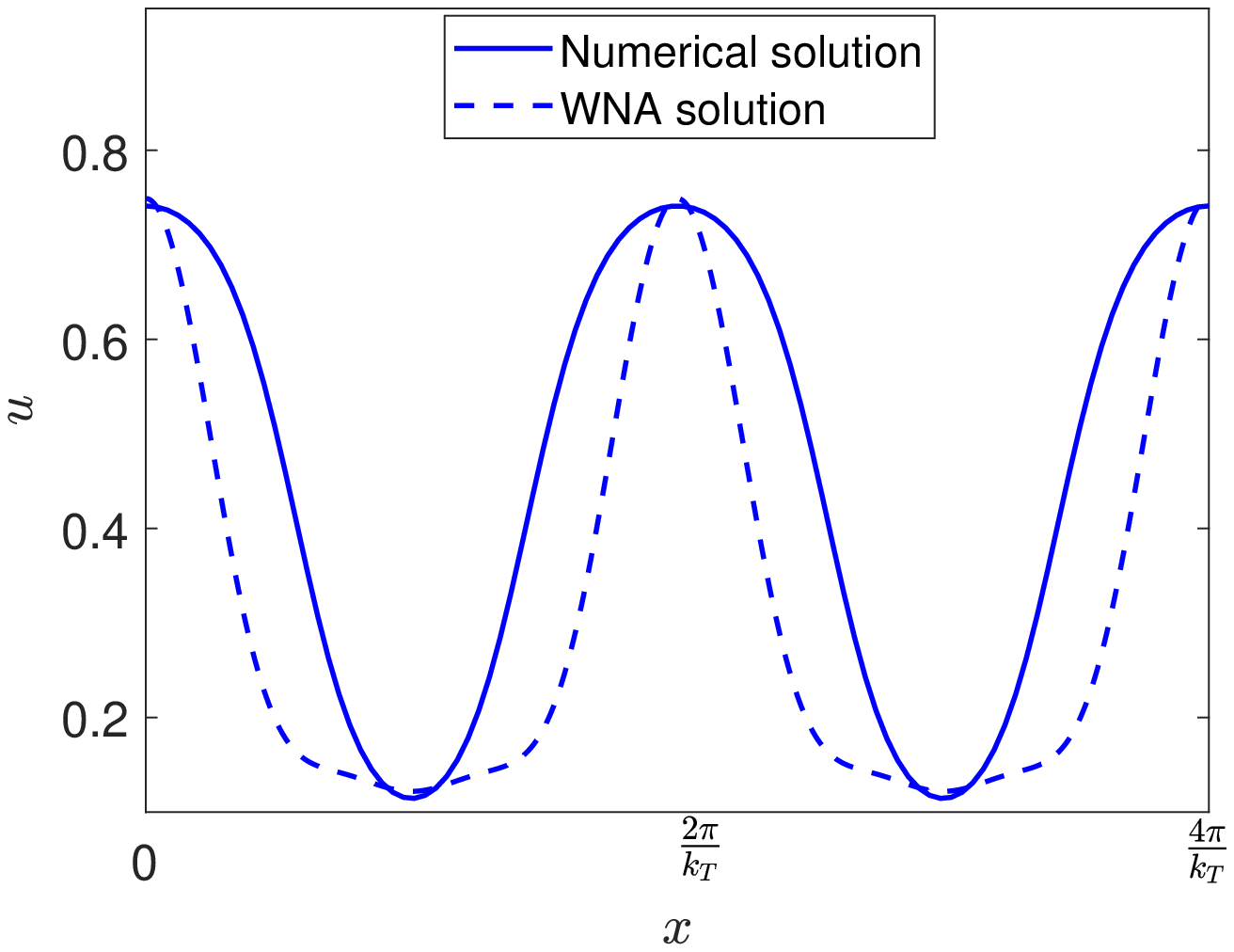}\\ 
 \caption{}
  \end{subfigure} 
\caption{Comparison of solution up to third order from weakly nonlinear analysis  and numerical solution: (a) lower brunch, (b) upper  brunch. The values of other parameters are  $a=7$, $b=5.65$, $f=0.98$, $e=0.95$, $d=80,$ $c=27$ and $k_T=0.283128$. }
\label{wnanumerical}
\end{figure}
We have plotted $u_\infty(0)$ 
against the prey-taxis   parameter  $c$ in Fig. \ref{wnabif}. For $c>c_T,$ the system \eqref{pde}  has  two stable branches of pattern solution and the unstable homogeneous stationary solution (HSS). We plot both the stable  branches for $c=27$ and compared them with the numerical solutions in Fig. \ref{wnanumerical}. The amplitudes of both solutions are in good agreement. For  $21.1348<c<c_T,$ we have two unstable branches of pattern solution, two stable branches of pattern solution and the stable HSS. When  $21.1348<c<c_T,$ numerical computation under random perturbation about the homogeneous steady state leads to the HSS, but a backward continuation of pattern solution from $c>c_T$  leads to the stable pattern solution. Thus, the system \eqref{pde} shows hysteresis cycle \cite{dey2022analytical,tulumello2014cross} in this parameter range.

\subsection{Effect of prey-taxis $c$ on Turing solution} 
In subsection \ref{Turingsubsection}, we have observed the appearance of the Turing pattern when the prey-taxis coefficient $c$ crosses a threshold $c_T.$ If we consider the parameter values $a=7$, $b=5.65$, $f=0.95$, $e=0.95$ and $d=100$, then the corresponding Turing threshold is $c_T=31.4793$. We simulate the spatio-temporal model with a small amplitude spatial perturbation around the coexisting homogeneous steady state $E_1^*=(0.2016,0.1402)$.  The system shows a stable homogeneous steady state for $c<c_T$ and a non-homogeneous Turing solution for $c>c_T$.  We plot the stationary Turing solutions of $u$ and $v$ for $c=35$ and $c=50$ in Fig. \ref{c3550}. Since the predator population avoid high prey density areas, the peak of the prey population corresponds to the trough of the predator population. As the value of $c$ increases from $c=35$ to $c=50$, the peak prey density also increases due to the group defense of the prey population. Consequently, the trough of the predator population decreases further (see Fig. \ref{c3550}). 
\begin{figure}[!h]
\centering
\begin{subfigure}[b]{.48\textwidth}
 \centering
\includegraphics[scale=0.55]{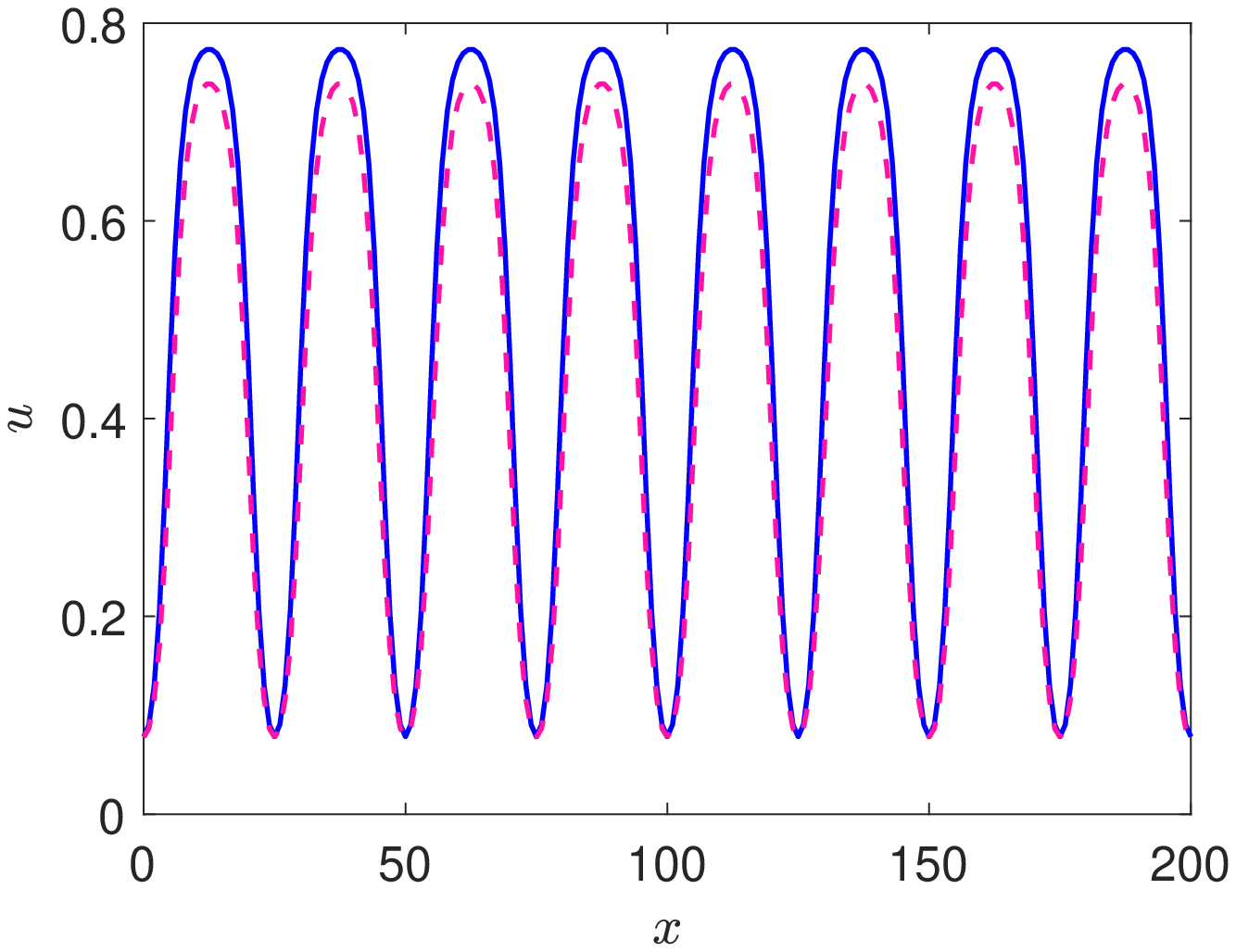}\\ 
 \caption{}
  \end{subfigure}
 \begin{subfigure}[b]{.48\textwidth}
 \centering
\includegraphics[scale=0.55]{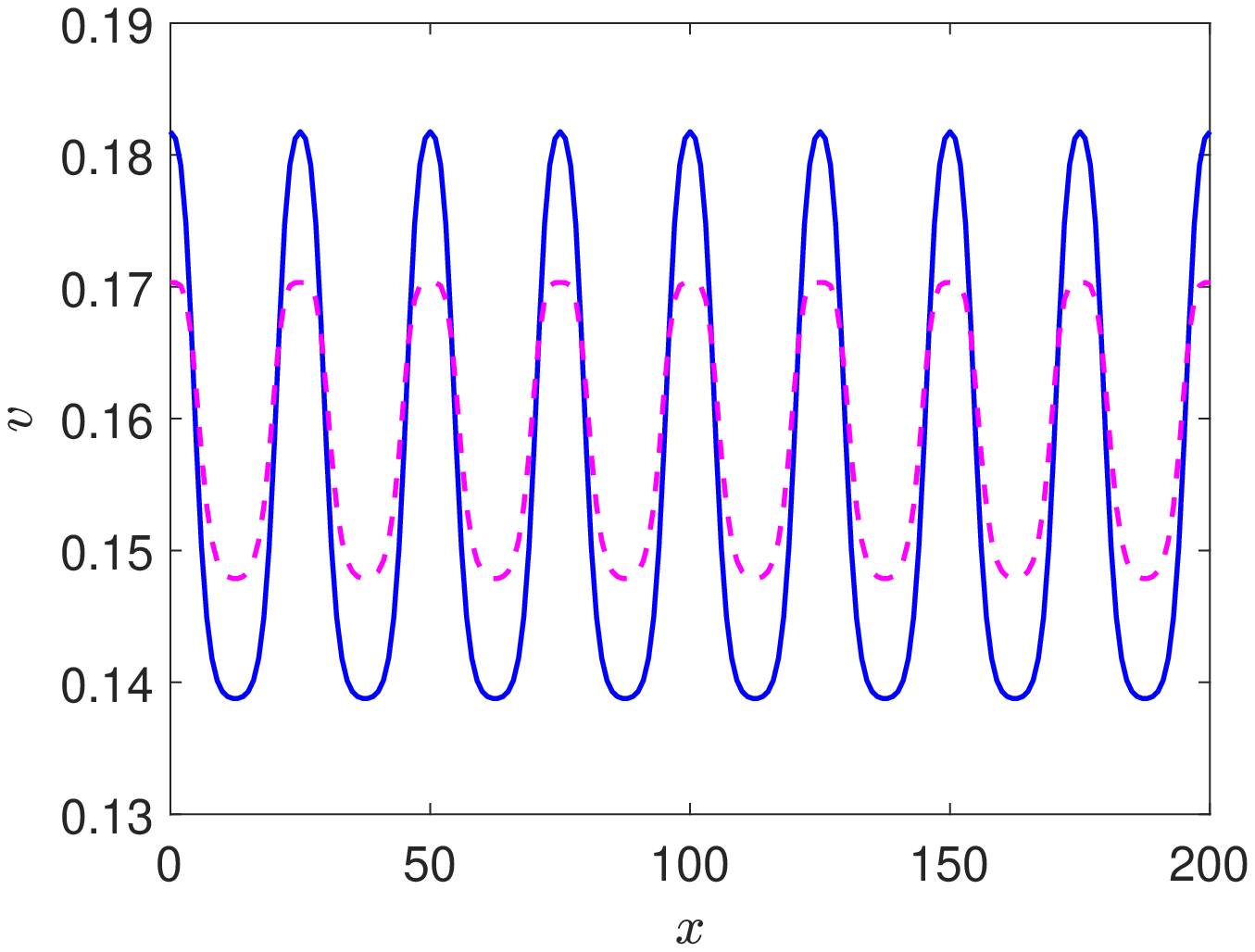}\\ 
 \caption{}
  \end{subfigure} 
\caption{Effect of prey-taxis $c$ on stationary Turing solutions: (a) prey solutions, (b) predator solutions. Here magenta dashed and blue solid curves denote the corresponding stationary Turing solutions for  $c=35$ and $c=50$, respectively. The values of other parameters are  $a=7$, $b=5.65$, $f=0.95$, $e=0.95$, and $d=100$. }
\label{c3550}
\end{figure}

Another important findings is the effects of the prey-taxis $c$ on the spatial averages of the populations. The spatial average of the prey population jumps from $0.2016$ for $c<c_T$ to $0.4516$ for $c=35$. Thus, the spatial average of the prey population becomes more than doubled due to Turing bifurcation. On the other hand, the predator population changes from $0.1402$ for $c<c_T$ to $0.1582$ for $c=35,$ a modest increase compared to prey population. Interestingly, the spatial average of the prey increases and that of predator decreases as $c$ is increased from $c=35$ to $c=50.$ Thus, an increase prey-taxis is beneficial to the prey species compared to the predator species.      

 \subsection{Spatio-temporal transient dynamics}


Here, we discuss taxis induced transient solutions for the system \eqref{pde}. We choose the parameter values $a=7$, $b=5.65$, $f=1.07$, $e=0.95$, and $d=80$. For this parameter set, we calculate $c_T=5.552$ using \eqref{ct}. Note that the corresponding temporal system has bistability between coexisting equilibrium point $E_1^*$ and axial equilibrium point $E_1$. For $c=5$, the system \eqref{pde} is neither Turing unstable nor Hopf unstable. Hence, the system reaches the homogeneous steady-state solution $E_1^*$ after the initial transients for $c=5$ when simulated with random perturbation around $E_1^*$. However, for $c>c_T$, we observe Turing instability. Usually, the system settles down to a stationary non-homogeneous state in Turing unstable domain. However, for $c=6>c^T,$ we find that the system ultimately settles down to homogeneous steady predator-free state $E_1$ (see Fig. \ref{transition}).  The system initially evolves towards the Turing solution and the intermediate Turing solution satisfies the corresponding dispersion relation $H(k^2)<0$ for this parameter value. Though the Turing solution persists for considerable amount of time, but it ultimately becomes unstable leading to the homogeneous steady predator-free state. Interestingly, for lower values of $c$ with $c<c_T$,  coexistence occurs, whereas higher value of $c$ with $c>c_T$ leads to the extinction of the predator species [see Fig. \ref{transition}(c)]. Thus, the extinction described is induced by the prey-taxis.
\begin{figure}[!t]
\begin{subfigure}[b]{.45\textwidth}
 \centering
\includegraphics[scale=0.45]{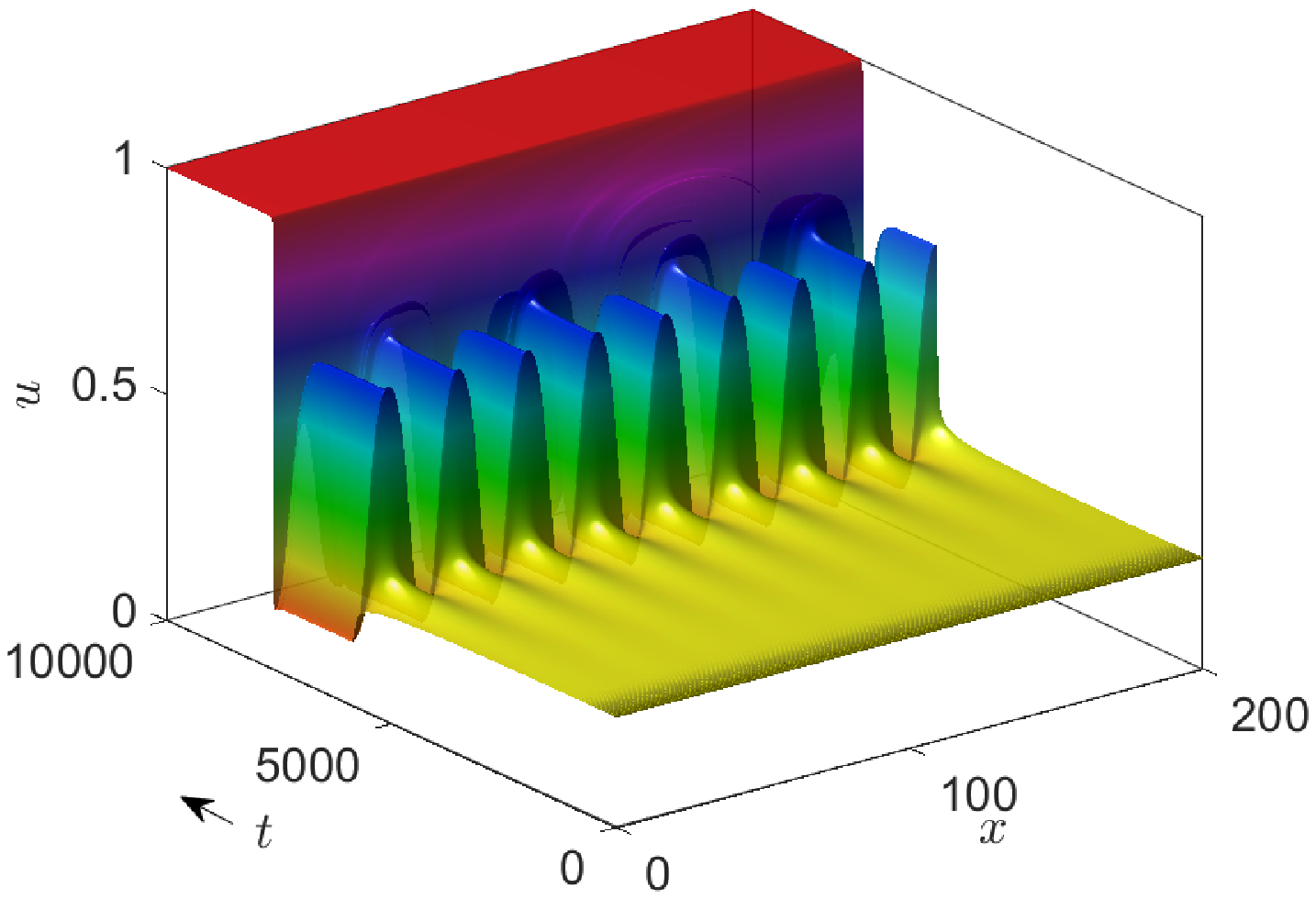}\\ 
 \caption{}
  \end{subfigure}
  \begin{subfigure}[b]{.45\textwidth}
 \centering
 \hspace*{-4cm}
\includegraphics[scale=0.45]{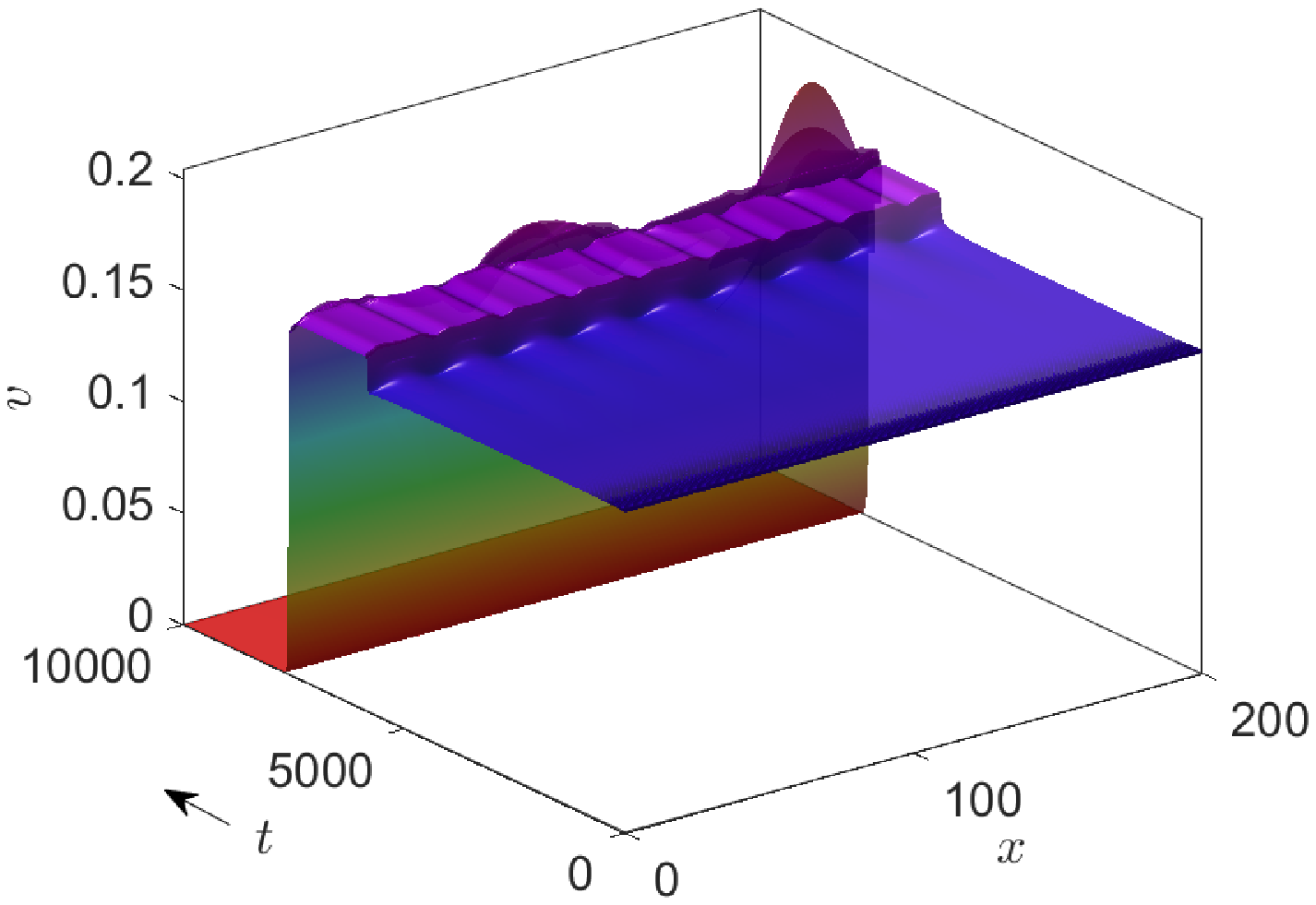}\\ 
 \caption{}
  \end{subfigure}
\begin{subfigure}[b]{.45\textwidth}
 \centering
\includegraphics[scale=0.45]{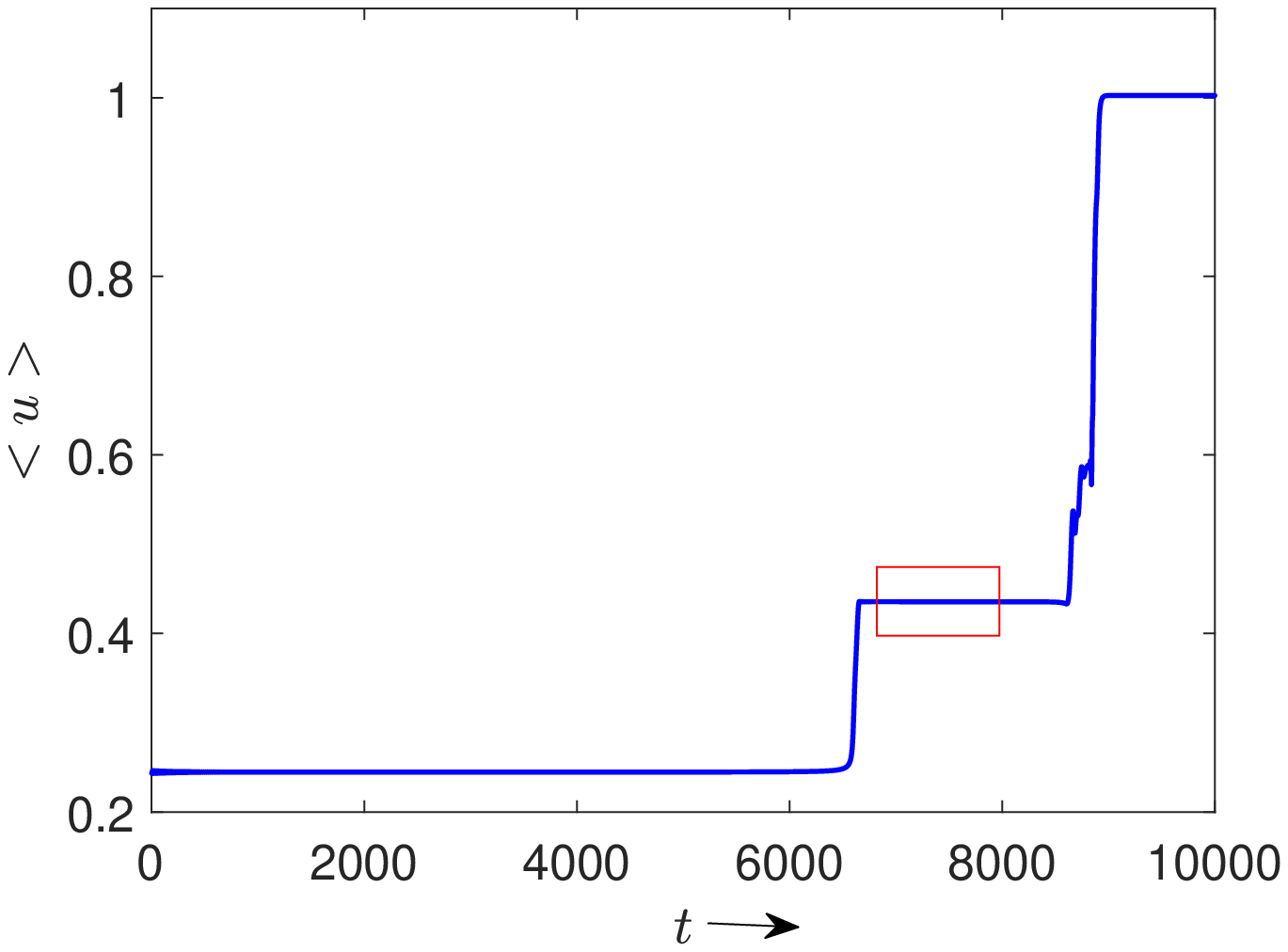}\\ 
 \caption{}
  \end{subfigure}
 \begin{subfigure}[b]{.45\textwidth}
 \centering
\includegraphics[scale=0.45]{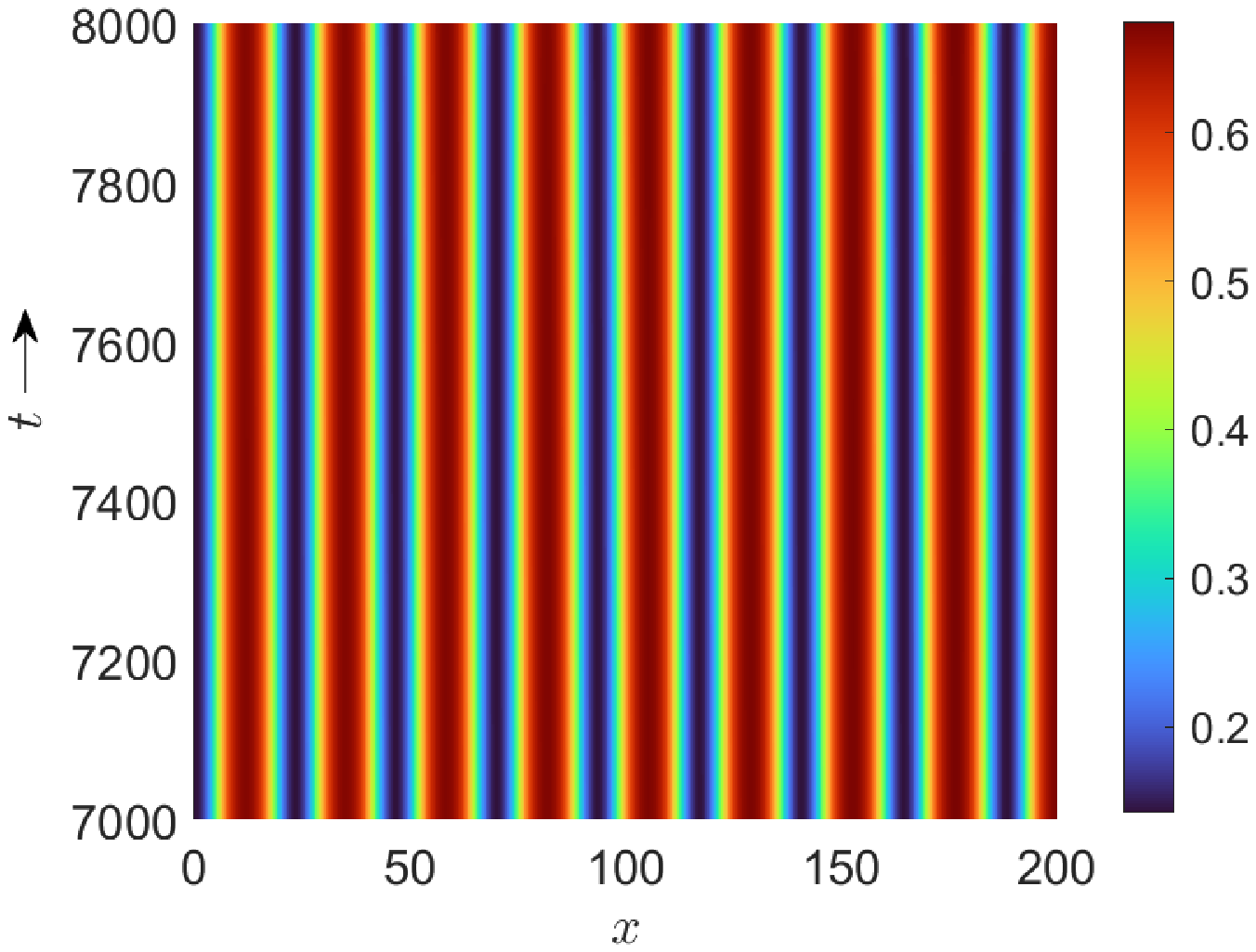}\\ 
 \caption{}
  \end{subfigure} 
 \caption{Transient dynamics leading to homogeneous steady predator-free state for $c=6$: (a) space-time plot of the prey species, (b) space-time plot of the predator species, (c) spatial average of prey species against time, (d) persistence of Turing solution during intermediate stage. Other parameter values are $a=7,\;b=5.65,\; e=0.95\;f=1.07,$ and $d=80$.} 
\label{transition}
\end{figure}

\begin{figure}[!h]
\begin{subfigure}[b]{.48\textwidth}
 \centering
\includegraphics[scale=0.6]{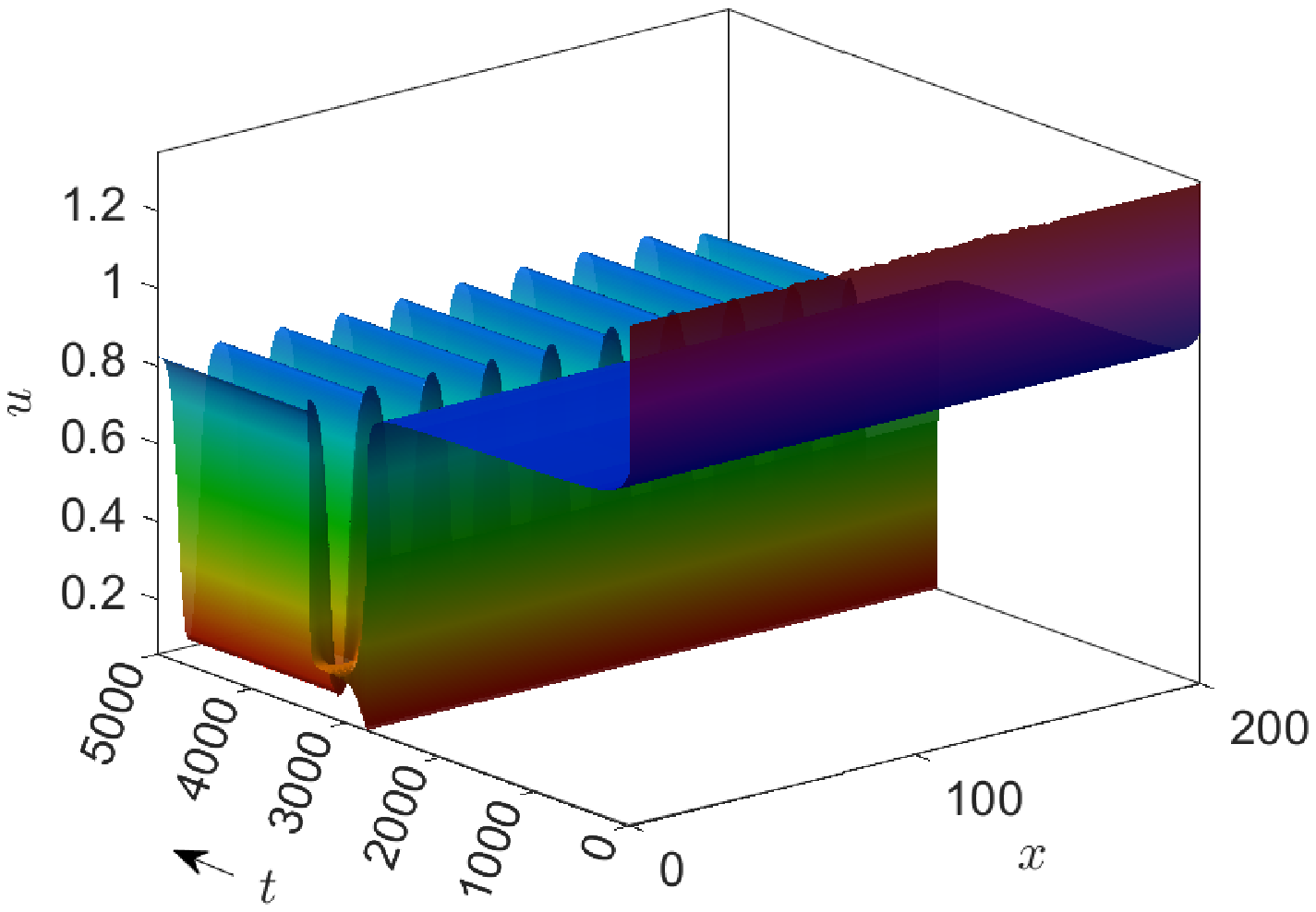}\\ 
 \caption{}
  \end{subfigure}
 \begin{subfigure}[b]{.48\textwidth}
 \centering
\includegraphics[scale=0.6]{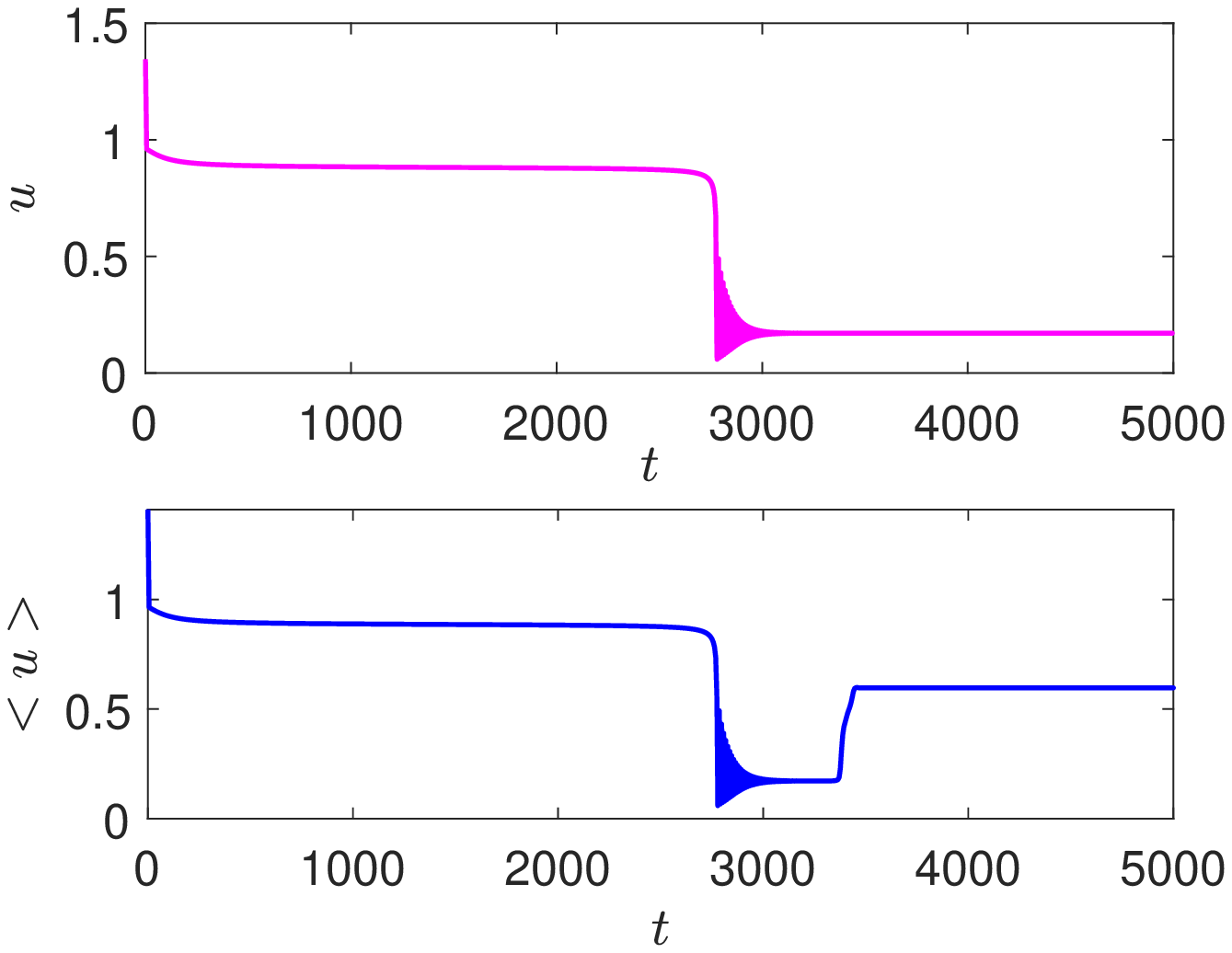}\\ 
 \caption{}
  \end{subfigure} 
 \caption{Effects of the diffusion and taxis on the long transient dynamics of the temporal model: (a) space-time plot of the prey species $u,$ (b) variation of the temporal solution $u$ of \eqref{ode} and the spatial average of the spatio-temporal model \eqref{pde} against time. The values of other parameters are  $a=7,\;b=7,\; e=0.95,\;f=0.8013,\; d=80$,  and $c=40$.}
\label{sptrans}
\end{figure}

Next, we examine the effects of the diffusion and taxis on the long transient and hysteresis states  observed in the temporal model (see Fig. \ref{steadystatetransient}). We consider the same parameter values as in Fig. \ref{steadystatetransient}(a) together with the diffusion parameter $d=80$, taxis parameter $c=40$, and a domain of length $L=200$. The initial condition is chosen as
$$u(x,0)=\begin{cases}
    1.4+0.01\xi(x), \text{ for } |x-100|<50, \\
    1.4,\qquad \qquad \;\;\,\text{ otherwise,}
\end{cases},\; \; v(x,0)=\begin{cases}
    0.05+0.01\xi(x), \text{ for } |x-100|<50, \\
    0.05, \qquad \qquad \;\;\,\text{ otherwise.}\end{cases}$$     
    Here, $\xi(x)$ is the Gaussian noise function. We choose this specific initial condition in order to compare our solution with the temporal solution. The corresponding spatio-temporal dynamics is shown in Fig. \ref{sptrans}. We plot the spatial average $<u>$ of the prey population for the system \eqref{pde} against time in Fig. \ref{sptrans}(b). We have also shown their corresponding variation $u$ of the temporal model \eqref{ode}. It shows that the transient time for both systems is almost the same, although their final destination is different. In contrast to the temporal system, the space-time solution evolves to a Turing pattern due to Turing instability. 
    
    Finally, we discuss the effects of the diffusion and taxis on the long oscillatory transient dynamics observed in the temporal model [see Fig. \ref{global_transient}(b)]. For this purpose, we consider the initial condition similar to that mentioned above together with $d=80$ and $c=40$. However, the long oscillatory transient dynamics does not appear for the spatio-temporal solution. Figure \ref{sptransientdyn} shows that the predator becomes extinct in a short time and the solution approaches homogeneous steady predator-free state $E_1$.

\begin{figure}[!h]
\centering
\includegraphics[scale=0.6]{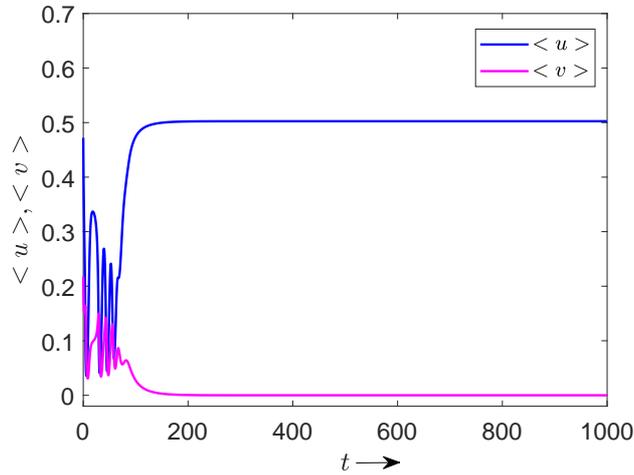}\\
\caption{Variation of the spatial averages of the prey and predator populations against time. The values of other parameters are  $a=7,\;b=7,\; e=0.95, \;f=0.867682, \; d=80$,  and $c=40$.  }
\label{sptransientdyn}
\end{figure}

\begin{figure}[!h]
\begin{subfigure}[b]{.33\textwidth}
 \centering
\includegraphics[scale=0.4]{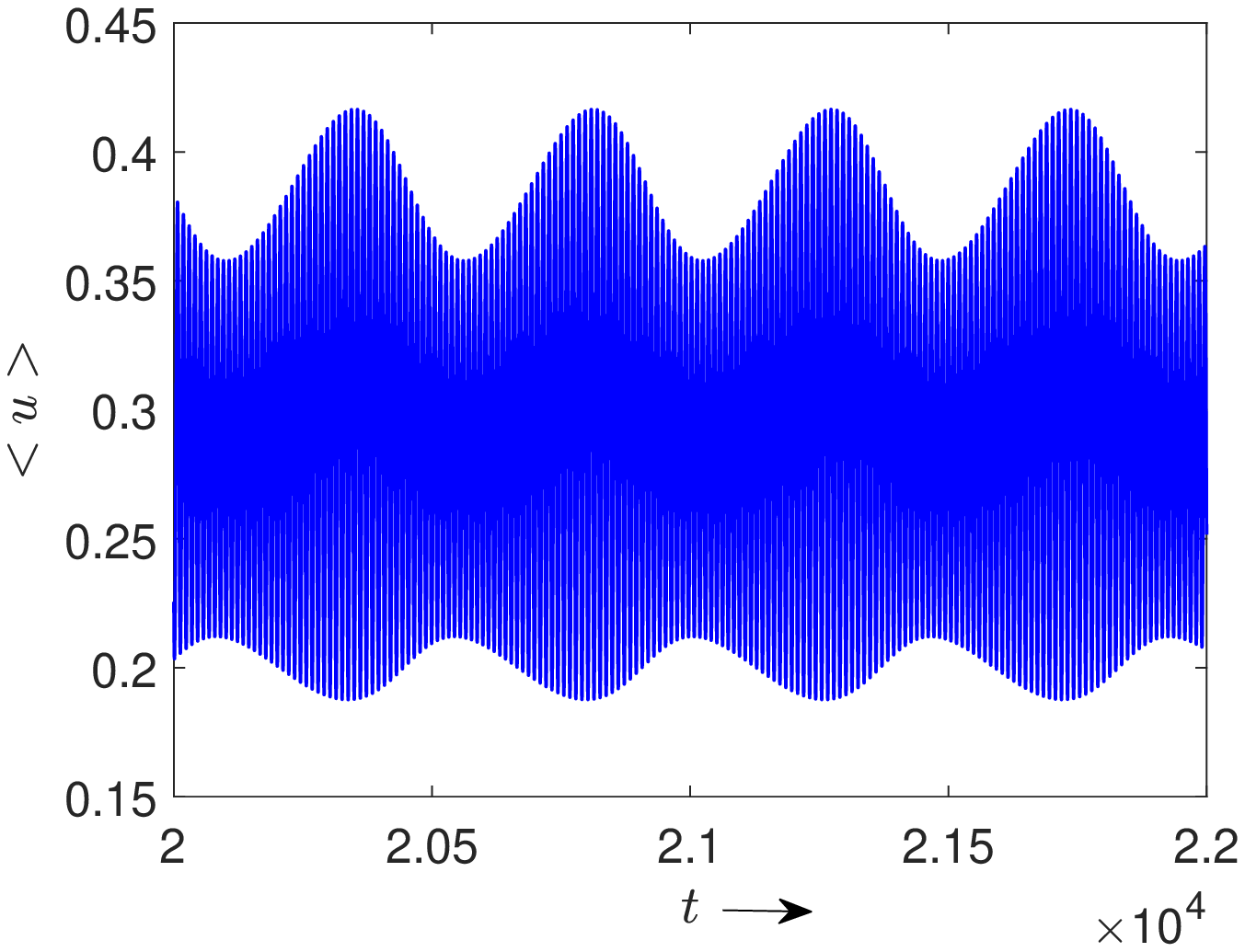}\\ 
 \caption{}
  \end{subfigure}\begin{subfigure}[b]{.33\textwidth}
 \centering
\includegraphics[scale=0.4]{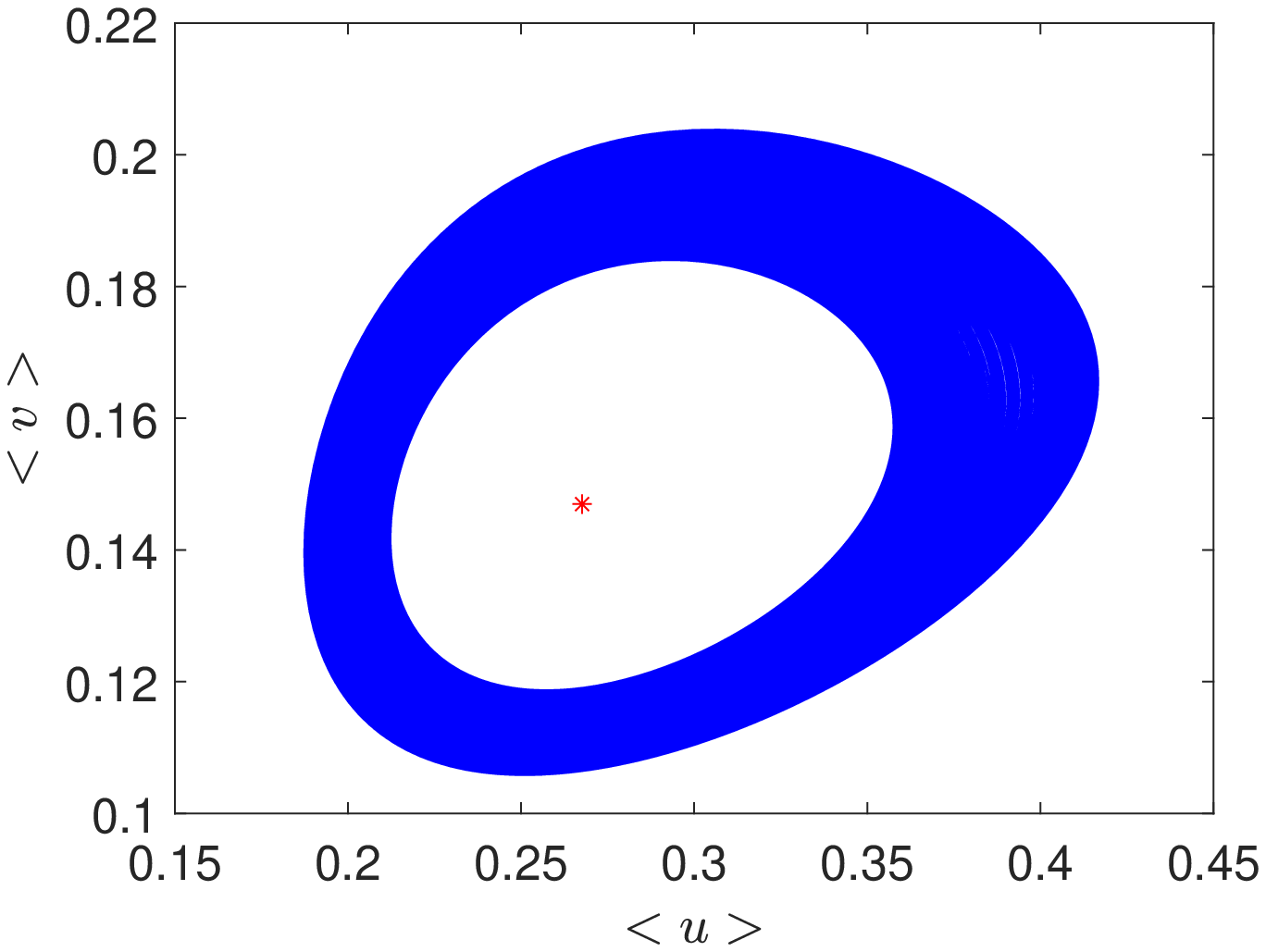}\\ 
 \caption{}
\end{subfigure}
 \begin{subfigure}[b]{0.33\textwidth}
 \centering
\includegraphics[scale=0.4]{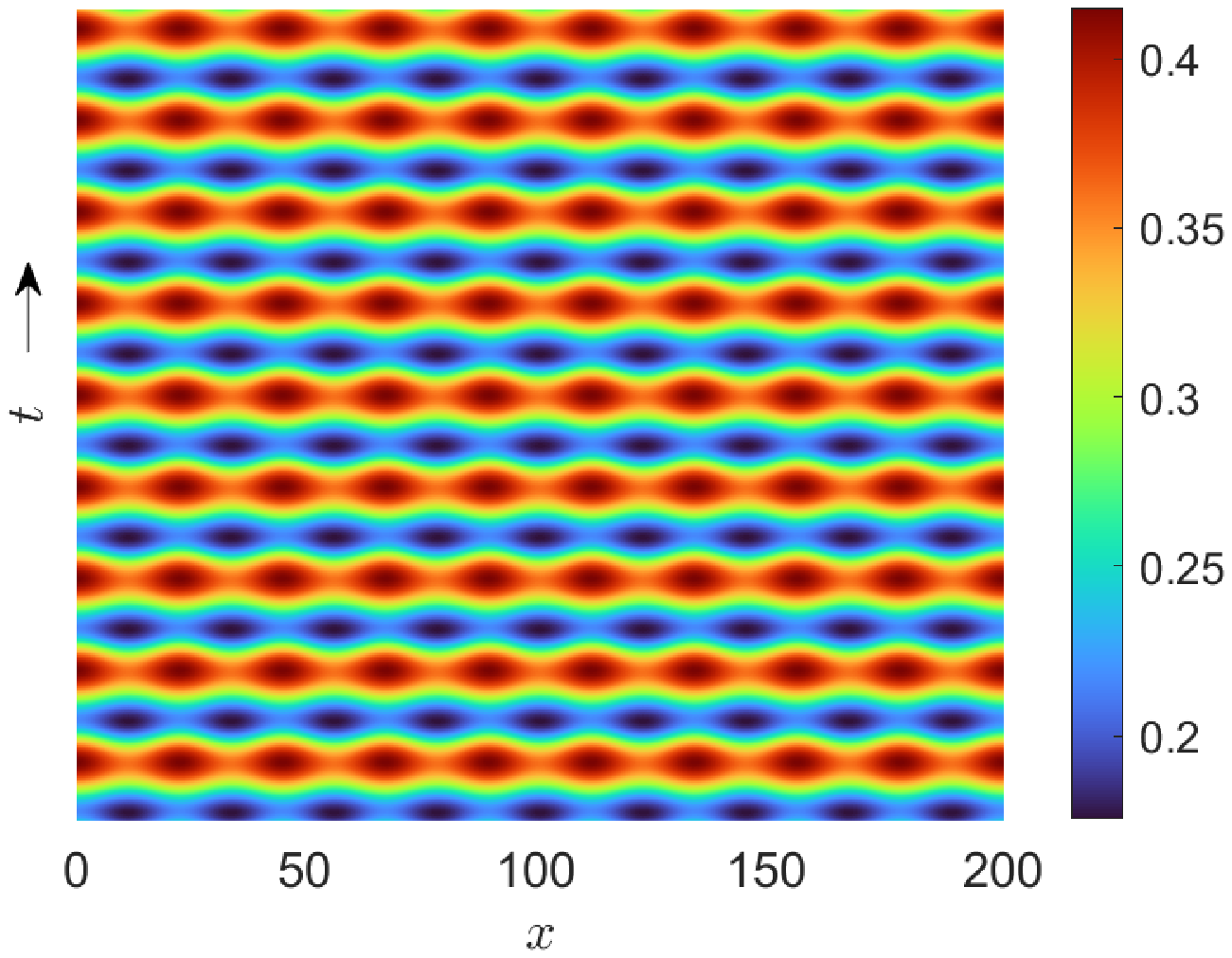}\\ 
 \caption{}
  \end{subfigure}
  
 \caption{Time oscillatory and spatially non-homogeneous solution: (a) variation of spatial average $<u>$ against time, (b) phase portrait of spatial averages $<u>$ and $<v>,$ (c) contour plot of space-time solution $u$ after the initial transient. The values of other parameters are  $a=7,\;b=5.65,\; e=0.95\;f=1.12,\; d=40$,  and $c=5$.} 
\label{osc}
\end{figure}

\subsection{Oscillatory solution}
Apart from the homogeneous and non-homogeneous steady-state solutions, the spatio-temporal system \eqref{pde} also exhibits oscillatory solutions in the Hopf region. The parameter values $a=7,\;b=5.65,\; e=0.95,\;f=1.12,\; d=40$, and $c=5$ lie in the Hopf region.  Note that the temporal system \eqref{ode} exhibits bistability between the periodic solution around $E_1^*$ and the axial equilibrium $E_1$ with the associated temporal parameters from this set. The corresponding spatio-temporal system \eqref{12} shows an oscillatory in time and  non-homogeneous in space solution shown in Fig. \ref{osc}. The phase portrait of spatial averages $<u>$ and $<v>$ [see  Fig. \ref{osc}(b)] confirms the quasi-periodic nature of the solution. 
The parameter set moves into the Turing-Hopf region with an increase in $c$, and the predator species becomes extinct leading to the homogeneous steady predator-free state. This final state corresponds to the axial equilibrium of the bistable states and it is similar to the taxis-induced extinction discussed earlier. The parameter set moves well inside the Hopf region with a decrease in $c$. The system now exhibits a homogeneous oscillatory solution  corresponding to the stable periodic solution around $E_1^*$.


\section{Conclusion} \label{discussion}
The primary goal of our work is to investigate the effects of group defense among prey species on the spatio-temporal distribution of both the prey and predator populations modelled by a Bazykin-type prey-predator model. The group defense of the prey species is incorporated using a non-monotonic functional response in the temporal model. 
Due to group defense, the predator species avoid areas with high prey density. We also include repellent prey-taxis  to  take into account this response of the predator species. The temporal model \eqref{ode} exhibits a range of complex  dynamics, including bistability, tristability, global bifurcations, and long transient dynamics. The corresponding spatio-temporal system \eqref{pde} possesses global bounded solutions. Existence of non-homogeneous stationary solution above the Turing threshold $c_T$ has been established using WNA. Further, exhaustive numerical simulations have been performed to validate the results of WNA and to investigate long transient dynamics present in the spatio-temporal systems.

We have used bifurcation analysis on the temporal system \eqref{ode} to identify various local and global bifurcations by considering $b$ and $f$ as bifurcation parameters. Note that the parameters $b$ and $f$ depend on the strength of group defense of prey species and the death rate of predator species, respectively. Through multiple one-parametric bifurcation diagrams and a two-parametric bifurcation diagram, we have observed that a higher strength of group defense or a larger death rate of predator species can lead to predator extinction in the coexisting dynamics. However, the initial population plays a crucial role in determining the final state of the system. In the case of bistability and tristability scenarios, the system can lead to co-existing steady-state dynamics or co-existing oscillatory dynamics, or a predator extinction state depending on the initial condition.

We have observed long transient dynamics in the temporal system in which the system spends a considerable time around a stationary or oscillatory state before reaching the final state. These dynamics depend on the initial conditions and  parameter values. When the group defense parameter $b$ is above the cusp point threshold value, then an increase in predator's death rate $f$ causes appearance of two new coexisting equilibria due to a saddle-node bifurcation. When $f$ is decreased slightly, these two  new coexisting equilibria disappear and a narrow region develops between two non-trivial nullclines. When the initial prey or predator population is small,  the trajectories are constrained to move through this narrow region leading to the long stationary transition dynamics [see Fig. \ref{steadystatetransient}(b)]. Also, a significant change is observed between the population levels in the transient state and the final state. A similar long transient, but oscillatory state, is observed due to saddle-node bifurcation of limit cycle, which is a global bifurcation. In contrast to the  long stationary transient discussed above, the final steady state is approximately the average of population over an oscillatory cycle in this case (see Fig. \ref{global_transient}).

A homogeneous steady-state of the spatio-temporal system \eqref{pde}, which corresponds to a stable co-existence equilibrium of the corresponding temporal model, can become Turing unstable when the prey-taxis coefficient $c$ crosses a threshold $c_T.$ A Turing solution refers to a stationary periodic in space solution. A Turing pattern develops and persists in the system when the corresponding temporal model has a single stable co-existing equilibrium point and all the other equilibria are unstable. In this case, we have employed WNA to derive Turing solution for $c$ near $c_T$ and the theoretical findings have been validated using numerical simulations (see Fig. \ref{wnanumerical}). However, if the temporal model has bistability or tristability, then the Turing solution may not persist. We have shown an example (see Fig. \ref{transition}) where the spatio-temporal system initially approaches a Turing solution but ultimately settles down to predator-extinction state. The reason for this behaviour is the presence of a stable coexisting and a stable predator-free equilibria of the corresponding temporal model. The perturbations around the homogeneous state corresponding to the coexisting equilibrium grow in magnitude and approach a Turing solution, but the system ultimately settles down to the homogeneous state corresponding to predator-free equilibrium. Similar to the temporal transients, the final state of the solution is unpredictable in this case too. Generally, an increase in group defense-induced prey-taxis can lead to predator extinction scenarios from the homogeneous stable coexistence state. 



We have also investigated whether the long transient dynamics observed in the temporal model persist in the extended  spatio-temporal system. Almost similar long stationary transient dynamics have been found in the presence of diffusion and taxis (see Fig. \ref{sptrans}) but the final state consists of a Turing pattern instead of the homogeneous steady state $E_1^*$. This is due to the taxis parameter $c$ lying in the Turing domain corresponding to the homogeneous state $E_1^*.$  But in the case of transient oscillatory dynamics of the temporal model,  the spatio-temporal system does not exhibit any transient dynamics. In this case, the system rapidly  reaches the predator extinction homogeneous steady state $E_1$ (see Fig. \ref{sptransientdyn}). Thus, the persistence of long transient dynamics of the temporal model depends on the parameter values of the spatio-temporal model. Another important observation is the appearance of non-homogeneous oscillatory pattern solutions for certain parameter sets lying in the Hopf region (see Fig. \ref{osc}). Increasing the value of taxis parameter $c$ towards the Turing-Hopf region leads to the appearance of a transient Turing pattern but the system ultimately  settles down to the homogeneous steady state $E_1.$ On the other hand, the  non-homogeneous oscillatory pattern solution becomes a homogeneous oscillatory pattern solution with decreasing $c$ towards the interior of the Hopf region. 

In summary, the prey-taxis is beneficial to the prey species compared to the predator species. It significantly influences the survival of the predator species. A stable coexisting state of the temporal model can become homogeneous predator-free state in the presence of prey-taxis. Thus, the prey-taxis plays a crucial role in the pattern formation scenario of a spatio-temporal prey-predator system.  \\[1em]
\noindent\textbf{Data Availability} The authors declare that no experimental data were used in the preparation of this manuscript. 
\\[1em]
\large{\textbf{Ethics declarations}}\\
\textbf{Conflict of interest} The authors declare that they have no conflict of interest.\\
\section*{Appendix A}
Using Gagliardo–Nirenberg inequality \cite{friedman1969partial} 
with $v\geq 0,$ we obtain 
\begin{equation*}
    \begin{split}
        \int\limits_\Omega v^q dx&= \int\limits_\Omega (v^{\frac{q}{2}})^2 dx\leq a_1\left(||\nabla v^\frac{q}{2}||_2^{\frac{2nq-2n}{nq-n+2}}. ||v^\frac{q}{2}||_\frac{2}{q}^\frac{4}{nq-n+2}+||v^\frac{q}{2}||_\frac{2}{q}^2\right) \\
        &= a_1\left(||\nabla v^\frac{q}{2}||_2^{2\theta}.||v^\frac{q}{2}||_\frac{2}{q}^{2(1-\theta)}+||v^\frac{q}{2}||_\frac{2}{q}^2\right),
    \end{split}
\end{equation*}
where $a_1>0$ and $0<\theta=\frac{nq-n}{nq-n+2}<1.$ Using Young's inequality, we have $$||\nabla v^\frac{q}{2}||_2^{2\theta}.||v^\frac{q}{2}||_\frac{2}{q}^{2(1-\theta)}\leq a_2 \theta ||\nabla v^\frac{q}{2}||_2^2+a_2^{\frac{\theta}{\theta-1}}(1-\theta)||v^\frac{q}{2}||_\frac{2}{q}^2,$$ for any  $a_2>0.$
Thus we have \begin{equation*}
    \begin{split}
        \int\limits_\Omega v^q dx&\leq a_1\left(a_2 \theta ||\nabla v^\frac{q}{2}||_2^2+a_2^{\frac{\theta}{\theta-1}}(1-\theta)||v^\frac{q}{2}||_\frac{2}{q}^2+||v^\frac{q}{2}||_\frac{2}{q}^2\right)\\
        &=a_1a_2 \theta||\nabla v^\frac{q}{2}||_2^2 +a_1\left(a_2^{\frac{\theta}{\theta-1}}(1-\theta)+1\right)||v^\frac{q}{2}||_\frac{2}{q}^2\\
        &= \varepsilon_1 ||\nabla v^\frac{q}{2}||_2^2 + \varepsilon_2 ||v||_1^q,
    \end{split}
\end{equation*}
for any $\varepsilon_1>0$ and for some $\varepsilon_2>0$ depending on the $\varepsilon_1.$  Since $||v||_1\leq B,$ we obtain $$\di{\int\limits_\Omega v^q dx\leq\varepsilon_1||\nabla v^{\frac{q}{2}}||_2^2+M} \text{ for any }\varepsilon_1>0 \text{ and for some } M=\varepsilon_2 B^q>0.$$

\newpage

\end{document}